\documentclass{amsart}

\usepackage{amssymb}
\usepackage{amsmath}
\usepackage{amsthm}
\usepackage{amsfonts}

\usepackage{a4wide}
\usepackage{longtable}
\usepackage{multirow}

\newtheorem{lemma}{Lemma}[section]

\newtheorem{theorem}{Theorem}[section]
\newtheorem{corollary}{Corollary}[theorem]

\theoremstyle{definition}
\newtheorem{definition}{Definition}[section]

\theoremstyle{remark}
\newtheorem{remark}{Remark}[section]
\newtheorem{example}{Example}[section]
\newtheorem{conjecture}{Conjecture}[section]

\begin{document}

\title[Distribution of second \(p\)-class groups]
{The distribution of second \(p\)-class groups\\on coclass graphs}

\author{Daniel C. Mayer}
\address{Naglergasse 53\\8010 Graz\\Austria}
\email{algebraic.number.theory@algebra.at}
\urladdr{http://www.algebra.at}
\thanks{Research supported by the
Austrian Science Fund,
Grant Nr. J0497-PHY}

\subjclass[2000]{Primary 11R29, 11R37, 11R11, 11R16, 11R20; Secondary 20D15}
\keywords{\(p\)-class groups, \(p\)-class field tower, principalization of \(p\)-classes,
quadratic fields, cubic fields, quartic fields, dihedral fields,
metabelian \(p\)-groups, coclass graphs}

\date{December 25, 2012}

\dedicatory{Dedicated to the memory of Emil Artin}


\begin{abstract}
General concepts and strategies are developed
for identifying the isomorphism type of the
second \(p\)-class group \(G=\mathrm{Gal}(\mathrm{F}_p^2(K)\vert K)\),
that is
the Galois group of the second Hilbert \(p\)-class field \(\mathrm{F}_p^2(K)\),
of a number field \(K\),
for a prime \(p\).
The isomorphism type determines the position of \(G\)
on one of the coclass graphs \(\mathcal{G}(p,r)\), \(r\ge 0\),
in the sense of Eick, Leedham-Green, and Newman.
It is shown that,
for special types of the base field \(K\)
and of its \(p\)-class group \(\mathrm{Cl}_p(K)\),
the position of \(G\) is restricted to certain
admissible branches of coclass trees
by selection rules.
Deeper insight, in particular,
the density of population
of individual vertices on coclass graphs,
is gained by computing
the actual distribution of second \(p\)-class groups \(G\)
for various series of number fields \(K\)
having \(p\)-class groups \(\mathrm{Cl}_p(K)\) of fixed type
and \(p\in\lbrace 2,3,5,7\rbrace\).
\end{abstract}

\maketitle



\section{Introduction}
\label{s:Intro}

Let \(p\) denote a prime
and let \(K\) be an algebraic number field.
By the Hilbert \(p\)-class field \(\mathrm{F}_p^1(K)\) of \(K\)
we understand the maximal abelian unramified \(p\)-extension of \(K\).
The Hilbert \(p\)-class field tower, briefly \(p\)-\textit{tower}, of \(K\),
\(\mathrm{F}_p^0(K)\le\mathrm{F}_p^1(K)\le\mathrm{F}_p^2(K)\le\ldots\),
is defined recursively by
\(\mathrm{F}_p^0(K)=K\)
and
\(\mathrm{F}_p^n(K)=\mathrm{F}_p^1\left(\mathrm{F}_p^{n-1}(K)\right)\),
for \(n\ge 1\).
According to the uniqueness theorem of class field theory,
all members of the \(p\)-tower are Galois extensions of \(K\),
and their union
\(\mathrm{F}_p^\infty(K)=\cup_{n=0}^\infty\mathrm{F}_p^n(K)\)
is the maximal unramified pro-\(p\) extension of \(K\).
Let \(\mathrm{Cl}_p(K)\) be the \(p\)-class group of \(K\),
that is the Sylow \(p\)-subgroup of the class group \(\mathrm{Cl}(K)\).
If \(\mathrm{Cl}_p(K)=1\) is trivial,
then \(\mathrm{F}_p^1(K)=K\)
and the \(p\)-tower of \(K\) has \textit{length} \(\ell_p(K)=0\).
If \(\mathrm{Cl}_p(K)\ne 1\)
and \(\mathrm{F}_p^{n-1}(K)<\mathrm{F}_p^n(K)=\mathrm{F}_p^{n+1}(K)\),
for some \(n\ge 1\),
the \(p\)-tower is finite of length \(\ell_p(K)=n\).
Otherwise the \(p\)-tower of \(K\) is infinite
and \(\mathrm{F}_p^\infty(K)\) is an infinite Galois extension of \(K\),
having a pro-\(p\) group
\(\mathrm{G}_p^\infty(K)=\mathrm{Gal}(\mathrm{F}_p^\infty(K)\vert K)\)
as Galois group,
endowed with the Krull topology.
For each \(n\ge 1\),
the finite quotient
\(\mathrm{G}_p^n(K)=\mathrm{Gal}(\mathrm{F}_p^n(K)\vert K)
=\mathrm{G}_p^\infty(K)/\left(\mathrm{G}_p^\infty(K)\right)^{(n)}\)
of the \(p\)-\textit{tower group} \(\mathrm{G}_p^\infty(K)\) by the closed subgroup
\(\left(\mathrm{G}_p^\infty(K)\right)^{(n)}=\mathrm{Gal}(\mathrm{F}_p^\infty(K)\vert\mathrm{F}_p^n(K))\)
is called the \(n\)th \(p\)-\textit{class group} of \(K\),
in analogy to 
\(\mathrm{G}_p^1(K)=\mathrm{Gal}(\mathrm{F}_p^1(K)\vert K)\),
which is isomorphic to the (first) \(p\)-class group \(\mathrm{Cl}_p(K)\) of \(K\),
by Artin's reciprocity law.
All these higher \(p\)-class groups \(\mathrm{G}_p^n(K)\) of \(K\) with \(n\ge 2\)
are usually non-abelian and
share two essential common invariants, as the following theorem shows.
The germs of these general concepts are contained in Artin's famous papers
\cite{Ar1,Ar2}.



\begin{theorem}
\label{thm:TTTandTKT}

Suppose that \(n\ge 2\) and \(G=\mathrm{G}_p^n(K)\).
For any subgroup \(H\le G\) which contains the commutator subgroup \(G^\prime\) of \(G\),
there exists a unique intermediate field \(K\le L\le\mathrm{F}_p^1(K)\)
such that \(H=\mathrm{Gal}(\mathrm{F}_p^n(K)\vert L)\),
and the following statements hold.

\begin{enumerate}
\item
The abelianization \(H/H^\prime\) is isomorphic to the \(p\)-class group \(\mathrm{Cl}_p(L)\).
In particular,\\
\(G/G^\prime\simeq\mathrm{Cl}_p(K)\)
and \(G^\prime/G^{\prime\prime}\simeq\mathrm{Cl}_p(\mathrm{F}_p^1(K))\).
\item
The kernel \(\ker(T_{G,H})\) of the transfer \(T_{G,H}:G/G^\prime\to H/H^\prime\) is isomorphic to
the \(p\)-principaliza\-tion kernel of \(K\) in \(L\),
that is,
the kernel \(\ker(j_{L\vert K})\) of the natural class extension homomorphism
\(j_{L\vert K}:\mathrm{Cl}_p(K)\to\mathrm{Cl}_p(L)\).
In particular,\\
\(\ker(T_{G,G})\simeq\ker(j_{K\vert K})=1\) 
and \(\ker(T_{G,G^\prime})=G/G^\prime\simeq\ker(j_{\mathrm{F}_p^1(K)\vert K})=\mathrm{Cl}_p(K)\).
\end{enumerate}

\end{theorem}



\begin{proof}
Since \(H\) contains \(G^\prime\), \(H\) is a normal subgroup of
\(G=\mathrm{Gal}(\mathrm{F}_p^n(K)\vert K)\).
The intermediate field \(K\le L\le\mathrm{F}_p^1(K)\)
of degree \(\lbrack L:K\rbrack=(G:H)\),
such that \(H=\mathrm{Gal}(\mathrm{F}_p^n(K)\vert L)\),
is determined uniquely as the fixed field \(\mathrm{Fix}(H)\) within \(\mathrm{F}_p^n(K)\),
by the Galois correspondence.
From the viewpoint of class field theory,
the norm class group \(\mathrm{N}_{L\vert K}(\mathrm{Cl}_p(L))\) of \(L\vert K\) is isomorphic to
\(H/G^\prime=\mathrm{Gal}(\mathrm{F}_p^n(K)\vert L)/\mathrm{Gal}(\mathrm{F}_p^n(K)\vert\mathrm{F}_p^1(K))
\simeq\mathrm{Gal}(\mathrm{F}_p^1(K)\vert L)\)
and thus of index \(\lbrack L:K\rbrack\) in \(\mathrm{Cl}_p(K)\simeq\mathrm{Gal}(\mathrm{F}_p^1(K)\vert K)\),
as required.

\begin{enumerate}
\item
We have
\(H/H^\prime=\mathrm{Gal}(\mathrm{F}_p^n(K)\vert L)/\mathrm{Gal}(\mathrm{F}_p^n(K)\vert \mathrm{F}_p^1(L))
\simeq\mathrm{Gal}(\mathrm{F}_p^1(L)\vert L)\simeq\mathrm{Cl}_p(L)\),
by the Galois correspondence and Artin's reciprocity law
\cite{Ar1}.
\item
The isomorphism \(\ker(T_{G,H})\simeq\ker(j_{L\vert K})\)
is a consequence of the commutativity of the diagram
in Table
\ref{tbl:ExtensionAndTransfer},
which was proved by Artin
\cite{Ar2}
and investigated in more detail by Miyake
\cite{My}.
The special case \(\ker(T_{G,G^\prime})=G/G^\prime\)
is the principal ideal theorem
\cite{Fw}.
\end{enumerate}

\end{proof}



\renewcommand{\arraystretch}{1.2}

\begin{table}[ht]
\caption{Class extension homomorphism \(j_{L\vert K}\) and transfer \(T_{G,H}\)}
\label{tbl:ExtensionAndTransfer}
\begin{center}
\begin{tabular}{|ccccc|}
\hline
                   &                      & \(j_{L\vert K}\)                    &                        &                   \\
                   & \(\mathrm{Cl}_p(K)\) & \(\longrightarrow\)                 & \(\mathrm{Cl}_p(L)\)   &                   \\
 Artin isomorphism & \(\updownarrow\)     & \(///\)                             & \(\updownarrow\)       & Artin isomorphism \\
                   & \(G/G^\prime\)       & \(\longrightarrow\)                 & \(H/H^\prime\)         &                   \\
                   &                      & \(T_{G,H}\)                         &                        &                   \\
\hline
\end{tabular}
\end{center}
\end{table}



\noindent
Since each finite quotient \(G=\mathrm{G}_p^n(K)\)
of the \(p\)-tower group
\(\mathrm{G}_p^\infty(K)\) of \(K\),
which is isomorphic to the inverse limit
\(\lim\limits_{\longleftarrow}{}_{n\ge 1}\,\mathrm{G}_p^n(K)\),
behaves in the same manner with respect to
the kernels \(\ker(T_{G,H})\) and targets \(H/H^\prime\) of the transfers
\(T_{G,H}:G/G^\prime\to H/H^\prime\), for \(G^\prime\le H\le G\),
we define two invariants \(\tau(K)=\tau(G)\) and \(\varkappa(K)=\varkappa(G)\)
either of the entire \(p\)-tower of \(K\)
or of the individual \(n\)th \(p\)-class group \(G\).



\begin{definition}

Let \(p\) be a prime and \(K\) be a number field.

\begin{enumerate}
\item
The family \(\tau(K)=(\mathrm{Cl}_p(L))_{K\le L\le\mathrm{F}_p^1(K)}\)
of \(p\)-class groups
of all intermediate fields \(L\) between \(K\) and \(\mathrm{F}_p^1(K)\)
is called \textit{transfer target type}, briefly TTT, of the \(p\)-tower of \(K\).
\item
The family \(\varkappa(K)=(\ker(j_{L\vert K}))_{K\le L\le\mathrm{F}_p^1(K)}\)
of \(p\)-principalization kernels of \(K\)
in all intermediate fields \(L\) between \(K\) and \(\mathrm{F}_p^1(K)\)
is called \textit{transfer kernel type}, briefly TKT, of the \(p\)-tower of \(K\).
\end{enumerate}

\end{definition}



\noindent
In general, third and higher \(p\)-class groups \(\mathrm{G}_p^n(K)\) of \(K\), \(n\ge 3\),
are non-metabelian with rather complex structure.
For this reason, we focus our investigation on the \textit{second} \(p\)-class group
\(G=\mathrm{G}_p^2(K)=\mathrm{Gal}(\mathrm{F}_p^2(K)\vert K)\)
which is \textit{metabelian} with commutator subgroup
\[G^\prime=\mathrm{Gal}(\mathrm{F}_p^2(K)\vert\mathrm{F}_p^1(K))\simeq\mathrm{Cl}_p(\mathrm{F}_p^1(K)).\]
It is the simplest group admitting the calculation of
the TTT, \(\tau(K)=(H/H^\prime)_{G^\prime\le H\le G}\),
and the TKT, \(\varkappa(K)=(\ker(T_{G,H}))_{G^\prime\le H\le G}\), of \(K\)
by means of the transfers from \(G\) to the subgroups \(H\le G\) containing \(G^\prime\).



\subsection{Identifying \(\mathrm{G}_p^2(K)\) via \(\varkappa(K)\) and \(\tau(K)\)}
\label{ss:IdSndPeClsGrp}

First, we illustrate that
second \(p\)-class groups \(\mathrm{G}_p^2(K)\) of number fields \(K\)
can frequently but not always be identified uniquely by means of TKT and TTT.



\noindent
For \(p=5\), we apply our recently calculated TKTs of \S\
\ref{ss:StemPhi6}
to prove the following Theorem.
It gives criteria
for second \(5\)-class groups of number fields
in terms of TKTs which were unknown up to now.



\begin{theorem}
\label{thm:Snd5ClsGrp}
Let \(K\) be an arbitrary number field
with \(5\)-class group \(\mathrm{Cl}_5(K)\) of type \((5,5)\).
In the following four cases,
the second \(5\)-class group \(\mathrm{G}_5^2(K)\) of \(K\)
is determined uniquely by
the TKT and TTT of \(K\).

\begin{enumerate}
\item
\(\varkappa(K)=(1,2,3,4,5,6)\) (identity),
\(\tau(K)=\left((5,5,5)^6\right)\)
\(\Longrightarrow\)\\
\(\mathrm{G}_5^2(K)\simeq\langle 3125,14\rangle\).
\item
\(\varkappa(K)=(1,2,5,3,6,4)\) (\(4\)-cycle),
\(\tau(K)=\left((5,25)^4,(5,5,5)^2\right)\)
\(\Longrightarrow\)\\
\(\mathrm{G}_5^2(K)\simeq\langle 3125,11\rangle\).
\item
\(\varkappa(K)=(5,1,2,6,4,3)\) (\(6\)-cycle),
\(\tau(K)=\left((5,25)^6\right)\)
\(\Longrightarrow\)\\
\(\mathrm{G}_5^2(K)\simeq\langle 3125,12\rangle\).
\item
\(\varkappa(K)=(3,1,2,5,6,4)\) (two \(3\)-cycles),
\(\tau(K)=\left((5,25)^6\right)\)
\(\Longrightarrow\)\\
\(\mathrm{G}_5^2(K)\simeq\langle 3125,9\rangle\).
\end{enumerate}

\noindent
In one case, there are two possibilities for \(\mathrm{G}_5^2(K)\).\\
\(\varkappa(K)=(6,1,2,4,3,5)\) (\(5\)-cycle),
\(\tau(K)=\left((5,25)^5,(5,5,5)\right)\)
\(\Longrightarrow\)\\
either \(\mathrm{G}_5^2(K)\simeq\langle 3125,8\rangle\)
or \(\mathrm{G}_5^2(K)\simeq\langle 3125,13\rangle\).

\noindent
The powers in TTTs denote iteration and
the \(5\)-groups are identified by their numbers in the SmallGroups library
\cite{BEO}.
\end{theorem}


\begin{proof}
In section \S\
\ref{ss:StemPhi6},
we prove that the metabelianization
\(\mathrm{G}_5^2(K)=\mathrm{Gal}(\mathrm{F}_5^2(K)\vert K)=G/G^{\prime\prime}\)
of the \(5\)-tower group \(G=\mathrm{G}_5^\infty(K)\)
of any algebraic number field \(K\) with \(5\)-class group \(\mathrm{Cl}_5(K)\) of type \((5,5)\),
having one of the five pairs of TKT and TTT in Theorem
\ref{thm:Snd5ClsGrp},
is one of the six terminal top vertices
\(\langle 3125,8\ldots 9\rangle\) and \(\langle 3125,11\ldots 14\rangle\)
of the coclass graph \(\mathcal{G}(5,2)\) in Figure
\ref{fig:Typ55Cocl2}.
\end{proof}



\begin{remark}
We conjecture that the TKT alone
suffices for the characterization of
\(\mathrm{G}_5^2(K)\)
in Theorem
\ref{thm:Snd5ClsGrp}.
\end{remark}



\begin{example}
\label{exm:Snd5ClsGrp}
Discriminants \(D\) with smallest absolute values
of complex quadratic fields \(K=\mathbb{Q}(\sqrt{D})\)
having one of the five pairs of TKT and TTT in Theorem
\ref{thm:Snd5ClsGrp}
are given by
\(-89751\), \(-37363\), \(-11199\), \(-17944\), \(-12451\),
in the same order.
They were computed by means of MAGMA
\cite{MAGMA}.
According to section \S\
\ref{sss:StatScnd5ClgpCocl2},
the vertices \(\langle 3125,8\rangle\) and \(\langle 3125,13\rangle\)
of coclass graph \(\mathcal{G}(5,2)\) in Figure
\ref{fig:Typ55Cocl2},
corresponding to the last case of Theorem
\ref{thm:Snd5ClsGrp},
are populated by
\(167\) occurrences \((17.4\%)\) 
of \(959\) complex quadratic fields \(K=\mathbb{Q}(\sqrt{D})\)
with \(-2\,270\,831\le D<0\) and \(\mathrm{Cl}_5(K)\) of type \((5,5)\).
This shows that even the last case alone
occurs with rather high density.
\end{example}



\noindent
For \(p=3\), we use four TKTs which occurred repeatedly in the literature
\cite{SoTa,HeSm,BrGo}.
These TKTs define infinite sequences,
in fact periodic coclass families (\S\
\ref{s:CoclGrph}),
of possible groups \(\mathrm{G}_3^2(K)\),
and neither Heider and Schmithals
\cite{HeSm}
nor Brink and Gold
\cite{Br,BrGo}
have been aware that the TTT is able to identify a unique member of the sequences,
as we proved in
\cite{Ma1,Ma3}.

\begin{theorem}
\label{thm:SectionE}
Let \(K\) be an arbitrary number field with \(3\)-class group \(\mathrm{Cl}_3(K)\) of type \((3,3)\).
In the following four cases,
the second \(3\)-class group \(\mathrm{G}_3^2(K)\) of \(K\) is either determined uniquely
or up to the sign of the relational exponent \(\gamma\)
by the TKT and the parametrized TTT of \(K\), for each integer \(j\ge 2\).

\begin{enumerate}
\item
\(\varkappa(K)=(1,3,1,3)\), \(\tau(K)=\left((3^j,3^{j+1}),(3,9)^2,(3,3,3)\right)\)
\(\Longrightarrow\)\\
\(\mathrm{G}_3^2(K)\simeq G_0^{2j+2,2j+3}(1,-1,1,1)\).
\item
\(\varkappa(K)=(2,3,1,3)\), \(\tau(K)=\left((3^j,3^{j+1}),(3,9)^2,(3,3,3)\right)\)
\(\Longrightarrow\)\\
\(\mathrm{G}_3^2(K)\simeq G_0^{2j+2,2j+3}(0,-1,\pm 1,1)\).
\item
\(\varkappa(K)=(1,2,3,1)\), \(\tau(K)=\left((3^j,3^{j+1}),(3,9)^3\right)\)
\(\Longrightarrow\)\\
\(\mathrm{G}_3^2(K)\simeq G_0^{2j+2,2j+3}(1,0,-1,1)\).
\item
\(\varkappa(K)=(2,2,3,1)\), \(\tau(K)=\left((3^j,3^{j+1}),(3,9)^3\right)\)
\(\Longrightarrow\)\\
\(\mathrm{G}_3^2(K)\simeq G_0^{2j+2,2j+3}(0,0,\pm 1,1)\).
\end{enumerate}

\noindent
The groups in the first two cases are located
on the coclass tree \(\mathcal{T}(\langle 243,6\rangle)\) in Fig.
\ref{fig:TreeQTyp33Cocl2},
the groups in the last two cases
on the coclass tree \(\mathcal{T}(\langle 243,8\rangle)\) in Fig.
\ref{fig:TreeUTyp33Cocl2}.
\(3\)-groups of order \(3^n\) and index \(m\) of nilpotency
are identified by their parametrized presentations
given in the form \(G_\rho^{m,n}(\alpha,\beta,\gamma,\delta)\) in \S\
\ref{sss:PrmPres2}.
\end{theorem}

\begin{proof}
We proved that the second derived quotient
\(\mathrm{G}_3^2(K)=G/G^{\prime\prime}\)
of the \(3\)-tower group \(G=\mathrm{G}_3^\infty(K)\)
of any algebraic number field \(K\) with \(3\)-class group \(\mathrm{Cl}_3(K)\) of type \((3,3)\),
transfer kernel type 
E.6, \(\varkappa(K)=(1,3,1,3)\),
resp. E.14, \(\varkappa(K)=(2,3,1,3)\)
\cite[Tbl. 6, p. 492]{Ma2},
and parametrized transfer target type \(\tau(K)=\left((3^j,3^{j+1}),(3,9)^2,(3,3,3)\right)\), \(j\ge 2\),
is isomorphic to the unique group \(G_0^{2j+2,2j+3}(1,-1,1,1)\),
resp. to one of the two groups \(G_0^{2j+2,2j+3}(0,-1,\pm 1,1)\),
of the coclass tree \(\mathcal{T}(\langle 243,6\rangle)\)
in Figure
\ref{fig:TreeQTyp33Cocl2}
\cite[Thm. 4.4, Tbl. 8]{Ma3}.\\
Similarly, we proved for transfer kernel type
E.8, \(\varkappa(K)=(1,2,3,1)\),
resp. E.9, \(\varkappa(K)=(2,2,3,1)\),
and parametrized transfer target type \(\tau(K)=\left((3^j,3^{j+1}),(3,9)^3\right)\), \(j\ge 2\),
that \(\mathrm{G}_3^2(K)\)
is isomorphic to the unique group \(G_0^{2j+2,2j+3}(1,0,-1,1)\),
resp. to one of the two groups \(G_0^{2j+2,2j+3}(0,0,\pm 1,1)\),
of the coclass tree \(\mathcal{T}(\langle 243,8\rangle)\)
in Figure
\ref{fig:TreeUTyp33Cocl2}.
\end{proof}



\begin{example}
\label{exm:SectionE}
By
\cite[Thm. 5.2, p. 492]{Ma1},
the complex quadratic field \(K=\mathbb{Q}(\sqrt{-9748})\) is a number field
having the TKT and TTT of the last case in Theorem
\ref{thm:SectionE},
actually with smallest absolute discriminant.
It was first mentioned by Scholz and Taussky
\cite[p. 25]{SoTa}.
Among the \(93\) complex quadratic fields \(K=\mathbb{Q}(\sqrt{D})\)
with discriminants \(-6\cdot 10^4<D<0\) and \(\mathrm{Cl}_3(K)\) of type \((3,3)\),
there are \(11\) cases \((12\%)\) having the TKT and TTT of the last case in Theorem
\ref{thm:SectionE}.
So even the last case alone occurs quite frequently.
\end{example}



\noindent
In contrast, we can also prove that
certain metabelian \(p\)-groups are excluded as second \(p\)-class groups
for special base fields.
The following negative result for \(p=3\)
gives an exact justification for a particular instance of our
\textit{weak leaf conjecture}, Cnj.
\ref{cnj:WeakLeafCnj}.

\begin{theorem}
\label{thm:SpecWeakLeaf}
The \(3\)-group \(\langle 243,4\rangle\), resp. \(\langle 243,9\rangle\),
cannot occur as second \(3\)-class group \(\mathrm{G}_3^2(K)\)
for a complex quadratic field \(K=\mathbb{Q}(\sqrt{D})\), \(D<0\),
whose TKT and TTT are given by
\(\varkappa(K)=(4,4,4,3)\) and \(\tau(K)=\left((3,9),(3,3,3)^3\right)\),
resp.
\(\varkappa(K)=(2,1,4,3)\) and \(\tau(K)=\left((3,9)^4\right)\).
\end{theorem}

\begin{proof}
The assumption that one of the groups
\(g\in\lbrace\langle 243,4\rangle,\langle 243,9\rangle\rbrace\)
were the second \(3\)-class group \(g=\mathrm{G}_3^2(K)\) of a complex quadratic field \(K\)
implies two contradictory consequences.
On the one hand,
both groups \(g\) are of class \(\mathrm{cl}(g)=3\),
whence the fourth lower central \(\gamma_4(g)=1\) is trivial.
According to Heider and Schmithals
\cite[p. 20]{HeSm},
any number field \(K\) whose second \(p\)-class group \(g=\mathrm{G}_p^2(K)\)
has trivial \(\gamma_4(g)=1\)
possesses a \(p\)-tower of length \(\ell_p(K)=2\).
On the other hand,
since \(K\) is complex quadratic, its \(3\)-tower group \(G\) must be a Schur \(\sigma\)-group
\cite{Sh},
\cite[p. 58]{KoVe}.
The Schur multiplier \(\mathrm{H}_2(g,\mathbb{Z})\) of both groups \(g\)
is non-trivial of order \(3\),
as can be verified by means of GAP
\cite{GAP}.
Hence they cannot be Schur \(\sigma\)-groups
\cite[p. 6]{BBH}.
Therefore, the \(3\)-tower of \(K\) cannot stop at the second stage, \(G\ne\mathrm{G}_3^2(K)\), and
\(G\) must be a non-metabelian group of derived length at least \(3\),
that is, the \(3\)-tower of \(K\) has length \(\ell_3(K)\ge 3\).
\end{proof}



\subsection{Length of \(p\)-towers}
\label{ss:TowerLength}

As the following Theorems
\ref{thm:3TowerLength2}--\ref{thm:3TowerLengthAtLeast3}
show,
the length \(\ell_p(K)\) of the \(p\)-tower of \(K\)
can either be determined exactly or at least be estimated by a lower bound,
once the second \(p\)-class group \(\mathrm{G}_p^2(K)\) of \(K\) 
and its properties are known in sufficient detail.



\noindent
A criterion for \(3\)-towers of exact length \(2\)
was proved in three independent ways by Scholz and Taussky
\cite{SoTa},
by Heider and Schmithals
\cite{HeSm},
and by Brink and Gold
\cite{Br,BrGo}.
With our new methods, we can give a short proof of this criterion.



\begin{theorem}
\label{thm:3TowerLength2}
Let \(K\) be an arbitrary number field
with \(3\)-class group \(\mathrm{Cl}_3(K)\) of type \((3,3)\).
In the following two cases,
the second \(3\)-class group \(\mathrm{G}_3^2(K)\) of \(K\)
and the TTT of \(K\) are determined uniquely by
the TKT of \(K\).

\begin{enumerate}
\item
\(\varkappa(K)=(2,2,4,1)\)
\(\Longrightarrow\)
\(\mathrm{G}_3^2(K)\simeq\langle 243,5\rangle\),
\(\tau(K)=\left((3,9)^3,(3,3,3)\right)\).
\item
\(\varkappa(K)=(4,2,2,4)\)
\(\Longrightarrow\)
\(\mathrm{G}_3^2(K)\simeq\langle 243,7\rangle\),
\(\tau(K)=\left((3,9)^2,(3,3,3)^2\right)\).
\end{enumerate}

\noindent
In both cases,
the \(3\)-class field tower of \(K\) has exact length \(\ell_3(K)=2\).
\end{theorem}

\begin{proof}
The metabelianization
\(\mathrm{G}_3^2(K)\simeq G/G^{\prime\prime}\)
of the \(3\)-tower group \(G=\mathrm{G}_3^\infty(K)\)
of any algebraic number field \(K\) with \(3\)-class group \(\mathrm{Cl}_3(K)\) of type \((3,3)\)
having transfer kernel type D.10, \(\varkappa(K)=(2,2,4,1)\), resp. D.5, \(\varkappa(K)=(4,2,2,4)\)
\cite[Tbl. 6, p. 492]{Ma2},
is isomorphic to the terminal top vertex
\(\langle 243,5\rangle\), resp. \(\langle 243,7\rangle\),
of the sporadic part \(\mathcal{G}_0(3,2)\) of the coclass graph \(\mathcal{G}(3,2)\)
in Figure
\ref{fig:Typ33Cocl2},
according to Nebelung
\cite[Thm. 6.14, p. 208]{Ne1}.
In
\cite[Thm. 4.2, Tbl. 4]{Ma3}
it is shown that
the corresponding transfer target type is given by
\(\tau(K)=\left((3,9)^3,(3,3,3)\right)\), resp. \(\tau(K)=\left((3,9)^2,(3,3,3)^2\right)\).\\
According to the proof of
\cite[Thm. 4.2, p. 14]{BBH},
\(\langle 243,5\rangle\) and \(\langle 243,7\rangle\) are Schur \(\sigma\)-groups.\\
However,
independently from \(K\) being complex quadratic or not,
when the second derived quotient \(G/G^{\prime\prime}\simeq\mathrm{G}_3^2(K)\) of \(G\)
is a Schur \(\sigma\)-group,
then the \(3\)-tower group \(G\) of \(K\) must be isomorphic to it, \(G\simeq\mathrm{G}_3^2(K)\),
by the argument given in
\cite[Lem. 4.10]{BoEl}.
Consequently, the \(3\)-tower of \(K\) stops at the second stage and has length \(\ell_3(K)=2\).
\end{proof}


\begin{remark}
We point out that the figure in
\cite[p. 10]{BBH}
is not a coclass graph in our sense
(\S\ \ref{ss:CoclGrph}),
since it contains vertices of four different coclass graphs
\(\mathcal{G}(3,r)\), \(1\le r\le 4\),
partially connected by edges of depth \(2\).
The top level of this figure,
where \(\langle 243,5\rangle\) and \(\langle 243,7\rangle\)
are emphasized by surrounding circles,
coincides with the top vertices of our Figure
\ref{fig:Typ33Cocl2}.
\end{remark}



\begin{example}
\label{exm:3TowerLength2}
Discriminants \(D\) with smallest absolute values
of complex quadratic fields \(K=\mathbb{Q}(\sqrt{D})\)
having one of the two TKTs in Theorem
\ref{thm:3TowerLength2}
are given by
\(-4027\), \(-12131\),
in the same order.
The first was communicated by Scholz and Taussky
\cite[p. 22]{SoTa},
the second by Heider and Schmithals
\cite[p. 19]{HeSm}.
Corresponding minimal discriminants of real quadratic fields
are \(422573\), \(631769\)
\cite[Tbl. 4, p. 498]{Ma1}.
Among the \(2020\) complex quadratic fields \(K=\mathbb{Q}(\sqrt{D})\)
with discriminants \(-10^6<D<0\) and \(\mathrm{Cl}_3(K)\) of type \((3,3)\),
there are \(936\) cases \((46.3\%)\)
having one of the two pairs of TKT and TTT in Theorem
\ref{thm:3TowerLength2}
\cite[Tbl. 3, p. 497]{Ma1}.
So these types of fields are definitely among the high-champs
with respect to density of population.
\end{example}



\noindent
In the next Theorem,
the second \(3\)-class group \(\mathrm{G}_3^2(K)\)
is not at all determined by the TKT \(\varkappa(K)\) alone.
Furthermore,
we must restrict ourselves to an estimate of the \(3\)-tower length \(\ell_3(K)\ge 3\).

\begin{theorem}
\label{thm:3TowerLengthAtLeast3}
Let \(K=\mathbb{Q}(\sqrt{D})\), \(D<0\), be a complex quadratic field
with \(3\)-class group \(\mathrm{Cl}_3(K)\) of type \((3,3)\).
In the following two cases,
the second \(3\)-class group \(\mathrm{G}_3^2(K)\) of \(K\)
is determined uniquely by
the TKT and the TTT of \(K\).

\begin{enumerate}
\item
\(\varkappa(K)=(4,4,4,3)\),
\(\tau(K)=\left((3,9),(3,3,3)^3\right)\)
\(\Longrightarrow\)
\(\mathrm{G}_3^2(K)\simeq\langle 729,45\rangle\).
\item
\(\varkappa(K)=(2,1,4,3)\),
\(\tau(K)=\left((3,9)^4\right)\)
\(\Longrightarrow\)
\(\mathrm{G}_3^2(K)\simeq\langle 729,57\rangle\).
\end{enumerate}

\noindent
In both cases,
\(K\) has a \(3\)-class field tower of length \(\ell_3(K)\ge 3\).
\end{theorem}

\begin{proof}
We proved that the second derived quotient
\(\mathrm{G}_3^2(K)=G/G^{\prime\prime}\)
of the \(3\)-tower group \(G=\mathrm{G}_3^\infty(K)\)
of a complex quadratic field \(K\) with \(3\)-class group \(\mathrm{Cl}_3(K)\) of type \((3,3)\),
transfer kernel type H.4, \(\varkappa(K)=(4,4,4,3)\), resp. G.19, \(\varkappa(K)=(2,1,4,3)\)
\cite[Tbl. 6, p. 492]{Ma2},
and transfer target type \(\tau(K)=\left((3,9),(3,3,3)^3\right)\), resp. \(\tau(K)=\left((3,9)^4\right)\)
\cite[Thm. 4.3, Tbl. 6]{Ma3},
is isomorphic to the unique vertex
\(\langle 729,45\rangle\), resp. \(\langle 729,57\rangle\),
of the sporadic part \(\mathcal{G}_0(3,2)\)
of coclass graph \(\mathcal{G}(3,2)\)
in Figure
\ref{fig:Typ33Cocl2}.
For an arbitrary number field \(K\), several other candidates for \(\mathrm{G}_3^2(K)\) are possible.
However, for a complex quadratic field \(K\),
\(\langle 243,45\rangle\), resp. \(\langle 243,57\rangle\), are discouraged by Theorem
\ref{thm:SpecWeakLeaf},
and the siblings \(\langle 729,44\rangle\) and \(\langle 729,46\ldots 47\rangle\), resp. \(\langle 729,56\rangle\),
of \(\langle 729,45\rangle\), resp. \(\langle 729,57\rangle\),
do not admit the mandatory automorphism of order \(2\) acting as inversion on the abelianization.\\
Since \(K\) is complex quadratic, its \(3\)-tower group \(G\) must be a Schur \(\sigma\)-group
\cite{Sh},
\cite[p. 58]{KoVe}.
However,
neither \(\langle 729,45\rangle\) nor \(\langle 729,57\rangle\) is a Schur \(\sigma\)-group
\cite[p. 6]{BBH},
because the Schur multiplier is non-trivial of type \((3,3)\),
as can be verified with the aid of GAP
\cite{GAP}.
Therefore, the \(3\)-tower of \(K\) cannot stop at the second stage, \(G\ne\mathrm{G}_3^2(K)\), and
\(G\) must be a non-metabelian group of derived length at least \(3\),
that is, the \(3\)-tower has length \(\ell_3(K)\ge 3\).
\end{proof}



\begin{example}
\label{exm:3TowerLengthAtLeast3}
Discriminants \(D\) with smallest absolute values
of complex quadratic fields \(K=\mathbb{Q}(\sqrt{D})\)
having one of the two pairs of TKT and TTT in Theorem
\ref{thm:3TowerLengthAtLeast3}
are given by
\(-3896\), \(-12067\),
in the same order.
They were communicated by Heider and Schmithals
\cite[p. 19]{HeSm}.
Among the \(2020\) complex quadratic fields \(K=\mathbb{Q}(\sqrt{D})\)
with discriminants \(-10^6<D<0\) and \(\mathrm{Cl}_3(K)\) of type \((3,3)\),
there are \(391\) cases \((19.4\%)\)
having one of the two pairs of TKT and TTT in Theorem
\ref{thm:3TowerLengthAtLeast3}
\cite[Tbl. 3, p. 497]{Ma1}.
This shows that fields with \(3\)-towers of at least three stages
occur quite frequently.
\end{example}



Note that the proofs of the preceding Theorems
\ref{thm:3TowerLength2},
and
\ref{thm:3TowerLengthAtLeast3}
are very brief.
This is the beginning of powerful new methods of research concerning
the maximal unramified pro-\(p\) extensions of number fields
by joining the coclass theory of finite \(p\)-groups and
suitable generalizations of Schur \(\sigma\)-groups
\cite{BBH2}.
We optimistically expect further prolific impact
of these new foundations on the investigation of
\(p\)-towers and \(p\)-principalization,
although Artin called the capitulation problem \lq hopeless\rq.



\subsection{Overview}
\label{ss:Overview}

In \S\S\
\ref{ss:ScndClgpCocl1}
and
\ref{ss:ScndClgpTyp33CoclGe2},
we analyze number fields \(K\)
with \(p\)-class group \(\mathrm{Cl}_p(K)\) of type \((p,p)\).
Based on
\cite{Ma1},
we prove that the \(p\)-class numbers \(\mathrm{h}_p(L_i)=\lvert\mathrm{Cl}_p(L_i)\rvert\)
of \textit{two distinguished} intermediate fields \(L_i\), \(1\le i\le 2\),
lying strictly between \(K\) and \(\mathrm{F}_p^1(K)\),
and the \(p\)-class number \(\mathrm{h}_p(\mathrm{F}_p^1(K)\)
of the Hilbert \(p\)-class field of \(K\),
that is, the orders of three special members of the TTT \(\tau(K)\),
contain sufficient information for determining
the order \(\lvert G\rvert=3^n\), class \(c=\mathrm{cl}(G)\), coclass \(r=\mathrm{cc}(G)\),
and the so-called \textit{defect of commutativity} \(k=k(G)\)
of the second \(p\)-class group \(G=\mathrm{G}_p^2(K)\) of \(K\).
These invariants are related by the equation \(n=\mathrm{cl}(G)+\mathrm{cc}(G)\)
and restrict \(G\) to the \textit{finite} subset
of groups of equal order \(3^n\) of the coclass graph \(\mathcal{G}(p,r)\).
If the TKT \(\varkappa(K)\) is known additionally,
the position of \(G\) can be restricted further,
either to a branch \(\mathcal{B}\) of a coclass tree \(\mathcal{T}\),
forming a subgraph of \(\mathcal{G}(p,r)\),
or even to a unique isomorphism type of metabelian \(p\)-groups.

Group theoretic foundations concerning coclass graphs
and their mainlines, parametrized presentations,
polarization, and defect
are provided in preliminary sections \S\S\
\ref{s:CoclGrph},
\ref{ss:MtabCocl1},
and
\ref{ss:MtabTyp33CoclGe2}.

In
\cite{Ma3},
it was shown for number fields \(K\) of type \((p,p)\),
that the abelian type invariants of the \(p\)-class groups \(\mathrm{Cl}_p(L_i)\) of
\textit{all} intermediate fields \(K<L_i<\mathrm{F}_p^1(K)\), \(1\le i\le p+1\),
that is, the structures of the \textit{first layer} of the TTT \(\tau(K)\),
usually determine the TKT \(\varkappa(K)\),
at least in the case that \(\mathrm{G}_p^2(K)\)
is one of the most densely populated metabelian \(p\)-groups.

The density of population of a metabelian \(p\)-group \(G\)
by second \(p\)-class groups \(\mathrm{G}_p^2(K)\) of certain base fields \(K\)
can be calculated explicitly from a purely group theoretic probability measure
by non-abelian generalizations
of the Cohen-Lenstra-Martinet asymptotic,
as developed recently by Boston, Bush, Hajir
\cite{BBH},
and also by Bembom
\cite{Bm},
resp. Boy
\cite{By},
under supervision by Mihailescu, resp. Malle.
The heuristic is in good accordance with our extensive computational results
for quadratic base fields in
\cite{Ma1}.
These results have in fact actually been used in
\cite[pp. 5, 126]{Bm}.
Further, they eliminate all incomplete IPADs (index-\(p\) abelianization data),
which coincide with the first layer of our TTTs,
and correct the frequencies given in
\cite[Tbl. 1--2, pp. 17--18]{BBH},
which are uniformly slightly too low.

It is to be expected that similar strategies,
exploiting the interplay between TTT and TKT,
but now extended to the \textit{higher layers} of these invariants,
can be used to identify the isomorphism type of the second \(p\)-class group \(G=\mathrm{G}_p^2(K)\) of
number fields \(K\) with more complicated \(p\)-class group \(\mathrm{Cl}_p(K)\),
for example of type \((p^2,p)\) or \((p,p,p)\).
Extensions in this direction will be presented
in subsequent papers
\cite{Ma4,AZTM}.
An outlook is given in section \S\
\ref{s:DoubleLayer}.



\section{Visualizing finite \(p\)-groups on coclass graphs}
\label{s:CoclGrph}

\subsection{Periodic patterns}
\label{ss:Periodicity}

An important purpose of this paper is
to emphasize that coclass graphs \(\mathcal{G}(p,r)\)
are particularly well suited for visualizing
periodic properties
\cite{dS,EkLg}
of infinite sequences of finite \(p\)-groups \(G\),
such as parametrized power-commutator presentations
\cite{Bl1,Ne1},
automorphism groups \(\mathrm{Aut}(G)\),
Schur multipliers \(\mathrm{H}_2(G,\mathbb{Z}_p)\)
and other cohomology groups of \(G\),
transfer kernel types \(\varkappa(G)\)
\cite{Ma2},
transfer target types \(\tau(G)\)
\cite{Ma3},
and defect of commutativity \(k(G)\) expressed by the depth \(\mathrm{dp}(G)\)
(Corollaries
\ref{cor:DpthCocl1}
and
\ref{cor:DpthCoclGe2}).
In number theoretic applications,
selection rules for second \(p\)-class groups \(G=\mathrm{G}_p^2(K)\) of special base fields \(K\)
\cite{Ma1}
are additional periodic properties.
Computational results on the density of distribution of second \(p\)-class groups
can also be represented very clearly on coclass graphs.



\subsection{Coclass graphs}
\label{ss:CoclGrph}

For a given prime \(p\), Leedham-Green and Newman
\cite{LgNm}
have defined the structure of a directed graph \(\mathcal{G}(p)\)
on the set of all isomorphism classes of finite \(p\)-groups.
Two vertices are connected by a directed edge \(H\to G\) 
if \(G\) is isomorphic to the last lower central quotient \(H/\gamma_c(H)\) of \(H\),
where \(c\) denotes the nilpotency class \(\mathrm{cl}(H)\) of \(H\).\\
If the condition \(\lvert H\rvert=p\lvert G\rvert\) is imposed on the edges,
\(\mathcal{G}(p)\) is partitioned into countably many disjoint subgraphs \(\mathcal{G}(p,r)\), \(r\ge 0\),
called \textit{coclass graphs} of \(p\)-groups \(G\) of coclass \(r=\mathrm{cc}(G)=n-\mathrm{cl}(G)\),
where \(\lvert G\rvert=p^n\)
\cite[p. 155, 166]{LgMk}.
A coclass graph \(\mathcal{G}(p,r)\)
is a forest of finitely many coclass trees \(\mathcal{T}_i\),
each with a single infinite mainline having a pro-\(p\) group of coclass \(r\) as its inverse limit,
and additionally contains a set \(\mathcal{G}_0(p,r)\) of finitely many sporadic groups outside of coclass trees,
 \(\mathcal{G}(p,r)=\left(\cup_i\,\mathcal{T}_i\right)\cup\mathcal{G}_0(p,r)\).



The terminology concerning the structure of coclass graphs \(\mathcal{G}(p,r)\)
with a prime \(p\ge 2\) and an integer \(r\ge 0\)
must be recalled briefly.
We adopt the most recent view of coclass graphs, which is given by
Eick and Leedham-Green
\cite{EkLg},
and by Dietrich, Eick, Feichtenschlager
\cite[p. 46]{DEF}.
 
\begin{itemize}
\item
The \textit{coclass} \(\mathrm{cc}(G)\) of a finite \(p\)-group \(G\)
of order \(\lvert G\rvert=p^n\) and nilpotency class \(\mathrm{cl}(G)\)
is defined by \(n=\mathrm{cl}(G)+\mathrm{cc}(G)\).
\item
By a \textit{vertex} of the coclass graph \(\mathcal{G}(p,r)\) we understand
the isomorphism class of a finite \(p\)-group \(G\) of coclass \(\mathrm{cc}(G)=r\).
\item
The vertex \(H\) is an \textit{immediate descendant} of the vertex \(G\),
if \(G\) is isomorphic to the last lower central quotient \(H/\gamma_c(H)\) of \(H\),
where \(c=\mathrm{cl}(H)\) denotes the nilpotency class of \(H\),
and \(\gamma_c(H)\) is cyclic of order \(p\),
that is, \(\mathrm{cl}(H)=1+\mathrm{cl}(G)\) and \(\lvert H\rvert=p\lvert G\rvert\).
In this case, \(H\) and \(G\) are connected by a \textit{directed edge} \(H\to G\)
of the coclass graph
and \(G\) is called the \textit{parent} \(G=\pi(H)\) of \(H\).
\item
A \textit{capable vertex} has at least one immediate descendant,
whereas a \textit{terminal vertex} has no immediate descendants.
\item
The vertex \(G_m\) is a \textit{descendant} of the vertex \(G_0\),
if there is a \textit{path} \((G_j\to G_{j-1})_{m\ge j\ge 1}\) of directed edges from \(G_m\) to \(G_0\).
In particular, the vertex \(G_0\) is descendant of itself, with empty path.
\item
The \textit{tree} \(\mathcal{T}(G)\) with root \(G\)
consists of all descendants of the vertex \(G\).
\item
A \textit{coclass tree}
is a maximal rooted tree containing exactly one infinite path.
\item
The \textit{mainline} \((M_{j+1}\to M_j)_{j\ge n}\)
of a coclass tree \(\mathcal{T}(M_n)\) with root \(M_n\) of order \(\lvert M_n\rvert=p^n\)
is its unique maximal infinite path.
The projective limit
\(S=\lim\limits_{\longleftarrow}{}_{j\ge n}\,M_j\),
is an infinite pro-\(p\) group,
whose finite quotients by closed subgroups
return the mainline vertices \(M_j\).
\item
For \(i\ge n\), the \textit{branch} \(\mathcal{B}(M_i)\)
of a coclass tree \(\mathcal{T}(M_n)\) with tree root \(M_n\)
and mainline \((M_{j+1}\to M_j)_{j\ge n}\)
is the difference set \(\mathcal{T}(M_i)\setminus\mathcal{T}(M_{i+1})\).
The branch \(\mathcal{B}(M_i)\) is briefly denoted by \(\mathcal{B}_i\)
and we assume that the order of the branch root \(M_i\) is \(\lvert M_i\rvert=p^i\).
\item
The \textit{depth} \(\mathrm{dp}(G)=m-j\)
of a vertex \(G\) of order \(\lvert G\rvert=p^m\) on a branch \(\mathcal{B}(M_j)\) of a coclass tree
is its distance from the branch root \(M_j\) of order \(\lvert M_j\rvert=p^j\) on the mainline.
For \(d\ge 1\),
\(\mathcal{B}_d(M_j)\) denotes the \textit{pruned branch of bounded depth} \(d\) with root \(M_j\).
\item
The \textit{periodic sequence} \(\mathcal{S}(G)\) of a vertex \(G\in\mathcal{B}_d(M_i)\)
of order \(\lvert G\rvert=p^m\), \(i\le m\le i+d\),
on a coclass tree of \(\mathcal{G}(p,r)\),
where \(M_i\) denotes the vertex of order \(p^i\) on the mainline
and \(i\) is sufficiently large so that periodicity has set in already
\cite{EkLg},
is the infinite sequence \((G_{m+j\ell})_{j\ge 0}\) of vertices defined recursively by
\(G_m=G\) and \(G_{m+j\ell}=\varphi_{i+(j-1)\ell}(G_{m+(j-1)\ell})\), for \(j\ge 1\),
using the periodicity isomorphisms of graphs
\(\varphi_{i+(j-1)\ell}:\mathcal{B}_d(M_{i+(j-1)\ell})\to\mathcal{B}_d(M_{i+j\ell})\)
with period length \(\ell\), which is a divisor of \(p^{r+1}(p-1)\).
\end{itemize}



\section{\(p\)-Groups with single layered metabelianization of type \((p,p)\)}
\label{s:SingleLayer}

\subsection{Metabelian \(p\)-groups \(G\) of coclass \(\mathrm{cc}(G)=1\)}
\label{ss:MtabCocl1}

For an arbitrary prime \(p\ge 2\),
let \(G\) be a metabelian \(p\)-group
of order \(\lvert G\rvert=p^n\) and
nilpotency class \(\mathrm{cl}(G)=n-1\), where \(n\ge 3\).
In the terminology of Blackburn
\cite{Bl1}
and Miech
\cite{Mi},
\(G\) is of maximal class,
that is, of coclass \(\mathrm{cc}(G)=1\),
whence the commutator factor group \(G/G^\prime\) of \(G\) is of type \((p,p)\).
The converse is only true for \(p=2\):
A \(2\)-group \(G\) with \(G/G^\prime\simeq(2,2)\) is of coclass \(1\),
a fact which is usually attributed to Taussky
\cite{Ta1}.
The lower central series of \(G\) is defined
recursively by \(\gamma_1(G)=G\) and
\(\gamma_j(G)=\lbrack\gamma_{j-1}(G),G\rbrack\) for \(j\ge 2\).
Nilpotency of \(G\) is expressed by
\(\gamma_{n-1}(G)>\gamma_n(G)=1\).



\subsubsection{Polarization and defect}
\label{sss:PlrzDfct}

The \textit{two-step centralizer}
\(\chi_2(G)
=\lbrace g\in G\mid\lbrack g,u\rbrack\in\gamma_4(G)\text{ for all }u\in\gamma_2(G)\rbrace\)
of the two-step factor group \(\gamma_2(G)/\gamma_4(G)\),
which can also be defined by
\[\chi_2(G)/\gamma_4(G)
=\mathrm{Centralizer}_{G/\gamma_4(G)}(\gamma_2(G)/\gamma_4(G))\,,\]
is the largest subgroup of \(G\) such that
\(\lbrack\chi_2(G),\gamma_2(G)\rbrack\le\gamma_4(G)\).
It is characteristic,
contains the commutator subgroup \(\gamma_2(G)\), and
coincides with \(G\) if and only if \(n=3\).
For \(n\ge 4\),
\(\chi_2(G)\) is one of the maximal subgroups \((H_i)_{1\le i\le p+1}\) of \(G\)
and causes a \textit{polarization} among them,
which will be standardized in Definition
\ref{dfn:NatOrdCocl1}.
Let the isomorphism invariant \(k=k(G)\) of \(G\) be defined by
\[\lbrack\chi_2(G),\gamma_2(G)\rbrack=\gamma_{n-k}(G)\,,\]
where \(k=0\) for \(n=3\), \(0\le k\le n-4\) for \(n\ge 4\),
and \(0\le k\le\min\lbrace n-4,p-2\rbrace\) for \(n\ge p+1\),
according to Miech
\cite[p. 331]{Mi}.
\(k(G)\) provides a measure for the deviation from the maximal degree of commutativity
\(\lbrack\chi_2(G),\gamma_2(G)\rbrack=1\)
and will be called \textit{defect of commutativity} of \(G\).



\subsubsection{Parametrized presentation}
\label{sss:PrmtPres}

Suppose that generators of \(G=\langle x,y\rangle\) are selected such that
\(x\in G\setminus\chi_2(G)\), if \(n\ge 4\), and \(y\in\chi_2(G)\setminus\gamma_2(G)\),
and define the main commutator by
\(s_2=\lbrack y,x\rbrack\in\gamma_2(G)\)
and the higher commutators by
\(s_j=\lbrack s_{j-1},x\rbrack=s_{j-1}^{x-1}\in\gamma_j(G)\) for \(j\ge 3\).
We use identifiers \(s_j\) to emphasize those elements of \(G\) for which
addition of symbolic exponents \(f_1,f_2\) in the group ring \(\mathbb{Z}\lbrack G\rbrack\)
is commutative, \(s_j^{f_1+f_2}=s_j^{f_2+f_1}\).
Nilpotency of \(G\) is expressed by \(s_n=1\)
and a \textit{power-commutator presentation} of \(G\)
with generators \(x,y,s_2,\ldots,s_{n-1}\) is given as follows.
There are two relations for \(p\)th powers of the generators \(x\) and \(y\) of \(G\),

\begin{equation}
\label{eqn:PwrRelCocl1}
x^p=s_{n-1}^w\quad\text{ and }\quad
y^p\prod_{\ell=2}^p\,s_\ell^{\binom{p}{\ell}}=s_{n-1}^z
\quad \text{ with exponents }\quad 0\le w,z\le p-1\,,
\end{equation}

\noindent
according to Miech
\cite[p. 332, Thm. 2, (3)]{Mi}.
Blackburn uses the notation \(\delta=w\) and \(\gamma=z\)
for these relational exponents
\cite[p. 84, (36), (37)]{Bl1}.

Additionally, the group \(G\) satisfies
relations for \(p\)th powers of the higher commutators,
\[s_{j+1}^p\prod_{\ell=2}^p\,s_{j+\ell}^{\binom{p}{\ell}}=1\quad\text{ for }1\le j\le n-2\,,\]
and the commutator relation of Miech
\cite[p. 332, Thm. 2, (2)]{Mi}, containing the defect \(k=k(G)\),

\begin{equation}
\label{eqn:CmtRelCocl1}
\lbrack y,s_2\rbrack=\prod_{\ell=1}^k s_{n-\ell}^{a(n-\ell)}
\in\lbrack\chi_2(G),\gamma_2(G)\rbrack=\gamma_{n-k}(G)\,,
\end{equation}

\noindent
with exponents \(0\le a(n-\ell)\le p-1\) for \(1\le\ell\le k\), and \(a(n-k)>0\), if \(k\ge 1\).
Blackburn restricts his investigations to \(k\le 2\) and uses the notation
\(\beta=a(n-1)\) and \(\alpha=a(n-2)\)
\cite[p. 82, (33)]{Bl1}.

By \(G_a^n(z,w)\) we denote 
the representative of an isomorphism class of
metabelian \(p\)-groups \(G\) of coclass \(\mathrm{cc}(G)=1\)
and order \(\lvert G\rvert=p^n\),
which satisfies the relations
(\ref{eqn:PwrRelCocl1})
and
(\ref{eqn:CmtRelCocl1})
with a fixed system of exponents
\(a=(a(n-k),\ldots,a(n-1))\),\(w\), and \(z\).
We have \(a=0\) if and only if \(k=0\).



\subsubsection{A distinguished maximal subgroup}
\label{sss:DstgMaxSbgp1}

Since the maximal normal subgroups \(H_i\), \(1\le i\le p+1\), of \(G\)
contain the commutator subgroup \(G^\prime\)
as a normal subgroup of index \(p\),
they are of the shape \(H_i=\langle g_i,G^\prime\rangle\)
with suitable generators \(g_i\),
and we can arrange them in a fixed order.

\begin{definition}
\label{dfn:NatOrdCocl1}
The \textit{polarization} or \textit{natural order}
of the maximal subgroups \((H_i)_{1\le i\le p+1}\) of \(G\)
is given by the \textit{distinguished first generator} \(g_1=y\in\chi_2(G)\)
and the other generators \(g_i=xy^{i-2}\notin\chi_2(G)\) for \(2\le i\le p+1\),
provided that \(\lvert G\rvert\ge p^4\).
Then, in particular \(\chi_2(G)=H_1=\langle y,G^\prime\rangle\).
\end{definition}



\subsubsection{Parents of CF groups}
\label{sss:PrntCocl1}

Together with group counts in Blackburn's theorems
\cite[p. 88, Thm. 4.1--4.3]{Bl1},
Theorem
\ref{thm:PrntCocl1}
describes the structure
of the \textit{metabelian skeleton}
of the unique coclass tree \(\mathcal{T}(C_p\times C_p)\)
\cite[\S\ 1, p. 851]{Dt1}
of the coclass graph \(\mathcal{G}(p,1)\)
with an arbitrary prime \(p\ge 2\).
The graph consists of all isomorphism classes of CF \textit{groups}
(with cyclic factors) 
\cite[\S\ 4, p. 264]{AHL}
of coclass \(1\).



\begin{theorem}
\label{thm:PrntCocl1}
Let \(p\ge 2\) be an arbitrary prime,
and \(G\) be a metabelian \(p\)-group of coclass \(\mathrm{cc}(G)=1\)
having defect of commutativity \(k=k(G)\),
such that \(G\simeq G_a^n(z,w)\) with parameters
\(n\ge 3\),
\(a=(a(n-k),\ldots,a(n-1))\),
\(0\le a(n-k),\ldots,a(n-1),w,z<p\), where \(a(n-k)>0\), if \(k\ge 1\),
that is, \(G\) is of order \(\lvert G\rvert=p^n\) and nilpotency class \(\mathrm{cl}(G)=n-1\).
Then the parent \(\pi(G)\) of \(G\) on the coclass tree \(\mathcal{T}(C_p\times C_p)\)
is given by
\[\pi(G)\simeq
\begin{cases}
C_p\times C_p, & \text{ if } n=3 \text{ (and thus } k=0), \\
G_0^{n-1}(0,0), & \text{ if } n\ge 4,\ k=0, \\
G_0^{n-1}(0,0), & \text{ if } n\ge 5,\ k=1 \text{ (and thus } p\ge 3), \\
G_{\tilde a}^{n-1}(0,0), & \text{ where }\tilde a=(a(n-k),\ldots,a(n-2)),\text{ if } n\ge 6,\ k\ge 2 \text{ (and thus } p\ge 5).
\end{cases}
\]
\end{theorem}



\begin{remark}
\label{rmk:PrntCocl1}
The various cases of Theorem
\ref{thm:PrntCocl1}
can be described as follows.

\begin{enumerate}
\item
In the first case, \(n=3\),
where \(G\simeq G_0^3(0,w)\) with \(0\le w\le 1\) is an extra-special \(p\)-group of order \(p^3\),
the parent \(\pi(G)\) is the abelian root \(C_p\times C_p\) of the tree \(\mathcal{T}(C_p\times C_p)\),
which can formally be viewed as \(G_0^2(0,0)\).
\item
In the second and third case of a group \(G\) of defect \(k\le 1\),
the parent \(\pi(G)\) is a mainline group.
\item
In the last case of a group \(G\) of higher defect \(k\ge 2\),
which can occur only for \(p\ge 5\),
the parent \(\pi(G)\) lies outside of the mainline and
the defect \(\tilde k\) and family \(\tilde a\) of relational exponents of \(\pi(G)\)
are given by
\(\tilde k=k-1\)
and
\(\tilde a=(\tilde a((n-1)-(k-1)),\ldots,\tilde a((n-1)-1))=(a(n-k),\ldots,a(n-2))\),
where \(\tilde a((n-1)-(k-1))>0\).
We point out that
the parent is always characterized by parameters \(\tilde z=0\) and \(\tilde w=0\).
\end{enumerate}

\end{remark}



\begin{proof}
For \(n=3\), \(G\) is an extra special \(p\)-group of nilpotency class \(\mathrm{cl}(G)=n-1=2\)
having the commutator subgroup \(\gamma_2(G)\) as its last (non-trivial) lower central \(\gamma_{n-1}(G)\).
In this special case, the definition of the parent \(\pi(G)=G/\gamma_{n-1}(G)\) of \(G\)
yields the abelianization \(\pi(G)=G/\gamma_2(G)\) of type \((p,p)\),
which is isomorphic to the root \(C_p\times C_p\)
of \(\mathcal{T}(C_p\times C_p)\).

For \(n\ge 4\), \(G\) can be assumed to be isomorphic to a group \(G\simeq G_a^n(z,w)\)
with pc-presentation consisting of the relations
(\ref{eqn:PwrRelCocl1})
and
(\ref{eqn:CmtRelCocl1})
for the two generators \(x,y\),
\[x^p=s_{n-1}^w,\qquad y^p\prod_{\ell=2}^{p}\,s_{\ell}^{\binom{p}{\ell}}=s_{n-1}^z,\qquad \lbrack y,s_2\rbrack=\prod_{\ell=1}^k\,s_{n-\ell}^{a(n-\ell)}.\]
Since the parent \(\pi(G)=G/\gamma_{n-1}(G)\) of \(G\) is defined as the last lower central quotient,
we denote the left coset of an element \(g\in G\) with respect to \(\gamma_{n-1}(G)\) by \(\bar{g}=g\cdot\gamma_{n-1}(G)\)
and we obtain \(\bar{s}_{n-1}=1\), because \(\gamma_{n-1}(G)=\langle s_{n-1}\rangle\).
Therefore, the nilpotency class of the parent is \(\mathrm{cl}(\pi(G))=\mathrm{cl}(G)-1=n-2\)
and a pc-presentation of \(\pi(G)\) is given by
\[\bar{x}^p=\bar{s}_{n-1}^w=1,\qquad \bar{y}^p\prod_{j=2}^{p}\,\bar{s}_{j}^{\binom{p}{j}}=\bar{s}_{n-1}^z=1,\qquad \lbrack\bar{y},\bar{s}_2\rbrack=\prod_{\ell=1}^k\,\bar{s}_{n-\ell}^{a(n-\ell)},\]
where the last product equals \(1\), if \(k\le 1\),
and \(\prod_{\ell=2}^k\,\bar{s}_{n-\ell}^{a(n-\ell)}\ne 1\), if \(k\ge 2\), because \(\bar{s}_{n-k}^{a(n-k)}\ne 1\).
Since the order of the parent is \(\lvert\pi(G)\rvert=\lvert G\rvert:\lvert\gamma_{n-1}(G)\rvert=p^n:p=p^{n-1}\),
the coclass remains the same \(\mathrm{cc}(\pi(G))=n-1-\mathrm{cl}(\pi(G))=n-1-(n-2)=1=\mathrm{cc}(G)\).
\end{proof}



The following principle,
that the kernel \(\varkappa(1)\) of the transfer from \(G\) to the first distinguished maximal subgroup \(H_1=\chi_2(G)\)
decides about the relation between depth \(\mathrm{dp}(G)\) and defect \(k=k(G)\) of \(G\),
will turn out to be crucial
for metabelian \(p\)-groups \(G\) of coclass \(\mathrm{cc}(G)\ge 2\), too.

\begin{corollary}
\label{cor:DpthCocl1}
For a metabelian \(p\)-group \(G\) of coclass \(\mathrm{cc}(G)=1\) with defect of commutativity \(k=k(G)\),
the depth \(\mathrm{dp}(G)\) of \(G\)
on the coclass tree \(\mathcal{T}(C_p\times C_p)\) of \(\mathcal{G}(p,1)\) is given by
\[\mathrm{dp}(G)=
\begin{cases}
k+1, & \text{ if } \varkappa(1)\ne 0, \\
k,   & \text{ if } \varkappa(1)=0,
\end{cases}
\]
with respect to the natural order of the maximal subgroups of \(G\).
\end{corollary}

\begin{proof}
Theorem
\ref{thm:PrntCocl1}
shows that
\((G_0^n(0,0))_{n\ge 2}\) is the mainline of the coclass tree \(\mathcal{T}(C_p\times C_p)\),
consisting of all groups \(G\) of depth \(\mathrm{dp}(G)=0\) and defect \(k=0\),
because each of these vertices occurs as a parent and possesses infinitely many descendants,
whereas the groups \(G_a^n(0,0)\) with \(a\ne 0\), \(k\ge 1\) can only have finitely many descendants,
due to the bound \(k\le p-2\) by Miech
\cite{Mi}.
Since the defect of any group \(G_a^n(z,w)\) with parameter \(a=0\) is given by \(k=0\),
all the other groups \(G=G_0^n(z,w)\), \((z,w)\ne (0,0)\),
which contain \(H_1\) as an abelian maximal subgroup,
must be located as terminal vertices at depth \(\mathrm{dp}(G)=1=k+1\),
because they never occur as a parent.
On the other hand, the third and fourth case of Theorem
\ref{thm:PrntCocl1}
show that the relation between the defects of parent \(\pi(G)\) and immediate descendant \(G\)
is given by \(\tilde k=k-1\) for any group \(G=G_a^n(z,w)\), \(a\ne 0\),
with positive defect \(k\ge 1\), whence the depth,
being the number of steps required to reach the mainline
by successive construction of parents, \((G,\pi(G),\pi^2(G),\ldots)\),
is given by \(\mathrm{dp}(G)=k\).
Finally, the groups \(G=G_0^n(z,w)\), \((z,w)\ne (0,0)\),
containing the abelian maximal subgroup \(H_1\),
are characterized uniquely by a partial transfer \(\varkappa(1)\ne 0\)
to the distinguished maximal subgroup \(H_1\),
according to
\cite[Thm. 2.5--2.6]{Ma2}.
\end{proof}



We conjecture that the following property 
of mainline groups of \(\mathcal{G}(p,1)\)
might be true for mainline groups on any
coclass tree of \(\mathcal{G}(p,r)\), \(r\ge 1\).

\begin{corollary}
\label{cor:MainLineCocl1}
Mainline groups of \(\mathcal{G}(p,1)\),
that is, groups of depth \(\mathrm{dp}(G)=0\),
must have a total transfer \(\varkappa(1)=0\)
to the distinguished maximal subgroup \(H_1=\chi_2(G)\).
The converse is only true for \(p=2\):
A \(2\)-group \(G\in\mathcal{T}(C_2\times C_2)\)
having \(\varkappa(1)=0\) is mainline.

\end{corollary}

\begin{proof}
The statement for \(p\ge 2\) is an immediate consequence of Corollary
\ref{cor:DpthCocl1}
and it only remains to prove the converse for \(p=2\).
This, however, is contained in
\cite[Thm. 2.6, p. 481]{Ma2}.
\end{proof}



Concerning the transfer kernel type \(\varkappa(G)\) of a \(p\)-group \(G\) of coclass \(1\)
we can state:

\begin{corollary}
\label{cor:TKTCocl1}

The transfer kernel types of groups on the unique coclass tree
\(\mathcal{T}(C_p\times C_p)\) of coclass graph \(\mathcal{G}(p,1)\)
are given by the following rules.

\begin{enumerate}
\item
The root \(C_p\times C_p\) is
of TKT \(\mathrm{a}.1\) \((0^{p+1})\) for any prime \(p\ge 2\).
The extra-special group \(G_0^3(0,1)\) is
of TKT \(\mathrm{A}.1\) \((1^{p+1})\) for odd \(p\ge 3\),
and of TKT \(\mathrm{Q}.5\) \((123)\) for \(p=2\).
In the sequel, these exceptions are excluded.
\item
Mainline groups are
of TKT \(\mathrm{a}.1\) \((0^{p+1})\) for odd \(p\ge 3\),
and of TKT \(\mathrm{d}.8\) \((032)\) for \(p=2\).
\item
Groups of depth \(1\) and defect \(0\) are
of TKT either \(\mathrm{a}.2\) \((1,0^p)\) or \(\mathrm{a}.3\) \((2,0^p)\) for \(p\ge 3\),
and of TKT either \(\mathrm{Q}.6\) \((132)\) or \(\mathrm{S}.4\) \((232)\) for \(p=2\).
\item
Groups of positive defect \(1\le k\le p-2\) are
exclusively of TKT \(\mathrm{a}.1\) \((0^{p+1})\).
\end{enumerate}

\end{corollary}

\begin{proof}
This is a result of combining Theorem
\ref{thm:PrntCocl1}
with Theorems 2.5 and 2.6 in
\cite{Ma2}.
\end{proof}



\subsection{Second \(p\)-class groups \(G=\mathrm{G}_p^2(K)\) of coclass \(\mathrm{cc}(G)=1\)}
\label{ss:ScndClgpCocl1}

\subsubsection{Weak transfer target type \(\tau_0(G)\) expressed by \(p\)-class numbers}
\label{sss:wTTTCocl1}

The group theoretic information
on the second \(p\)-class group \(G=\mathrm{G}_p^2(K)\),
that is, order, class, coclass, and defect,
is contained in the \(p\)-class numbers
of the distinguished extension \(L_1\)
and of the Hilbert \(p\)-class field \(\mathrm{F}_p^1(K)\).
Additionally,
the principalization \(\kappa(1)\) of \(K\) in the distinguished extension \(L_1\)
determines the connection between defect and depth of \(G\).

\begin{theorem}
\label{thm:wTTTCocl1}

Let \(K\) be an arbitrary number field
with \(p\)-class group \(\mathrm{Cl}_p(K)\) of type \((p,p)\).
Suppose that the second \(p\)-class group
\(G=\mathrm{Gal}(\mathrm{F}_p^2(K)\vert K)\)
is abelian or metabelian of coclass \(\mathrm{cc}(G)=1\)
with defect \(k=k(G)\),
order \(\lvert G\rvert=p^n\),
and class \(\mathrm{cl}(G)=n-1\),
where \(n\ge 2\).
With respect to the natural order
among the maximal subgroups of \(G\),
the weak transfer target type \(\tau_0(G)\) of \(G\),
that is, the family of \(p\)-class numbers of the multiplet \((L_1,\ldots,L_{p+1})\)
of unramified cyclic extension fields of \(K\)
of relative prime degree \(p\ge 2\) is given for the first layer by

\begin{eqnarray*}
\tau_0(G) = (\mathrm{h}_p(L_1),\mathrm{h}_p(L_2),\ldots,\mathrm{h}_p(L_{p+1}) =
\begin{cases}
(\overbrace{p\ldots,p}^{p+1\text{ times}}), & \text{ if }n=2,\\
(p^{\mathrm{cl}(G)-k},\overbrace{p^2,\ldots,p^2}^{p\text{ times}}), & \text{ if }n\ge 3,
\end{cases}
\end{eqnarray*}

\noindent
where defect \(k\) and depth \(\mathrm{dp}(G)\) of \(G\) are related by
\[k=
\begin{cases}
\mathrm{dp}(G)-1, & \text{ if } \varkappa(1)\ne 0, \\
\mathrm{dp}(G),   & \text{ if } \varkappa(1)=0,
\end{cases}
\]

\noindent
and for the single member of the second layer by
\[\mathrm{h}_p(\mathrm{F}_p^1(K)) = p^{\mathrm{cl}(G)-1}.\]

\end{theorem}

\begin{proof}
The statement is a succinct version of
\cite[Thm. 3.2]{Ma1},
expressed by concepts more closely related to the
position of \(G\) on the coclass graph \(\mathcal{G}(p,1)\)
and to the transfer kernel type \(\varkappa(G)\) of \(G\),
using Corollary
\ref{cor:DpthCocl1}.
\end{proof}

\begin{remark}
Whereas \(\mathrm{h}_p(L_2),\ldots,\mathrm{h}_p(L_{p+1})\) only indicate that \(\mathrm{cc}(G)=1\),
the \(p\)-class number \(\mathrm{h}_p(\mathrm{F}_p^1(K))\) of the Hilbert \(p\)-class field of \(K\)
determines the order \(p^n\), \(n=\mathrm{cl}(G)+1\), and class of \(G\),
and the distinguished \(\mathrm{h}_p(L_1)\) gives the defect \(k\) of \(G\).\\
With respect to the mainline \((M_j)_{j\ge 2}\) of the coclass tree \(\mathcal{T}(C_p\times C_p)\),
the order \(\lvert M_i\rvert=3^i\) of the branch root \(M_i\) of \(G\) is given by
\(i=n-\mathrm{dp}(G)=\mathrm{cl}(G)+1-\mathrm{dp}(G)\),
where
\[\mathrm{dp}(G)=
\begin{cases}
k,   & \text{ if } \varkappa(1)=0, \\
k+1, & \text{ if } \varkappa(1)\ne 0.
\end{cases}
\]
\end{remark}



\subsubsection{The complete coclass graph \(\mathcal{G}(2,1)\)}
\label{sss:Distr2Cocl1}

We start our investigation of special cases by showing that
the distribution of second \(2\)-class groups \(\mathrm{G}_2^2(K)\)
of complex quadratic fields \(K=\mathbb{Q}(\sqrt{D})\), \(D<0\), with \(\mathrm{Cl}_2(K)\simeq(2,2)\)
on \(\mathcal{G}(2,1)\) is not restricted by selection rules.
This distribution will only be given qualitatively, without exact counts.

\begin{theorem}
\label{thm:2Cocl1}
The diagram in Figure
\ref{fig:TKT2Cocl1}
visualizes the complete coclass graph \(\mathcal{G}(2,1)\)
up to order \(2^8=256\).
It is periodic with length \(1\).
The first period consists of branch \(\mathcal{B}_3\),
whereas branch \(\mathcal{B}_2\) is irregular and forms the pre-period.
\end{theorem}

\begin{proof}
\(\mathcal{G}(2,1)\) begins with two abelian groups of order \(2^2\),
the isolated cyclic group \(C_4\), having different abelianization,
and Klein's four group \(V_4\),
that is the bicyclic root \(C_2\times C_2\)
of the unique coclass tree \(\mathcal{T}(C_2\times C_2)\).\\
As immediate descendants of the root,
\(\mathcal{G}(2,1)\) contains
the capable mainline group \(D(8)\)
and the terminal group \(Q(8)\),
both of order \(2^3\).

Applying Blackburn's results
\cite{Bl1}
on counts of metabelian \(p\)-groups of maximal class
and order \(p^n\) with \(n\ge 4\),
to the special case \(p=2\),
we only need to consider
metabelian groups containing an abelian maximal subgroup,
characterized by defect \(k=0\).
They consist of the capable mainline group \(D(2^n)=G_0^n(0,0)\),
the terminal group \(Q(2^n)=G_0^n(0,1)\), and
the terminal group \(S(2^n)=G_0^n(1,0)\),
which is expressed by
specialization of
\cite[p. 88, Thm. 4.3]{Bl1}
to \(p=2\).
The count is independent from \(n\), yielding the constant number
\(2+(n-2,p-1)=3\).

\end{proof}

We recall from
\cite{Ma2}
that the transfer kernel types \(\varkappa(G)\) for \(p\)-groups of coclass \(\mathrm{cc}(G)=1\)
are exceptional in the case \(p=2\),
compared to the uniform standard case of odd primes \(p\ge 3\).

\begin{theorem}
\label{thm:TKT2Cocl1}
Table
\ref{tab:TKT2Cocl1}
gives the transfer kernel type \(\varkappa(G)\) of all
non-isolated vertices \(G\), having abelianization \(G/G^\prime\simeq(2,2)\),
on the coclass graph \(\mathcal{G}(2,1)\).
The \(2\)-groups \(G\) are identified by their Blackburn invariants
\(\lvert G\rvert=2^n\)
and \(a,w,z\) as exponents in the relations
(\ref{eqn:PwrRelCocl1})
and
(\ref{eqn:CmtRelCocl1}).
The graph information gives the depth \(\mathrm{dp}(G)\) and the location of each \(2\)-group \(G\)
with respect to the unique coclass tree \(\mathcal{T}(C_2\times C_2)\)
of \(\mathcal{G}(2,1)\).\\
The mainline, consisting of the dihedral \(2\)-groups \(D(2^n)=G_0^n(0,0)\)
including the abelian root \(C_2\times C_2=D(4)=G_0^2(0,0)\), is characterized
by the total transfer \(\varkappa(1)=0\) to the distinguished maximal subgroup \(H_1=\chi_2(G)\).
Total transfers \(\varkappa(i)=0\) are counted by \(\nu(G)\).
\end{theorem}

\begin{table}[ht]
\caption{\(\varkappa(G)\), \(\nu(G)\) in dependence on non-isolated \(G\in\mathcal{G}(2,1)\)}
\label{tab:TKT2Cocl1}
\begin{center}
\begin{tabular}{|c|cc|ccc|c|cc|ccc|}
\hline
 \multicolumn{7}{|c|}{\(2\)-group \(G_a^n(z,w)\) of coclass \(1\)}                  & \multicolumn{2}{|c|}{graph information} & \multicolumn{3}{|c|}{transfer kernels} \\
 \multicolumn{7}{|c|}{\downbracefill}                                               & \multicolumn{2}{|c|}{\downbracefill}    & \multicolumn{3}{|c|}{\downbracefill}   \\
 \(G\)             & \(\mathrm{cl}(G)\) &     \(n\) & \(a\) & \(z\) & \(w\) & \(k\) & \(\mathrm{dp}(G)\) & tree position      & TKT & \(\varkappa(G)\) & \(\nu(G)\)          \\
\hline
 \(C_2\times C_2\) &              \(1\) &     \(2\) & \(0\) & \(0\) & \(0\) & \(0\) &              \(0\) & root               & a.1 & \((000)\)     & \(3\)            \\
 \(Q(8)\)          &              \(2\) &     \(3\) & \(0\) & \(0\) & \(1\) & \(0\) &              \(1\) & pre-period         & Q.5 & \((123)\)     & \(0\)            \\
\hline
 \(D(2^n)\)        &          \(\ge 2\) & \(\ge 3\) & \(0\) & \(0\) & \(0\) & \(0\) &              \(0\) & mainline           & d.8 & \((032)\)     & \(1\)            \\
\hline
 \(Q(2^n)\)        &          \(\ge 3\) & \(\ge 4\) & \(0\) & \(0\) & \(1\) & \(0\) &              \(1\) & periodic sequence  & Q.6 & \((132)\)     & \(0\)            \\
 \(S(2^n)\)        &          \(\ge 3\) & \(\ge 4\) & \(0\) & \(1\) & \(0\) & \(0\) &              \(1\) & periodic sequence  & S.4 & \((232)\)     & \(0\)            \\
\hline
\end{tabular}
\end{center}
\end{table}

\begin{proof}
See
\cite[Thm. 2.6, Tbl. 2--3]{Ma2}
for the technique of determining kernels of transfers
and the definition of transfer kernel types as orbits of integer triplets \(\lbrack 0,3\rbrack^3\)
under the action of the symmetric group of degree \(3\).
\end{proof}



\begin{figure}[ht]
\caption{Population of \(\mathcal{G}(2,1)\) by groups \(G_2^2(K)\) of \(K=\mathbb{Q}(\sqrt{D})\), \(D<0\)}
\label{fig:TKT2Cocl1}

\setlength{\unitlength}{1cm}
\begin{picture}(12,16)(-9,-15)

\put(-8,0.5){\makebox(0,0)[cb]{Order \(2^n\)}}
\put(-8,0){\line(0,-1){12}}
\multiput(-8.1,0)(0,-2){7}{\line(1,0){0.2}}
\put(-8.2,0){\makebox(0,0)[rc]{\(4\)}}
\put(-7.8,0){\makebox(0,0)[lc]{\(2^2\)}}
\put(-8.2,-2){\makebox(0,0)[rc]{\(8\)}}
\put(-7.8,-2){\makebox(0,0)[lc]{\(2^3\)}}
\put(-8.2,-4){\makebox(0,0)[rc]{\(16\)}}
\put(-7.8,-4){\makebox(0,0)[lc]{\(2^4\)}}
\put(-8.2,-6){\makebox(0,0)[rc]{\(32\)}}
\put(-7.8,-6){\makebox(0,0)[lc]{\(2^5\)}}
\put(-8.2,-8){\makebox(0,0)[rc]{\(64\)}}
\put(-7.8,-8){\makebox(0,0)[lc]{\(2^6\)}}
\put(-8.2,-10){\makebox(0,0)[rc]{\(128\)}}
\put(-7.8,-10){\makebox(0,0)[lc]{\(2^7\)}}
\put(-8.2,-12){\makebox(0,0)[rc]{\(256\)}}
\put(-7.8,-12){\makebox(0,0)[lc]{\(2^8\)}}

\put(-0.1,-0.1){\framebox(0.2,0.2){}}
\put(-2.1,-0.1){\framebox(0.2,0.2){}}
\multiput(0,-2)(0,-2){5}{\circle*{0.2}}
\multiput(-2,-2)(0,-2){6}{\circle*{0.2}}
\multiput(-4,-4)(0,-2){5}{\circle*{0.2}}
\multiput(0,0)(0,-2){6}{\circle{0.6}}
\multiput(-2,-2)(0,-2){6}{\circle{0.6}}
\multiput(-4,-4)(0,-2){5}{\circle{0.6}}

\multiput(0,0)(0,-2){5}{\line(0,-1){2}}
\multiput(0,0)(0,-2){6}{\line(-1,-1){2}}
\multiput(0,-2)(0,-2){5}{\line(-2,-1){4}}

\put(0,-10){\vector(0,-1){2}}
\put(-0.2,-11.5){\makebox(0,0)[rc]{main}}
\put(-0.2,-12){\makebox(0,0)[rc]{line}}
\put(0.2,-12){\makebox(0,0)[lc]{\(\mathcal{T}(C_2\times C_2)\)}}

\put(0,0.5){\makebox(0,0)[cb]{\(C_2\times C_2=V_4\)}}
\put(-2,0.5){\makebox(0,0)[cb]{\(C_4\)}}
\put(-3,-2){\makebox(0,0)[cc]{\(Q(8)\)}}
\put(0,-13){\makebox(0,0)[ct]{\(G^n_0(0,0)\)}}
\put(-2,-13){\makebox(0,0)[ct]{\(G^n_0(0,1)\)}}
\put(-4,-13){\makebox(0,0)[ct]{\(G^n_0(1,0)\)}}
\put(0,-13.5){\makebox(0,0)[ct]{\(=D(2^n)\)}}
\put(-2,-13.5){\makebox(0,0)[ct]{\(=Q(2^n)\)}}
\put(-4,-13.5){\makebox(0,0)[ct]{\(=S(2^n)\)}}

\put(1,0){\makebox(0,0)[lb]{\(\Gamma_1\)}}
\put(-2.1,0){\makebox(0,0)[rb]{\(\langle 1\rangle\)}}
\put(-0.3,0){\makebox(0,0)[rb]{\(\langle 2\rangle\)}}

\put(1,-1.7){\makebox(0,0)[lb]{\(\Gamma_2\)}}
\put(-2.1,-1.7){\makebox(0,0)[rb]{\(\langle 4\rangle\)}}
\put(-0.1,-1.7){\makebox(0,0)[rb]{\(\langle 3\rangle\)}}

\put(1,-3.7){\makebox(0,0)[lb]{\(\Gamma_3\)}}
\put(-4.1,-3.7){\makebox(0,0)[rb]{\(\langle 8\rangle\)}}
\put(-2.1,-3.7){\makebox(0,0)[rb]{\(\langle 9\rangle\)}}
\put(-0.1,-3.7){\makebox(0,0)[rb]{\(\langle 7\rangle\)}}

\put(1,-5.7){\makebox(0,0)[lb]{\(\Gamma_8\)}}
\put(-4.1,-5.7){\makebox(0,0)[rb]{\(\langle 19\rangle\)}}
\put(-2.1,-5.7){\makebox(0,0)[rb]{\(\langle 20\rangle\)}}
\put(-0.1,-5.7){\makebox(0,0)[rb]{\(\langle 18\rangle\)}}

\put(-4.1,-7.7){\makebox(0,0)[rb]{\(\langle 53\rangle\)}}
\put(-2.1,-7.7){\makebox(0,0)[rb]{\(\langle 54\rangle\)}}
\put(-0.1,-7.7){\makebox(0,0)[rb]{\(\langle 52\rangle\)}}

\put(-4.1,-9.7){\makebox(0,0)[rb]{\(\langle 162\rangle\)}}
\put(-1.9,-9.7){\makebox(0,0)[rb]{\(\langle 163\rangle\)}}
\put(-0.1,-9.7){\makebox(0,0)[rb]{\(\langle 161\rangle\)}}

\put(-6,-14.5){\makebox(0,0)[cc]{\textbf{TKT:}}}
\put(0,-14.5){\makebox(0,0)[cc]{d.8}}
\put(-2,-14.5){\makebox(0,0)[cc]{Q.6}}
\put(-4,-14.5){\makebox(0,0)[cc]{S.4}}
\put(0,-15){\makebox(0,0)[cc]{\((032)\)}}
\put(-2,-15){\makebox(0,0)[cc]{\((132)\)}}
\put(-4,-15){\makebox(0,0)[cc]{\((232)\)}}
\put(-6.7,-15.2){\framebox(7.4,1){}}
\put(-5,-2){\makebox(0,0)[rc]{\textbf{TKT:}}}
\put(-4.5,-2){\makebox(0,0)[cc]{Q.5}}
\put(-4.5,-2.5){\makebox(0,0)[cc]{\((123)\)}}
\put(-6.1,-2.7){\framebox(2.4,1){}}
\put(3,0){\makebox(0,0)[rc]{\textbf{TKT:}}}
\put(3.5,0){\makebox(0,0)[cc]{a.1}}
\put(3.5,-0.5){\makebox(0,0)[cc]{\((000)\)}}
\put(1.9,-0.7){\framebox(2.4,1){}}
\put(0.3,-0.5){\makebox(0,0)[lc]{\(D=-84\)}}
\put(0.3,-2.5){\makebox(0,0)[lc]{\(D=-408\)}}
\put(0.3,-4.5){\makebox(0,0)[lc]{\(D=-6\,168\)}}
\put(0.3,-6.5){\makebox(0,0)[lc]{\(D=-29\,208\)}}
\put(0.3,-8.5){\makebox(0,0)[lc]{\(D=-609\,816\)}}
\put(0.3,-10.5){\makebox(0,0)[lc]{\(D=-670\,872\)}}
\put(-1.5,-2.5){\makebox(0,0)[rc]{\(D=-120\)}}
\put(-1.5,-4.5){\makebox(0,0)[rc]{\(D=-312\)}}
\put(-1.5,-6.5){\makebox(0,0)[rc]{\(D=-888\)}}
\put(-1.5,-8.5){\makebox(0,0)[rc]{\(D=-3\,768\)}}
\put(-1.5,-10.5){\makebox(0,0)[rc]{\(D=-8\,952\)}}
\put(-1.5,-12.5){\makebox(0,0)[rc]{\(D=-40\,632\)}}
\put(-3.7,-4.5){\makebox(0,0)[rc]{\(D=-340\)}}
\put(-3.7,-6.5){\makebox(0,0)[rc]{\(D=-2\,260\)}}
\put(-3.7,-8.5){\makebox(0,0)[rc]{\(D=-5\,140\)}}
\put(-3.7,-10.5){\makebox(0,0)[rc]{\(D=-17\,140\)}}
\put(-3.7,-12.5){\makebox(0,0)[rc]{\(D=-165\,460\)}}
\end{picture}

\end{figure}
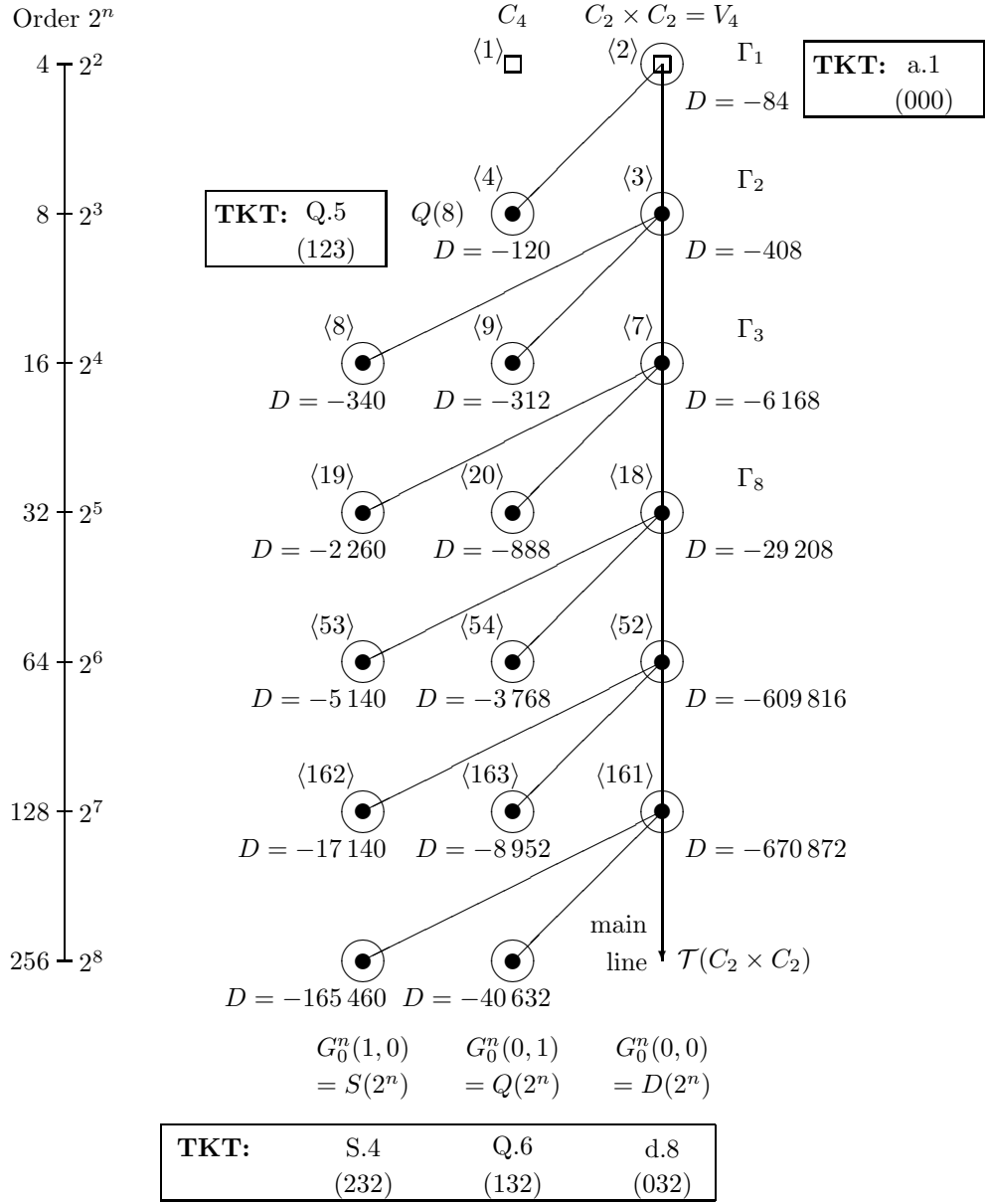

\noindent
The statements of Theorem
\ref{thm:TKT2Cocl1}
can be visualized very conveniently
by the diagram of a finite part of the coclass graph \(\mathcal{G}(2,1)\),
which is shown in Figure
\ref{fig:TKT2Cocl1}.
It contains the isolated vertex \(C_4\),
the root \(C_2\times C_2\) and branches \(\mathcal{B}_j\), \(2\le j\le 7\),
of the coclass tree \(\mathcal{T}(C_2\times C_2)\).
Branch \(\mathcal{B}_2\) consists of two initial exceptions,
the elementary abelian bicyclic \(2\)-group \(C_2\times C_2\) with TKT a.1, \(\varkappa=(000)\),
and the quaternion group \(Q(8)\) with TKT Q.5, \(\varkappa=(123)\).
Periodicity of length \(\ell=1\) sets in with branch  \(\mathcal{B}_3\)
which consists of the starting vertices of three periodic sequences,
\(\mathcal{S}(D(8))\), the mainline of \textit{dihedral} groups,
\(\mathcal{S}(Q(16))\), the sequence of \textit{generalized quaternion} groups, and
\(\mathcal{S}(S(16))\), the sequence of \textit{semi-dihedral} groups.
Transfer kernel types (TKT) in the bottom rectangle concern all vertices
in the periodic sequence located vertically above.
Large contour squares \(\square\) denote abelian groups and
big full discs {\Large \(\bullet\)} denote metabelian groups with defect \(k=0\).
A number in angles gives the identifier of a group in the SmallGroups Library
\cite{BEO}.
The symbols \(\Gamma_s\) denote isoclinism families given by Hall and Senior
\cite{HaSn}.
The population of each vertex is indicated by a surrounding circle
labelled by the discriminant \(D<0\) of a
suitable complex quadratic field \(K=\mathbb{Q}(\sqrt{D})\) of type \((2,2)\).
There are no selection rules for \(p=2\) and
the numerical results suggest the conjecture that
the tree \(\mathcal{T}(V_4)\) is covered entirely by second \(2\)-class groups \(G_2^2(K)\)
of complex quadratic fields \(K=\mathbb{Q}(\sqrt{D})\), \(D<0\).
Ground states are due to Kisilevsky
\cite[p. 277--278]{Ki2}.
All excited states have been determined with the aid of Theorem
\ref{thm:wTTTCocl1}.
See
\cite[\S\ 9]{Ma1}.
Here we refrain from giving the exact distribution up to some bound for \(\lvert D\rvert\)
and we do not claim that the given examples have minimal absolute discriminants.



\subsubsection{Selection Rule for quadratic base fields}
\label{sss:SelRuleCocl1}

Let \(K=\mathbb{Q}(\sqrt{D})\) be a quadratic number field with discriminant \(D\)
and \(p\)-class group \(\mathrm{Cl}_p(K)\) of type \((p,p)\),
where \(p\ge 3\) denotes an odd prime.
Then the \(p+1\) unramified cyclic extension fields \((L_1,\ldots,L_{p+1})\) of \(K\)
of relative prime degree \(p\) have
dihedral absolute Galois groups \(\mathrm{Gal}(L_i\vert K)\) of degree \(2p\),
according to
\cite[Prop. 4.1]{Ma1}.

\begin{theorem}
\label{thm:SelRuleCocl1}

Let \(G=\mathrm{Gal}(\mathrm{F}_p^2(K)\vert K)\)
be the second \(p\)-class group of \(K\).
If \(G\in\mathcal{G}(p,1)\), then \(K\) must be real quadratic, \(D>0\), and,
with respect to the natural order
of the maximal subgroups of \(G\),
the family of \(p\)-class numbers of the non-Galois subfields \(K_i\) of \(L_i\) is given by

\begin{eqnarray*}
(\mathrm{h}_p(K_1),\mathrm{h}_p(K_2),\ldots,\mathrm{h}_p(K_{p+1})) =
(p^{\frac{\mathrm{cl}(G)-\mathrm{dp}(G)}{2}},\overbrace{p,\ldots,p}^{p\text{ times}}),
\end{eqnarray*}

\noindent
where depth and defect \(k\) of \(G\) are related via the first component of the TKT \(\varkappa(K)\) by

\[\mathrm{dp}(G)=
\begin{cases}
k,   & \text{ if } \varkappa(1)=0, \\
k+1, & \text{ if } \varkappa(1)\ne 0.
\end{cases}
\]

\noindent
Consequently, the order \(\lvert M_i\rvert=p^i\) of the branch root \(M_i\) of \(G\in\mathcal{B}(M_i)\)
on the unique coclass tree of \(\mathcal{G}(p,1)\) with mainline \((M_j)_{j\ge 2}\)
must have \textit{odd} exponent
\[i=n-\mathrm{dp}(G)=\mathrm{cl}(G)+1-\mathrm{dp}(G)\equiv 1\pmod{2}.\]

\end{theorem}

\begin{remark}
Whereas \(\mathrm{h}_p(K_2),\ldots,\mathrm{h}_p(K_{p+1})\) do not give any information,
the distinguished \(p\)-class number \(\mathrm{h}_p(K_1)\)
enforces the congruence \(\mathrm{cl}(G)-\mathrm{dp}(G)\equiv 0\pmod{2}\).
\end{remark}

\begin{proof}
The statement is a compact version of
\cite[Thm. 4.1]{Ma1},
expressed by the depth \(\mathrm{dp}(G)\),
and thus more closely related to the
position of \(G\) on the coclass graph \(\mathcal{G}(p,1)\)
and to the transfer kernel type \(\varkappa(G)\) of \(G\),
using Corollary
\ref{cor:DpthCocl1}.
\end{proof}

\begin{theorem}
\label{thm:TKTpCocl1}
Table
\ref{tab:TKTpCocl1}
gives the transfer kernel type (TKT) \(\varkappa(G)\) of all
non-isolated metabelian vertices \(G\) on the coclass graph \(\mathcal{G}(p,1)\), for odd \(p\ge 3\).
The \(p\)-groups \(G\) are identified by their Blackburn-Miech invariants
\(\lvert G\rvert=p^n\)
and \(a,w,z\) as exponents in the relations
(\ref{eqn:PwrRelCocl1})
and
(\ref{eqn:CmtRelCocl1}).
The graph information gives the depth \(\mathrm{dp}(G)\) and the location of each \(p\)-group \(G\)
with respect to the unique coclass tree \(\mathcal{T}(C_p\times C_p)\)
of \(\mathcal{G}(p,1)\).\\
The mainline, consisting of the \(p\)-groups \(G_0^n(0,0)\)
including the abelian root \(C_p\times C_p=G_0^2(0,0)\),
and all groups of positive defect \(k\ge 1\)
are characterized
by the total transfer \(\varkappa(1)=0\) to the distinguished maximal subgroup \(H_1=\chi_2(G)\).
\end{theorem}

\renewcommand{\arraystretch}{1.2}
\begin{table}[ht]
\caption{\(\varkappa(G)\), \(\nu(G)\) in dependence on non-isolated \(G\in\mathcal{G}(p,1)\) for \(p\ge 3\)}
\label{tab:TKTpCocl1}
\begin{center}
\begin{tabular}{|c|cc|cccc|cc|ccc|}
\hline
 \multicolumn{7}{|c|}{\(p\)-Group \(G_a^n(z,w)\) of Coclass \(1\)}                             & \multicolumn{2}{|c|}{graph information} &  \multicolumn{3}{|c|}{transfer kernels}                         \\
 \multicolumn{7}{|c|}{\downbracefill}                                                          & \multicolumn{2}{|c|}{\downbracefill}    &  \multicolumn{3}{|c|}{\downbracefill}                           \\
 \(G\)             &\(\mathrm{cl}(G)\) &     \(n\) & \(a\)     & \(z\)     & \(w\) & \(k\)     & \(\mathrm{dp}(G)\) & tree position      &  TKT & \(\varkappa(G)\)                                  & \(\nu(G)\) \\
\hline
 \(C_p\times C_p\) &             \(1\) &     \(2\) &           &           &       & \(0\)     &              \(0\) & root               &  a.1 & \((\overbrace{0\ldots 0}^{p+1\text{ times}})\) & \(p+1\) \\
 \(G_0^3(0,1)\)    &             \(2\) &     \(3\) & \(0\)     & \(0\)     & \(1\) & \(0\)     &              \(1\) & pre-period         &  A.1 & \((\overbrace{1\ldots 1}^{p+1\text{ times}})\) & \(0\)   \\
\hline
 \(G_0^n(0,0)\)    &         \(\ge 2\) & \(\ge 3\) & \(0\)     & \(0\)     & \(0\) & \(0\)     &              \(0\) & mainline           &  a.1 & \((\overbrace{0\ldots 0}^{p+1\text{ times}})\) & \(p+1\) \\
\hline
 \(G_0^n(0,1)\)    &         \(\ge 3\) & \(\ge 4\) & \(0\)     & \(0\)     & \(1\) & \(0\)     &              \(1\) & periodic sequences &  a.2 & \((1\overbrace{0\ldots 0}^{p\text{ times}})\)  & \(p\)   \\
 \(G_0^n(z,0)\)    &         \(\ge 3\) & \(\ge 4\) & \(0\)     & \(\ne 0\) & \(0\) & \(0\)     &              \(1\) & periodic sequences &  a.3 & \((2\overbrace{0\ldots 0}^{p\text{ times}})\)  & \(p\)   \\
 \(G_a^n(z,w)\)    &         \(\ge 4\) & \(\ge 5\) & \(\ne 0\) & \( \)     & \( \) & \(\ge 1\) &          \(\ge 1\) & periodic sequences &  a.1 & \((\overbrace{0\ldots 0}^{p+1\text{ times}})\) & \(p+1\) \\
\hline
\end{tabular}
\end{center}
\end{table}

\begin{proof}
See
\cite[Thm. 2.5, Tab. 1]{Ma2}.
\end{proof}



\subsubsection{The complete coclass graph \(\mathcal{G}(3,1)\)}
\label{sss:3Cocl1}

This section and the following sections \S\S\
\ref{sss:TKTFromCoarseTTTCocl1}--\ref{sss:7Cocl1}
will show, that
second \(p\)-class groups \(\mathrm{G}_p^2(K)\)
of real quadratic fields \(K=\mathbb{Q}(\sqrt{D})\), \(D>0\), with \(\mathrm{Cl}_p(K)\simeq(p,p)\)
are only distributed on odd branches of the metabelian skeleton
of \(\mathcal{G}(p,1)\), for an odd prime \(p\ge 3\),
in contrast to the complete population of the coclass graph \(\mathcal{G}(2,1)\).
The effect is due to the number theoretic selection rule in Theorem
\ref{thm:SelRuleCocl1}.
The quantitative distribution for \(p\in\lbrace 3,5,7\rbrace\)
reveals a dominant population of ground states
and decreasing frequency of hits of excited states.

\begin{theorem}
\label{thm:3Cocl1}
The diagram in Figure
\ref{fig:Distr3Cocl1}
visualizes the complete coclass graph \(\mathcal{G}(3,1)\)
up to order \(3^8=6\,561\).
It is periodic with length \(2\).
The period consists of
branches \(\mathcal{B}_j\) with \(4\le j\le 5\),
whereas branches \(\mathcal{B}_j\) with \(2\le j\le 3\)
are irregular and form the pre-period.
\end{theorem}



\begin{figure}[ht]
\caption{Root \(C_3\times C_3\) and branches \(\mathcal{B}_j\), \(2\le j\le 7\), of coclass graph \(\mathcal{G}(3,1)\)}
\label{fig:Distr3Cocl1}

\setlength{\unitlength}{1cm}
\begin{picture}(16,15)(-8,-14)

\put(-8,0.5){\makebox(0,0)[cb]{Order \(3^n\)}}
\put(-8,0){\line(0,-1){12}}
\multiput(-8.1,0)(0,-2){7}{\line(1,0){0.2}}
\put(-8.2,0){\makebox(0,0)[rc]{\(9\)}}
\put(-7.8,0){\makebox(0,0)[lc]{\(3^2\)}}
\put(-8.2,-2){\makebox(0,0)[rc]{\(27\)}}
\put(-7.8,-2){\makebox(0,0)[lc]{\(3^3\)}}
\put(-8.2,-4){\makebox(0,0)[rc]{\(81\)}}
\put(-7.8,-4){\makebox(0,0)[lc]{\(3^4\)}}
\put(-8.2,-6){\makebox(0,0)[rc]{\(243\)}}
\put(-7.8,-6){\makebox(0,0)[lc]{\(3^5\)}}
\put(-8.2,-8){\makebox(0,0)[rc]{\(729\)}}
\put(-7.8,-8){\makebox(0,0)[lc]{\(3^6\)}}
\put(-8.2,-10){\makebox(0,0)[rc]{\(2\,187\)}}
\put(-7.8,-10){\makebox(0,0)[lc]{\(3^7\)}}
\put(-8.2,-12){\makebox(0,0)[rc]{\(6\,561\)}}
\put(-7.8,-12){\makebox(0,0)[lc]{\(3^8\)}}

\put(-0.1,-0.1){\framebox(0.2,0.2){}}
\put(-2.1,-0.1){\framebox(0.2,0.2){}}
\multiput(0,-2)(0,-2){5}{\circle*{0.2}}
\multiput(-2,-2)(0,-2){6}{\circle*{0.2}}
\multiput(-4,-4)(0,-2){5}{\circle*{0.2}}
\multiput(-6,-4)(0,-4){3}{\circle*{0.2}}
\multiput(2,-6)(0,-2){4}{\circle*{0.1}}
\multiput(4,-6)(0,-2){4}{\circle*{0.1}}
\multiput(6,-6)(0,-2){4}{\circle*{0.1}}

\multiput(0,0)(0,-2){5}{\line(0,-1){2}}
\multiput(0,0)(0,-2){6}{\line(-1,-1){2}}
\multiput(0,-2)(0,-2){5}{\line(-2,-1){4}}
\multiput(0,-2)(0,-4){3}{\line(-3,-1){6}}
\multiput(0,-4)(0,-2){4}{\line(1,-1){2}}
\multiput(0,-4)(0,-2){4}{\line(2,-1){4}}
\multiput(0,-4)(0,-2){4}{\line(3,-1){6}}

\put(0,-10){\vector(0,-1){2}}
\put(0.2,-11.5){\makebox(0,0)[lc]{main}}
\put(0.2,-12){\makebox(0,0)[lc]{line}}
\put(-0.2,-12.5){\makebox(0,0)[rc]{\(\mathcal{T}(C_3\times C_3)\)}}

\put(0.2,0){\makebox(0,0)[lt]{\(C_3\times C_3\)}}
\put(-2.2,0){\makebox(0,0)[rt]{\(C_9\)}}
\put(0.2,-2){\makebox(0,0)[lt]{\(G^3_0(0,0)\)}}
\put(-2.2,-2){\makebox(0,0)[rt]{\(G^3_0(0,1)\)}}
\put(-4,-4.5){\makebox(0,0)[cc]{\(\mathrm{Syl}_3A_9\)}}

\put(-3,0){\makebox(0,0)[cc]{\(\Phi_1\)}}
\put(-2.1,0.1){\makebox(0,0)[rb]{\(\langle 1\rangle\)}}
\put(-0.1,0.1){\makebox(0,0)[rb]{\(\langle 2\rangle\)}}

\put(2,-2){\makebox(0,0)[cc]{\(\Phi_2\)}}
\put(-2.1,-1.9){\makebox(0,0)[rb]{\(\langle 4\rangle\)}}
\put(-0.1,-1.9){\makebox(0,0)[rb]{\(\langle 3\rangle\)}}

\put(2,-4){\makebox(0,0)[cc]{\(\Phi_3\)}}
\put(-6.1,-3.9){\makebox(0,0)[rb]{\(\langle 8\rangle\)}}
\put(-4.1,-3.9){\makebox(0,0)[rb]{\(\langle 7\rangle\)}}
\put(-2.1,-3.9){\makebox(0,0)[rb]{\(\langle 10\rangle\)}}
\put(-0.1,-3.9){\makebox(0,0)[rb]{\(\langle 9\rangle\)}}

\put(-4,-6.5){\makebox(0,0)[cc]{\(\Phi_9\)}}
\put(-4.1,-5.9){\makebox(0,0)[rb]{\(\langle 25\rangle\)}}
\put(-2.1,-5.9){\makebox(0,0)[rb]{\(\langle 27\rangle\)}}
\put(-0.1,-5.9){\makebox(0,0)[rb]{\(\langle 26\rangle\)}}

\put(4,-6.5){\makebox(0,0)[cc]{\(\Phi_{10}\)}}
\put(2.1,-5.9){\makebox(0,0)[lb]{\(\langle 28\rangle\)}}
\put(4.1,-5.9){\makebox(0,0)[lb]{\(\langle 30\rangle\)}}
\put(6.1,-5.9){\makebox(0,0)[lb]{\(\langle 29\rangle\)}}

\put(-4,-8.5){\makebox(0,0)[cc]{\(\Phi_{35}\)}}
\put(-6.1,-7.9){\makebox(0,0)[rb]{\(\langle 98\rangle\)}}
\put(-4.1,-7.9){\makebox(0,0)[rb]{\(\langle 97\rangle\)}}
\put(-2.1,-7.9){\makebox(0,0)[rb]{\(\langle 96\rangle\)}}
\put(-0.1,-7.9){\makebox(0,0)[rb]{\(\langle 95\rangle\)}}

\put(4,-8.5){\makebox(0,0)[cc]{\(\Phi_{36}\)}}
\put(2.1,-7.9){\makebox(0,0)[lb]{\(\langle 100\rangle\)}}
\put(4.1,-7.9){\makebox(0,0)[lb]{\(\langle 99\rangle\)}}
\put(6.1,-7.9){\makebox(0,0)[lb]{\(\langle 101\rangle\)}}

\put(-0.1,-9.9){\makebox(0,0)[rb]{\(\langle 386\rangle\)}}

\put(0,-13){\makebox(0,0)[cc]{\(G^n_0(0,0)\)}}
\put(-2,-13){\makebox(0,0)[cc]{\(G^n_0(0,1)\)}}
\put(-4,-13){\makebox(0,0)[cc]{\(G^n_0(1,0)\)}}
\put(-6,-13){\makebox(0,0)[cc]{\(G^n_0(-1,0)\)}}
\put(2,-13){\makebox(0,0)[cc]{\(G^n_1(0,-1)\)}}
\put(4,-13){\makebox(0,0)[cc]{\(G^n_1(0,0)\)}}
\put(6,-13){\makebox(0,0)[cc]{\(G^n_1(0,1)\)}}

\put(2.5,0){\makebox(0,0)[cc]{\textbf{TKT:}}}
\put(3.5,0){\makebox(0,0)[cc]{a.1}}
\put(3.5,-0.5){\makebox(0,0)[cc]{\((0000)\)}}
\put(1.8,-0.7){\framebox(2.9,1){}}
\put(-6,-2){\makebox(0,0)[cc]{\textbf{TKT:}}}
\put(-5,-2){\makebox(0,0)[cc]{A.1}}
\put(-5,-2.5){\makebox(0,0)[cc]{\((1111)\)}}
\put(-6.7,-2.7){\framebox(2.9,1){}}

\put(-8,-14){\makebox(0,0)[cc]{\textbf{TKT:}}}
\put(0,-14){\makebox(0,0)[cc]{a.1}}
\put(-2,-14){\makebox(0,0)[cc]{a.2}}
\put(-4,-14){\makebox(0,0)[cc]{a.3}}
\put(-6,-14){\makebox(0,0)[cc]{a.3}}
\put(2,-14){\makebox(0,0)[cc]{a.1}}
\put(4,-14){\makebox(0,0)[cc]{a.1}}
\put(6,-14){\makebox(0,0)[cc]{a.1}}
\put(0,-14.5){\makebox(0,0)[cc]{\((0000)\)}}
\put(-2,-14.5){\makebox(0,0)[cc]{\((1000)\)}}
\put(-4,-14.5){\makebox(0,0)[cc]{\((2000)\)}}
\put(-6,-14.5){\makebox(0,0)[cc]{\((2000)\)}}
\put(2,-14.5){\makebox(0,0)[cc]{\((0000)\)}}
\put(4,-14.5){\makebox(0,0)[cc]{\((0000)\)}}
\put(6,-14.5){\makebox(0,0)[cc]{\((0000)\)}}
\put(-8.7,-14.7){\framebox(15.4,1){}}

\put(-4,-4){\oval(1.5,2)}
\put(-4,-4){\oval(5.6,1.5)}
\put(-4,-8){\oval(5.6,1.5)}
\multiput(4,-8)(0,-4){2}{\oval(5.9,1.5)}
\put(-6,-5.1){\makebox(0,0)[cc]{\underbar{\textbf{2083}}}}
\put(-4,-5.2){\makebox(0,0)[cc]{\underbar{\textbf{697}}}}
\put(-6,-9.1){\makebox(0,0)[cc]{\underbar{\textbf{72}}}}
\put(7.4,-8.3){\makebox(0,0)[cc]{\underbar{\textbf{147}}}}
\put(7.4,-12.3){\makebox(0,0)[cc]{\underbar{\textbf{1}}}}

\end{picture}

\end{figure}
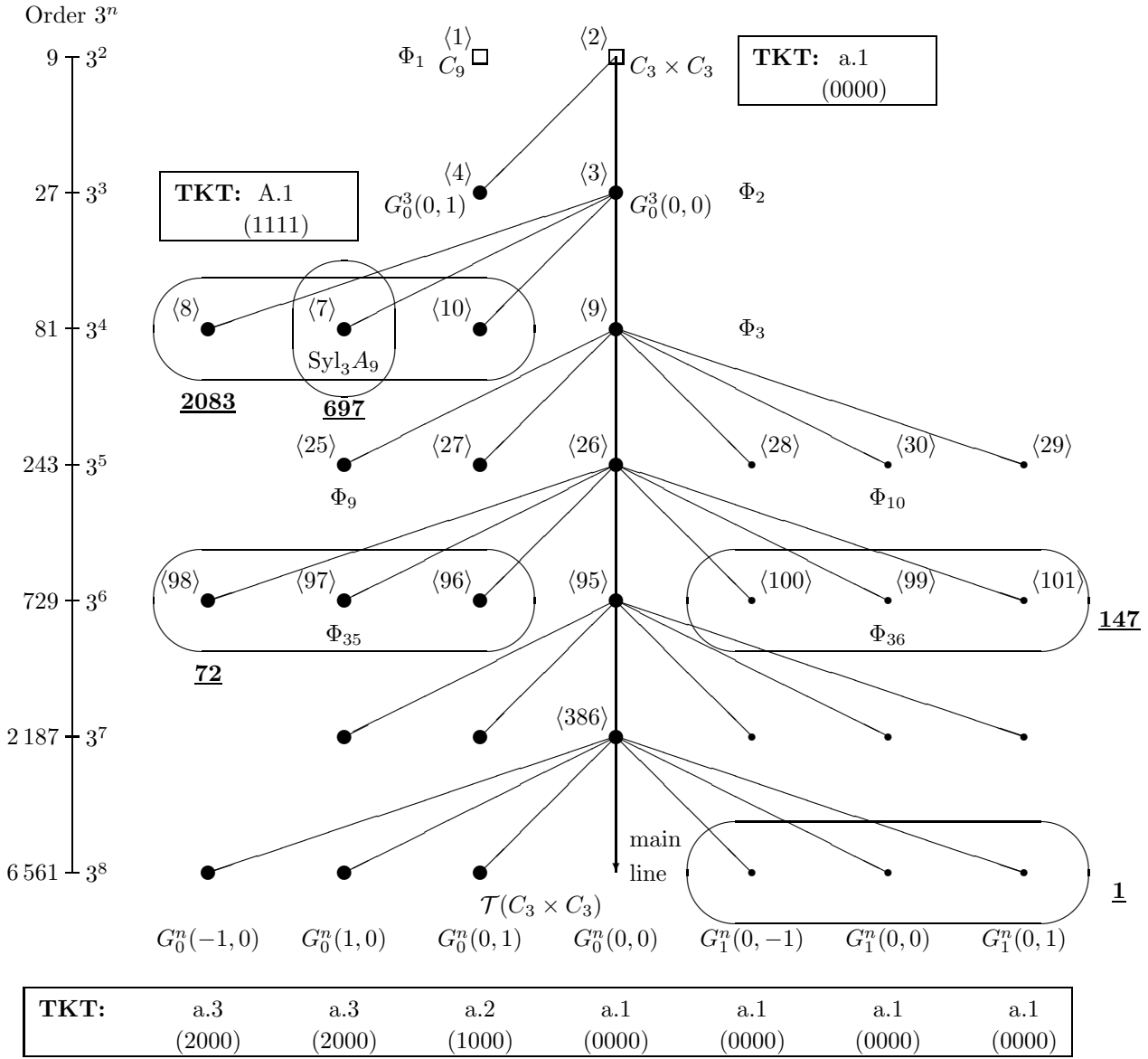

\begin{proof}
The top of \(\mathcal{G}(3,1)\) consists of two abelian groups of order \(3^2\),
the isolated cyclic group \(C_9\) and the bicyclic root \(C_3\times C_3\)
of the unique coclass tree \(\mathcal{T}(C_3\times C_3)\).\\
Immediate descendants of the root
are the two well-known extra special groups,
the capable mainline group \(G_0^3(0,0)\) of exponent \(3\)
and the terminal group \(G_0^3(0,1)\) of exponent \(9\),
both of order \(3^3\).

Blackburn's results
\cite{Bl1}
on counting metabelian \(p\)-groups of maximal class
and order \(p^n\) with \(n\ge 4\)
can now be applied to the special case \(p=3\),
which is entirely metabelian \cite[p. 26, Thm. 6]{Bl2}.\\
We start with
metabelian groups containing an abelian maximal subgroup,
which are characterized by defect \(k=0\).
They consist of the capable mainline group \(G_0^n(0,0)\),
the terminal group \(G_0^n(0,1)\) and
\((n-2,p-1)\) terminal groups of the form \(G_0^n(z,0)\).
Specialization of
\cite[p. 88, Thm. 4.3]{Bl1}
to \(p=3\) in dependence on \(n\) yields their number
\[
2+(n-2,p-1)=
\begin{cases}
2+1=3 & \text{ for } 5\le n\equiv 1\pmod{2}, \\
2+2=4 & \text{ for } 4\le n\equiv 0\pmod{2}.
\end{cases}
\]
Further,
the number of metabelian groups with defect \(k=1\),
which are terminal and of the form \(G_1^n(z,w)\) with \(a=(a(n-1))=(1)\),
is given, independently from \(n\ge 5\), by \(3\)
\cite[p. 88, Thm. 4.2]{Bl1}.
\end{proof}

\noindent
Vertices \(G\) of coclass graph \(\mathcal{G}(3,1)\) in Figure
\ref{fig:Distr3Cocl1}
are classified according to their defect \(k(G)\)
by using different symbols:

\begin{enumerate}
\item
large contour squares \(\square\) denote abelian groups,
\item
big full discs {\Large \(\bullet\)} denote metabelian groups with abelian maximal subgroup and \(k=0\),
\item
small full discs {\footnotesize \(\bullet\)} denote metabelian groups with defect \(k=1\).
\end{enumerate}

\noindent
The actual distribution of the \(2576\) second \(3\)-class groups \(G_3^2(K)\)
of real quadratic number fields \(K=\mathbb{Q}(\sqrt{D})\) of type \((3,3)\)
with discriminant \(0<D<10^7\) is represented by
underlined boldface counters of hits of vertices surrounded by the adjacent oval.
See
\cite[\S\ 6, Tab. 2]{Ma1}
and
\cite[\S\ 6, Tab. 11]{Ma3}.
The results verify the selection rule, Theorem
\ref{thm:SelRuleCocl1},
for groups \(G_3^2(\mathbb{Q}(\sqrt{D}))\), \(D>0\), 
and underpin the \textit{weak leaf conjecture}
\ref{cnj:WeakLeafCnj}
that
mainline vertices are forbidden for second \(3\)-class groups of quadratic fields.
A remarkably different behavior is revealed by certain biquadratic fields in Figure
\ref{fig:WimanBlackburn1}
of section \S\
\ref{ss:NewRsltESR}.

\begin{conjecture}
\label{cnj:WeakLeafCnj}
A vertex \(G\) on the metabelian skeleton \(\mathcal{M}(p,r)\)
of a coclass graph \(\mathcal{G}(p,r)\), with an odd prime \(p\ge 3\) and \(r\ge 1\),
cannot be realized
as second \(p\)-class group \(G_p^2(K)\) of a quadratic field \(K=\mathbb{Q}(\sqrt{D})\),
if it possesses a metabelian immediate descendant \(H\)
having the same transfer kernel type \(\varkappa(H)=\varkappa(G)\)
and a higher defect of commutativity \(k(H)>k(G)\).
\end{conjecture}



\subsubsection{Separating TKTs on \(\mathcal{G}(p,1)\) via first TTT}
\label{sss:TKTFromCoarseTTTCocl1}

For increasing odd primes \(p\ge 5\),
the structure \(\tau(1)\) of the \(p\)-class group \(\mathrm{Cl}_p(L_1)\)
of the distinguished first unramified extension \(L_1\) of degree \(p\) of an arbitrary base field \(K\)
with second \(p\)-class group \(G=\mathrm{G}_p^2(K)\) of coclass \(\mathrm{cc}(G)=1\)
admits the separation of more and more excited states
of the TKTs \(\mathrm{a}.2\), with fixed point \(\varkappa(1)=1\),
and \(\mathrm{a}.3\), without fixed point \(\varkappa(1)\in\lbrace 2,\ldots,p+1\rbrace\).
Further, the order of the exceptional \(p\)-group \(G\simeq\mathrm{Syl}_p A_{p^2}\)
becomes increasingly larger.

\begin{theorem}
\label{thm:TKTFromCoarseTTTCocl1}
Let \(p\ge 3\) be an odd prime
and \(G\in\mathcal{G}(p,1)\) a \(p\)-group
of order \(\lvert G\rvert=p^n\), \(n\ge 4\), depth \(\mathrm{dp}(G)=1\), and defect \(k(G)=0\).

\begin{enumerate}
\item
The exceptional case of TKT \(\mathrm{a}.3^\ast\),
having an elementary abelian first TTT \(\tau(1)\)
of elevated \(p\)-rank \(\mathrm{r}_p=p\),
occurs if and only if \(n=p+1\) and \(G\simeq G_0^{p+1}(1,0)=\mathrm{Syl}_p A_{p^2}\)
\item
The regular cases of the TKTs \(\mathrm{a}.2\) and \(\mathrm{a}.3\),
having a first TTT \(\tau(1)\) of usual \(p\)-rank \(\mathrm{r}_p\le p-1\),
can be separated by the structure \(\tau(1)\) of the distinguished first maximal subgroup \(H_1=\chi_2(G)\)
if and only if \(n\le p\). In this case,
\begin{enumerate}
\item
\(G\) is of TKT \(\mathrm{a}.2\) if and only if
\(\tau(1)=(\overbrace{p\ldots,p}^{n-1\text{ times}})\)
is elementary abelian of rank \(\mathrm{r}_p=n-1\le p-1\),
\item
\(G\) is of TKT \(\mathrm{a}.3\) if and only if
\(\tau(1)=(p^2,\overbrace{p\ldots,p}^{n-3\text{ times}})\)
is of rank \(\mathrm{r}_p=n-2\le p-2\), neither nearly homocyclic nor elementary abelian.
\end{enumerate}
\end{enumerate}

\end{theorem}

\begin{proof}
All groups \(G\in\mathcal{G}(p,1)\) of order \(\lvert G\rvert\ge p^n\), \(n\ge 4\),
depth \(\mathrm{dp}(G)=1\), and defect \(k(G)=0\)
are metabelian and contain the abelian distinguished maximal subgroup \(A=H_1=\chi_2(G)\),
having \(A^\prime=1\).
Thus, all statements are a consequence of
\cite[Thm. 7, p. 11]{HeSm}, where
\(G\) is of TKT \(\mathrm{a}.2\) if and only if \(G\simeq G_0^n(0,1)\), and
\(G\) is of TKT \(\mathrm{a}.3\) if and only if \(G\simeq G_0^n(z,0)\), \(z\notin\lbrace 0,1\rbrace\).
\end{proof}

Table
\ref{tbl:TKTFromCoarseTTTCocl1}
displays the possibilities for the first TTT \(\tau(1)\)
in dependence on the ground state (GS) and excited states (ES) of TKTs,
as stated in Theorem
\ref{thm:TKTFromCoarseTTTCocl1}
for the smallest odd primes \(p\in\lbrace 3,5,7\rbrace\).
Here, we assume a quadratic base field \(K=\mathbb{Q}(\sqrt{D})\),
taking into account the selection rule, Theorem
\ref{thm:SelRuleCocl1},
for odd branches.

\renewcommand{\arraystretch}{1.0}

\begin{table}[ht]
\caption{Separating TKT \(\mathrm{a}.2\), \(\mathrm{a}.3\) and \(\mathrm{a}.3*\) on \(\mathcal{G}(p,1)\)}
\label{tbl:TKTFromCoarseTTTCocl1}
\begin{center}
\begin{tabular}{|c|c|c|c|c|c|}
\hline
       &       &                                & \multicolumn{3}{|c|}{First TTT, \(\tau(1)\), for TKT} \\
 \(p\) & state & branch of                      & \multicolumn{3}{|c|}{\downbracefill} \\
       &       & \(\mathcal{T}(C_p\times C_p)\) & \(\mathrm{a}.2\) & \(\mathrm{a}.3\) & \(\mathrm{a}.3^\ast\) \\
\hline
 \(3\) & GS    & \(\mathcal{B}_3\)              & \((3^2,3)\)         & \((3^2,3)\)         & \((3,3,3)\) \\
       & ES 1  & \(\mathcal{B}_5\)              & \((3^3,3^2)\)       & \((3^3,3^2)\)       & --- \\
       & ES 2  & \(\mathcal{B}_7\)              & \((3^4,3^3)\)       & \((3^4,3^3)\)       & --- \\
\hline
 \(5\) & GS    & \(\mathcal{B}_3\)              & \((5,5,5)\)         & \((5^2,5)\)         & --- \\
       & ES 1  & \(\mathcal{B}_5\)              & \((5^2,5,5,5)\)     & \((5^2,5,5,5)\)     & \((5,5,5,5,5)\) \\
       & ES 2  & \(\mathcal{B}_7\)              & \((5^2,5^2,5^2,5)\) & \((5^2,5^2,5^2,5)\) & --- \\
\hline
 \(7\) & GS    & \(\mathcal{B}_3\)              & \((7,7,7)\)         & \((7^2,7)\)         & --- \\
       & ES 1  & \(\mathcal{B}_5\)              & \((7,7,7,7,7)\)     & \((7^2,7,7,7)\)     & --- \\
       & ES 2  & \(\mathcal{B}_7\)              & \((7^2,7,7,7,7,7)\) & \((7^2,7,7,7,7,7)\) & \((7,7,7,7,7,7,7)\) \\
\hline
\end{tabular}
\end{center}
\end{table}



\subsubsection{Metabelian \(5\)-groups \(G\) of coclass \(\mathrm{cc}(G)=1\)}
\label{sss:5Cocl1}

\begin{theorem}
\label{thm:5Cocl1}
The diagram in Figure
\ref{fig:Distr5Cocl1}
visualizes the metabelian skeleton \(\mathcal{M}(5,1)\)
of coclass graph \(\mathcal{G}(5,1)\)
up to order \(5^{11}=48\,828\,125\).
This subgraph of \(\mathcal{G}(5,1)\)
is periodic with length \(4\).
The period consists of
the branches \(\mathcal{B}_j\) with \(5\le j\le 8\),
whereas the branches \(\mathcal{B}_j\) with \(2\le j\le 4\)
are irregular and form the pre-period.
\end{theorem}



\begin{figure}[ht]
\caption{All metabelian \(5\)-groups of order up to \(5^{11}\) on \(\mathcal{G}(5,1)\)}
\label{fig:Distr5Cocl1}

\setlength{\unitlength}{0.8cm}
\begin{picture}(18,21)(-8,-20)

\put(-8,0.5){\makebox(0,0)[cb]{Order \(5^n\)}}
\put(-8,0){\line(0,-1){18}}
\multiput(-8.1,0)(0,-2){10}{\line(1,0){0.2}}
\put(-8.2,0){\makebox(0,0)[rc]{\(25\)}}
\put(-7.8,0){\makebox(0,0)[lc]{\(5^2\)}}
\put(-8.2,-2){\makebox(0,0)[rc]{\(125\)}}
\put(-7.8,-2){\makebox(0,0)[lc]{\(5^3\)}}
\put(-8.2,-4){\makebox(0,0)[rc]{\(625\)}}
\put(-7.8,-4){\makebox(0,0)[lc]{\(5^4\)}}
\put(-8.2,-6){\makebox(0,0)[rc]{\(3\,125\)}}
\put(-7.8,-6){\makebox(0,0)[lc]{\(5^5\)}}
\put(-8.2,-8){\makebox(0,0)[rc]{\(15\,625\)}}
\put(-7.8,-8){\makebox(0,0)[lc]{\(5^6\)}}
\put(-8.2,-10){\makebox(0,0)[rc]{\(78\,125\)}}
\put(-7.8,-10){\makebox(0,0)[lc]{\(5^7\)}}
\put(-8.2,-12){\makebox(0,0)[rc]{\(390\,625\)}}
\put(-7.8,-12){\makebox(0,0)[lc]{\(5^8\)}}
\put(-8.2,-14){\makebox(0,0)[rc]{\(1\,953\,125\)}}
\put(-7.8,-14){\makebox(0,0)[lc]{\(5^9\)}}
\put(-8.2,-16){\makebox(0,0)[rc]{\(9\,765\,625\)}}
\put(-7.8,-16){\makebox(0,0)[lc]{\(5^{10}\)}}
\put(-8.2,-18){\makebox(0,0)[rc]{\(48\,828\,125\)}}
\put(-7.8,-18){\makebox(0,0)[lc]{\(5^{11}\)}}

\put(-0.1,-0.1){\framebox(0.2,0.2){}}
\put(-2.1,-0.1){\framebox(0.2,0.2){}}
\multiput(0,-2)(0,-2){8}{\circle*{0.25}}
\multiput(-2,-2)(0,-2){9}{\circle*{0.25}}
\multiput(-4,-4)(0,-2){8}{\circle*{0.25}}
\multiput(2,-6)(0,-2){7}{\circle{0.2}}
\multiput(4,-6)(0,-2){7}{\circle{0.2}}
\multiput(5,-8)(0,-2){6}{\circle*{0.15}}
\multiput(7,-8)(0,-2){6}{\circle*{0.15}}
\multiput(6,-10)(0,-2){5}{\circle*{0.15}}
\multiput(9,-10)(0,-2){5}{\circle{0.1}}
\multiput(8,-12)(0,-2){4}{\circle{0.1}}

\multiput(0,0)(0,-2){8}{\line(0,-1){2}}
\multiput(0,0)(0,-2){9}{\line(-1,-1){2}}
\multiput(0,-2)(0,-2){8}{\line(-2,-1){4}}
\multiput(0,-4)(0,-2){7}{\line(1,-1){2}}
\multiput(0,-4)(0,-2){7}{\line(2,-1){4}}
\multiput(4,-6)(0,-2){6}{\line(1,-2){1}}
\multiput(4,-6)(0,-2){6}{\line(3,-2){3}}
\multiput(4,-8)(0,-2){5}{\line(1,-1){2}}
\multiput(7,-8)(0,-2){5}{\line(1,-1){2}}
\multiput(6,-10)(0,-2){4}{\line(1,-1){2}}

\put(0,-16){\vector(0,-1){2}}
\put(0.2,-17.5){\makebox(0,0)[lc]{main}}
\put(0.2,-18){\makebox(0,0)[lc]{line}}
\put(-0.2,-18.5){\makebox(0,0)[rc]{\(\mathcal{T}(C_5\times C_5)\)}}

\multiput(-4.3,-4)(0,-8){2}{\makebox(0,0)[rc]{\(2*\)}}
\multiput(-4.3,-8)(0,-8){2}{\makebox(0,0)[rc]{\(4*\)}}
\multiput(2.2,-6)(0,-4){4}{\makebox(0,0)[lc]{\(*5\)}}
\multiput(2.2,-8)(0,-8){2}{\makebox(0,0)[lc]{\(*6\)}}
\put(2.2,-12){\makebox(0,0)[lc]{\(*4\)}}
\put(5.1,-8){\makebox(0,0)[lc]{\(*9\)}}
\multiput(5.1,-10)(0,-4){3}{\makebox(0,0)[lc]{\(*10\)}}
\put(5.1,-12){\makebox(0,0)[lc]{\(*11\)}}
\put(5.1,-16){\makebox(0,0)[lc]{\(*13\)}}
\put(9.1,-10){\makebox(0,0)[lc]{\(*38\)}}
\multiput(8.1,-12)(0,-2){4}{\makebox(0,0)[lc]{\(*25\)}}
\multiput(9.1,-12)(0,-4){2}{\makebox(0,0)[lc]{\(*15\)}}
\multiput(9.1,-14)(0,-4){2}{\makebox(0,0)[lc]{\(*13\)}}

\put(-6,-7){\makebox(0,0)[cc]{\(\mathcal{B}_5\)}}
\put(-6,-9){\makebox(0,0)[cc]{\(\mathcal{B}_6\)}}
\put(-6,-11){\makebox(0,0)[cc]{\(\mathcal{B}_7\)}}
\put(-6,-13){\makebox(0,0)[cc]{\(\mathcal{B}_8\)}}

\put(-2.1,0.1){\makebox(0,0)[rb]{\(\langle 1\rangle\)}}
\put(-0.1,0.1){\makebox(0,0)[rb]{\(\langle 2\rangle\)}}

\put(-0.2,-1.3){\makebox(0,0)[rc]{\(\langle 3\rangle\)}}
\put(-2,-1.3){\makebox(0,0)[cc]{\(\langle 4\rangle\)}}
\put(-0.2,-3.3){\makebox(0,0)[rc]{\(\langle 7\rangle\)}}
\put(-2,-3.3){\makebox(0,0)[cc]{\(\langle 8\rangle\)}}
\put(-4,-3.3){\makebox(0,0)[cc]{\(\langle 9..10\rangle\)}}
\put(-0.2,-5.3){\makebox(0,0)[rc]{\(\langle 30\rangle\)}}
\put(-2,-5.3){\makebox(0,0)[cc]{\(\langle 31\rangle\)}}
\put(-4,-5.3){\makebox(0,0)[cc]{\(\langle 32\rangle\)}}
\put(2,-5.3){\makebox(0,0)[cc]{\(\langle 34\rangle\)}}
\put(4,-5.3){\makebox(0,0)[cc]{\(\langle 33\rangle\)}}
\put(-0.2,-7.3){\makebox(0,0)[rc]{\(\langle 630\rangle\)}}
\put(-2,-7.3){\makebox(0,0)[cc]{\(\langle 635\rangle\)}}
\put(-4,-7.3){\makebox(0,0)[cc]{\(\langle 631..634\rangle\)}}
\put(2,-7.3){\makebox(0,0)[cc]{\(\langle 637..642\rangle\)}}
\put(4,-7.3){\makebox(0,0)[cc]{\(\langle 636\rangle\)}}
\put(5,-7.3){\makebox(0,0)[cc]{\(\langle 652\rangle\)}}
\put(7,-7.3){\makebox(0,0)[cc]{\(\langle 651\rangle\)}}

\put(0.5,0.5){\makebox(0,0)[cb]{\(C_5\times C_5\)}}
\put(-1.5,0.5){\makebox(0,0)[cb]{\(C_{25}\)}}
\put(-3,-2){\makebox(0,0)[cc]{\(G^3_0(0,1)\)}}
\put(0,-19){\makebox(0,0)[ct]{\(G^n_0(0,0)\)}}
\put(-2,-19){\makebox(0,0)[ct]{\(G^n_0(0,1)\)}}
\put(-4,-19){\makebox(0,0)[ct]{\(G^n_0(z,0)\)}}
\put(3,-19){\makebox(0,0)[ct]{\(G^n_1(z,w)\)}}
\put(6,-19){\makebox(0,0)[ct]{\(G^n_a(z,w)\)}}
\put(8.5,-19){\makebox(0,0)[ct]{\(G^n_a(z,w)\)}}

\put(-6,-20){\makebox(0,0)[cc]{\textbf{TKT:}}}
\put(0,-20){\makebox(0,0)[cc]{a.1}}
\put(-2,-20){\makebox(0,0)[cc]{a.2}}
\put(-4,-20){\makebox(0,0)[cc]{a.3}}
\put(3,-20){\makebox(0,0)[cc]{a.1}}
\put(6,-20){\makebox(0,0)[cc]{a.1}}
\put(8.5,-20){\makebox(0,0)[cc]{a.1}}
\put(0,-20.5){\makebox(0,0)[cc]{\((000000)\)}}
\put(-2,-20.5){\makebox(0,0)[cc]{\((100000)\)}}
\put(-4,-20.5){\makebox(0,0)[cc]{\((200000)\)}}
\put(3,-20.5){\makebox(0,0)[cc]{\((000000)\)}}
\put(6,-20.5){\makebox(0,0)[cc]{\((000000)\)}}
\put(8.5,-20.5){\makebox(0,0)[cc]{\((000000)\)}}
\put(-7,-20.7){\framebox(16.5,1){}}
\put(-6,-2){\makebox(0,0)[cc]{\textbf{TKT:}}}
\put(-5.2,-2){\makebox(0,0)[lc]{A.1}}
\put(-5.4,-2.5){\makebox(0,0)[cc]{\((111111)\)}}
\put(-7,-2.7){\framebox(3,1){}}
\put(2,0){\makebox(0,0)[cc]{\textbf{TKT:}}}
\put(3,0){\makebox(0,0)[lc]{a.1}}
\put(2.8,-0.5){\makebox(0,0)[cc]{\((000000)\)}}
\put(1,-0.7){\framebox(3,1){}}

\multiput(0,-2)(0,-4){4}{\oval(1.5,1)}
\put(-4.2,-4){\oval(1.5,1)}
\put(-2,-4){\oval(1.5,1)}
\multiput(-3.1,-8)(0,-4){3}{\oval(3.5,1)}
\multiput(3,-8)(0,-4){3}{\oval(2.6,1)}
\multiput(6,-10)(0,-4){3}{\oval(2.6,1)}
\multiput(8.8,-12)(0,-4){2}{\oval(2.2,1)}

\put(-4.7,-4.7){\makebox(0,0)[cc]{\underbar{\textbf{292}}}}
\put(-2.7,-4.7){\makebox(0,0)[cc]{\underbar{\textbf{55}}}}
\put(-4.7,-8.7){\makebox(0,0)[cc]{\underbar{\textbf{3}}}}
\put(2.7,-8.7){\makebox(0,0)[cc]{\underbar{\textbf{13}}}}

\end{picture}

\end{figure}
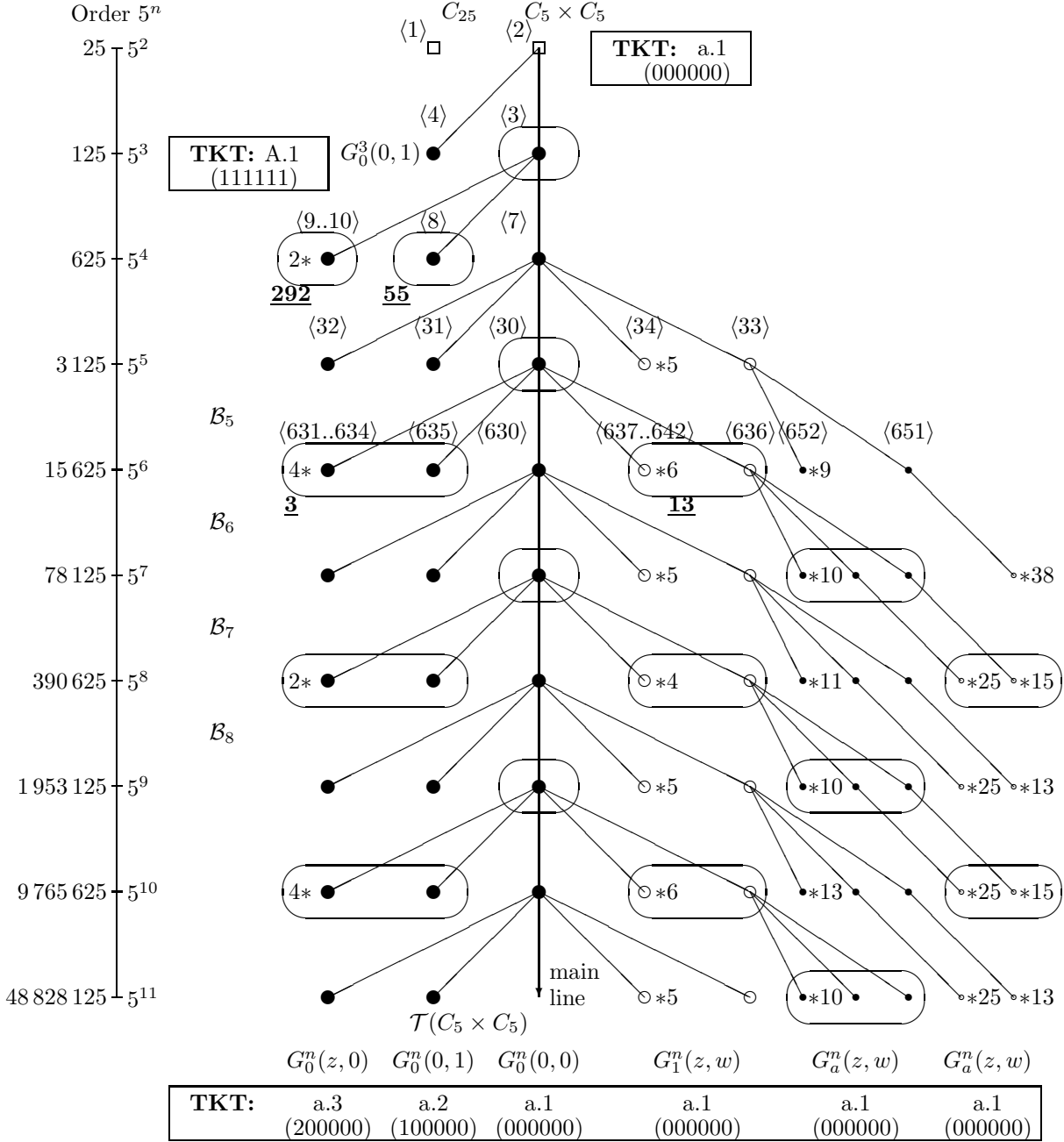

\noindent
Vertices of coclass graph \(\mathcal{G}(5,1)\) in Figure
\ref{fig:Distr5Cocl1}
are classified according to their defect \(k\)
by using different symbols:

\begin{enumerate}
\item
large contour squares \(\square\) represent abelian groups,
\item
big full discs {\Large \(\bullet\)} represent metabelian groups with defect \(k=0\),
\item
big contour circles {\Large \(\circ\)} represent metabelian groups with \(k=1\),
\item
small full discs {\footnotesize \(\bullet\)} represent metabelian groups with \(k=2\),
\item
small contour circles {\scriptsize \(\circ\)} represent metabelian groups with \(k=3\).
\end{enumerate}

\noindent
The symbol \(n\ast\) adjacent to a vertex denotes the multiplicity of a batch
of \(n\) immediate descendants sharing a common parent.
The selection rule, Theorem
\ref{thm:SelRuleCocl1},
for second \(5\)-class groups \(G_5^2(K)\)
of real quadratic number fields \(K=\mathbb{Q}(\sqrt{D})\), \(D>0\), 
is indicated by ovals surrounding admissible vertices.

\noindent
The actual distribution of the \(377\) second \(5\)-class groups \(G_5^2(\mathbb{Q}(\sqrt{D}))\)
with discriminant \(0<D\le 26\,695\,193\),
discussed in section
\ref{sss:StatScnd5ClgpCocl1},
is represented by
underlined boldface counters of hits of vertices in the adjacent oval.
The \(13\) cases of TKT \(\mathrm{a}.1\), \(\varkappa=(000000)\),
underpin the weak leaf conjecture
\ref{cnj:WeakLeafCnj}.

\begin{proof}
\(\mathcal{G}(5,1)\) starts with two abelian groups of order \(5^2\),
the isolated cyclic group \(C_{25}\) and the bicyclic root \(C_5\times C_5\)
of the unique coclass tree \(\mathcal{T}(C_5\times C_5)\).\\
As immediate descendants of the root,
\(\mathcal{G}(5,1)\) contains the two well-known extra special groups,
the capable mainline group \(G_0^3(0,0)\) of exponent \(5\)
and the terminal group \(G_0^3(0,1)\) of exponent \(25\),
both of order \(5^3\).

Now we use Blackburn's results
\cite{Bl1}
on counting metabelian \(p\)-groups of maximal class
and order \(p^n\) with \(n\ge 4\),
for the special case \(p=5\).\\
First, we consider
the metabelian groups containing an abelian maximal subgroup,
which are characterized by the defect \(k=0\).
They consist of the capable mainline group \(G_0^n(0,0)\),
the terminal group \(G_0^n(0,1)\) and
\((n-2,p-1)\) terminal groups of the form \(G_0^n(z,0)\).
Specialization of
\cite[p. 88, Thm. 4.3]{Bl1}
for \(p=5\) in dependence on \(n\) yields their number
\[
2+(n-2,p-1)=
\begin{cases}
2+1=3 & \text{ for } 5\le n\equiv 1\pmod{2}, \\
2+2=4 & \text{ for } 4\le n\equiv 0\pmod{4}, \\
2+4=6 & \text{ for } 6\le n\equiv 2\pmod{4}. \\
\end{cases}
\]
Next,
the number of metabelian groups with defect \(k=1\),
which contain exactly one capable group \(G_1^n(0,0)\) with \(a=(a(n-1))=(1)\) for every \(n\ge 5\),
is given by
\cite[p. 88, Thm. 4.2]{Bl1}:
\[
1+(2n-6,p-1)+(n-2,p-1)=
\begin{cases}
1+4+1=6 & \text{ for } 5\le n\equiv 1\pmod{2}, \\
1+2+2=5 & \text{ for } 8\le n\equiv 0\pmod{4}, \\
1+2+4=7 & \text{ for } 6\le n\equiv 2\pmod{4}. \\
\end{cases}
\]
Finally,
the number of metabelian groups with defect \(k=2\),
containing exactly two capable groups \(G_a^n(0,0)\) with \(a=(1,\pm 1)\) for every \(n\ge 7\)
\cite[\S\ 3, ramification level]{LgMk3},
but only one capable group \(G_{(1,-1)}^n(0,0)\) for \(n=6\),
is given by
\cite[p. 88, Thm. 4.1]{Bl1}:

\begin{eqnarray*}
p+(2n-7,p-1)+(n-2,p-1)  &=& 5+1+4=10 \text{ for } n=6,\\
2p+(2n-7,p-1)+(n-2,p-1) &=&
\begin{cases}
10+1+1=12 & \text{ for } 7\le n\equiv 1\pmod{2}, \\
10+1+2=13 & \text{ for } 8\le n\equiv 0\pmod{4}, \\
10+1+4=15 & \text{ for } 10\le n\equiv 2\pmod{4}. \\
\end{cases}
\end{eqnarray*}

\noindent
Since Blackburn restricts his investigations to defects \(k\le 2\),
we need a supplementary count of metabelian groups with defect \(k=3\).
The results for the head of the
\(4\) virtually periodic branches \(\mathcal{B}_j\) with \(14\le j\le 17\)
given by Dietrich, Eick, Feichtenschlager
\cite[Fig. 7--10, p. 57--60]{DEF}
and by Dietrich
\cite[Fig. 4--5, p. 1086]{Dt2}
are accumulated counts of metabelian and non-metabelian groups,
whereas the collars and tails entirely consist of non-metabelian groups.
According to private communications by H. Dietrich,
one of the two capable groups \(G_{(1,\pm 1)}^{n-1}(0,0)\) at depth two
has always \(25\) metabelian descendants, which are all terminal,
independently from \(n\ge 8\),
and the other
has \(15\), resp. \(13\), metabelian descendants, which are all terminal,
for even \(n\ge 8\), resp. odd \(n\ge 9\).
The count of metabelian groups with defect \(k=3\) is also given by Miech
\cite[Thm. 6--7, p. 336--337]{Mi}.
\end{proof}



\subsubsection{Distribution of \(\mathrm{G}_5^2(K)\) on \(\mathcal{G}(5,1)\)}
\label{sss:StatScnd5ClgpCocl1}

In Table
\ref{tbl:RealQuad5x5},
we list the \(8\) variants of second \(5\)-class groups \(\mathrm{G}_5^2(K)\)
for the \(377\) real quadratic fields \(K=\mathbb{Q}(\sqrt{D})\)
of type \((5,5)\) with discriminant \(0<D\le 26\,695\,193\),
mainly on the coclass graph \(\mathcal{G}(5,1)\),
but modestly also on \(\mathcal{G}(5,2)\).
\(\tau(0)\) denotes the \(5\)-class group of \(\mathrm{F}_5^1(K)\).
Schur \(\sigma\)-groups are starred.

\renewcommand{\arraystretch}{1.1}

\begin{table}[ht]
\caption{\(8\) variants of \(G=\mathrm{G}_5^2(K)\) for \(377\) \(K=\mathbb{Q}(\sqrt{D})\), \(0<D\le 26\,695\,193\)}
\label{tbl:RealQuad5x5}
\begin{center}
\begin{tabular}{|r||c||c|c||c|c||c|c|}
\hline
           \(D\) & \(\tau(K);\tau(0)\)               & Type & \(\varkappa(K)\) & \(G\)               & \(\mathrm{cc}(G)\) &   \(\#\) &   \(\%\) \\
\hline
    \(244\,641\) & \((5,5^2),(5,5)^5;(5,5)\)         &  a.3 & \((2,0^5)\)      & \(\langle 625,9\vert 10\rangle\) & \(1\) &  \(292\) & \(77.5\) \\
 \(1\,167\,541\) & \((5,5,5),(5,5)^5;(5,5)\)         &  a.2 & \((1,0^5)\)      & \(\langle 625,8\rangle\)         & \(1\) &   \(55\) & \(14.6\) \\
\hline
 \(1\,129\,841\) & \((5,5,5,5),(5,5)^5;(5,5,5,5)\)   &  a.1 & \((0^6)\)        & \(\langle 15625,637\ldots 642\rangle\)& \(1\) &   \(13\) &  \(3.4\) \\
 \(3\,812\,377\) & \((5,5,5,5^2),(5,5)^5;(5,5,5,5)\) &a.2,3\(\uparrow\)&\((?,0^5)\)&\(\langle 15625,631\ldots 635\rangle\)&\(1\)&\(3\)&  \(  \) \\
\hline
 \(4\,954\,652\) & \((5,5^2)^6;(5,5,5)\)             &      & \((B^6)\)        & \(\langle 3125,9\vert 10\vert 12\rangle\)& \(2\) &    \(7\) &   \(1.9\) \\
\(10\,486\,805\) & \((5,5,5)^2,(5,5^2)^4;(5,5,5)\)   &      & \((A^2,B^4)\)    & \(\langle 3125,7\vert 11\rangle\)& \(2\) &    \(2\) &   \(  \) \\
\(18\,070\,649\) & \((5,5,5),(5,5^2)^5;(5,5,5)\)     &      & \((A,B^5)\)      & \(\langle 3125,8\vert 13\rangle*\)& \(2\) &    \(1\) &   \(  \) \\
\hline
 \(7\,306\,081\) & \((5,5,5,5),(5,5,5),(5,5^2)^4\)   &      & \((0,A,B^4)\)    & \(\langle 3125,4\rangle\)        & \(2\) &    \(4\) &    \(1\) \\
\hline
\end{tabular}
\end{center}
\end{table}



\noindent
There occur
\(13\) cases \((3.4\%)\) of TKT \(\mathrm{a}.1\), \(\varkappa=(000000)\),
starting with \(D=1\,129\,841\),
\(3\) cases of the first excited state of TKT \(\mathrm{a}.2\uparrow\), \(\varkappa=(100000)\), or \(\mathrm{a}.3\uparrow\), \(\varkappa=(200000)\), for \(D\in\lbrace 3\,812\,377,19\,621\,905,21\,281\,673\rbrace\),
and \(55\) cases \((14.6\%)\) of the ground state of TKT \(\mathrm{a}.2\), \(\varkappa=(100000)\),
starting with \(D=1\,167\,541\).
The remaining \(292\) cases \((77.5\%)\) of the ground state of TKT \(\mathrm{a}.3\), \(\varkappa=(200000)\),
starting with \(D=244\,641\), are clearly dominating.
The TKTs were identified by means of Theorem
\ref{thm:TKTFromCoarseTTTCocl1},
taking into account the selection rule for quadratic base fields as given in Table
\ref{tbl:TKTFromCoarseTTTCocl1}.
The distribution of the corresponding second \(5\)-class groups \(\mathrm{G}_5^2(K)\)
on the coclass graph \(\mathcal{G}(5,1)\), resp. \(\mathcal{G}(5,2)\), is shown in Figure
\ref{fig:Distr5Cocl1},
resp.
\ref{fig:Typ55Cocl2}.
See also section \S\
\ref{sss:StatScnd5ClgpCocl2}.



\subsubsection{Metabelian \(7\)-groups \(G\) of coclass \(\mathrm{cc}(G)=1\)}
\label{sss:7Cocl1}



\begin{figure}[ht]
\caption{All \(7\)-groups of order up to \(7^{4}\) on \(\mathcal{G}(7,1)\)}
\label{fig:Distr7Cocl1}

\setlength{\unitlength}{0.8cm}
\begin{picture}(18,7.5)(-8,-6.5)

\put(-8,0.5){\makebox(0,0)[cb]{Order \(7^n\)}}
\put(-8,0){\line(0,-1){4}}
\multiput(-8.1,0)(0,-2){3}{\line(1,0){0.2}}
\put(-8.2,0){\makebox(0,0)[rc]{\(49\)}}
\put(-7.8,0){\makebox(0,0)[lc]{\(7^2\)}}
\put(-8.2,-2){\makebox(0,0)[rc]{\(343\)}}
\put(-7.8,-2){\makebox(0,0)[lc]{\(7^3\)}}
\put(-8.2,-4){\makebox(0,0)[rc]{\(2401\)}}
\put(-7.8,-4){\makebox(0,0)[lc]{\(7^4\)}}

\put(-0.1,-0.1){\framebox(0.2,0.2){}}
\put(-2.1,-0.1){\framebox(0.2,0.2){}}
\put(0,-2){\circle*{0.25}}
\multiput(-2,-2)(0,-2){2}{\circle*{0.25}}
\put(-4,-4){\circle*{0.25}}

\multiput(0,0)(0,-2){2}{\line(0,-1){2}}
\multiput(0,0)(0,-2){2}{\line(-1,-1){2}}
\put(0,-2){\line(-2,-1){4}}

\put(0,-2){\vector(0,-1){2}}
\put(0.2,-3.5){\makebox(0,0)[lc]{main}}
\put(0.2,-4){\makebox(0,0)[lc]{line}}
\put(1.8,-4.5){\makebox(0,0)[rc]{\(\mathcal{T}(C_7\times C_7)\)}}

\put(-4.3,-4){\makebox(0,0)[rc]{\(2*\)}}

\put(-2.1,0.1){\makebox(0,0)[rb]{\(\langle 1\rangle\)}}
\put(-0.1,0.1){\makebox(0,0)[rb]{\(\langle 2\rangle\)}}

\put(-0.2,-1.3){\makebox(0,0)[rc]{\(\langle 3\rangle\)}}
\put(-2,-1.3){\makebox(0,0)[cc]{\(\langle 4\rangle\)}}
\put(-2,-3.3){\makebox(0,0)[cc]{\(\langle 8\rangle\)}}
\put(-4,-3.3){\makebox(0,0)[cc]{\(\langle 9..10\rangle\)}}

\put(0.5,0.5){\makebox(0,0)[cb]{\(C_7\times C_7\)}}
\put(-1.5,0.5){\makebox(0,0)[cb]{\(C_{49}\)}}
\put(-3,-2){\makebox(0,0)[cc]{\(G^3_0(0,1)\)}}
\put(0,-5){\makebox(0,0)[ct]{\(G^n_0(0,0)\)}}
\put(-2,-5){\makebox(0,0)[ct]{\(G^n_0(0,1)\)}}
\put(-4,-5){\makebox(0,0)[ct]{\(G^n_0(z,0)\)}}
\put(6.5,-5){\makebox(0,0)[ct]{\(G^n_a(z,w)\)}}

\put(-6,-6){\makebox(0,0)[cc]{\textbf{TKT:}}}
\put(0,-6){\makebox(0,0)[cc]{a.1}}
\put(-2,-6){\makebox(0,0)[cc]{a.2}}
\put(-4,-6){\makebox(0,0)[cc]{a.3}}
\put(6.5,-6){\makebox(0,0)[cc]{a.1}}
\put(0,-6.5){\makebox(0,0)[cc]{\((0^8)\)}}
\put(-2,-6.5){\makebox(0,0)[cc]{\((1,0^7)\)}}
\put(-4,-6.5){\makebox(0,0)[cc]{\((2,0^7)\)}}
\put(6.5,-6.5){\makebox(0,0)[cc]{\((0^8)\)}}
\put(-7,-6.7){\framebox(14.5,1){}}
\put(-6,-2){\makebox(0,0)[cc]{\textbf{TKT:}}}
\put(-5.2,-2){\makebox(0,0)[lc]{A.1}}
\put(-5.4,-2.5){\makebox(0,0)[cc]{\((1^8)\)}}
\put(-7,-2.7){\framebox(3,1){}}
\put(2,0){\makebox(0,0)[cc]{\textbf{TKT:}}}
\put(3,0){\makebox(0,0)[lc]{a.1}}
\put(2.8,-0.5){\makebox(0,0)[cc]{\((0^8)\)}}
\put(1,-0.7){\framebox(3,1){}}

\put(0,-2){\oval(1.5,1)}
\put(-4.2,-4){\oval(1.5,1)}
\put(-2,-4){\oval(1.5,1)}

\put(-4.7,-4.7){\makebox(0,0)[cc]{\underbar{\textbf{13}}}}
\put(-2.7,-4.7){\makebox(0,0)[cc]{\underbar{\textbf{3}}}}

\end{picture}

\end{figure}

\noindent
Figure
\ref{fig:Distr7Cocl1}
visualizes the lowest range of the distribution of
second \(7\)-class groups \(\mathrm{G}_7^2(K)\)
for the \(17\) real quadratic fields \(K=\mathbb{Q}(\sqrt{D})\)
of type \((7,7)\) with discriminant \(0<D<10^7\)
on the coclass graph \(\mathcal{G}(7,1)\).
With the aid of MAGMA
\cite{MAGMA},
we found \(13\) cases, \(76\%\), of TKT
\(\mathrm{a}.3\), \(\varkappa=(20000000)\),
starting with \(D=1\,633\,285\),
and \(3\) occurrences, \(18\%\), of TKT
\(\mathrm{a}.2\), \(\varkappa=(10000000)\), 
for \(D\in\lbrace 2\,713\,121,6\,872\,024,9\,659\,661\rbrace\).
These two TKTs can be separated by means of Theorem
\ref{thm:TKTFromCoarseTTTCocl1}.
There were no cases of excited states,
but for the single discriminant \(D=6\,986\,985\),
\(\mathrm{G}_7^2(K)\) is a top vertex of \(\mathcal{G}(7,2)\)
without total \(7\)-principalization
and of Taussky's coarse transfer kernel type \(\kappa=(BBBBBBBB)\)
\cite{Ta2}.
Table
\ref{tbl:RealQuad7x7}
shows the corresponding TTT 
\(\tau(K)=(\mathrm{Cl}_7(L_i))_{1\le i\le 8}\)
using power notation for repetitions
and including \(\tau(0)=\mathrm{Cl}_7(\mathrm{F}_7^1(K))\),
separated by a semicolon.
\(7\)-groups \(G^n_a(z,w)\) of positive defect \(k\ge 1\) appear in higher branches
and are invisible in Figure
\ref{fig:Distr7Cocl1}.

\renewcommand{\arraystretch}{1.1}

\begin{table}[ht]
\caption{\(3\) variants of \(G=\mathrm{G}_7^2(K)\) for \(17\) \(K=\mathbb{Q}(\sqrt{D})\), \(0<D<10^7\)}
\label{tbl:RealQuad7x7}
\begin{center}
\begin{tabular}{|r||c||c|c||c|c||c|c|}
\hline
           \(D\) & \(\tau(K);\tau(0)\)       & Type & \(\varkappa(K)\) & \(G\)          & \(\mathrm{cc}(G)\) & \(\#\) & \(\%\) \\
\hline
 \(1\,633\,285\) & \((7,7^2),(7,7)^7;(7,7)\) &  a.3 & \((2,0^7)\) &\(\langle 2401,9\vert 10\rangle\)& \(1\) & \(13\) & \(76\) \\
 \(2\,713\,121\) & \((7,7,7),(7,7)^7;(7,7)\) &  a.2 & \((1,0^7)\)      & \(\langle 2401,8\rangle\)   & \(1\) &  \(3\) & \(18\) \\
\hline
 \(6\,986\,985\) & \((7,7^2)^8;(7,7,7)\)     &      & \((B^8)\)        & \(\langle 16807,10\vert 14\ldots 16\rangle\) & \(2\) &  \(1\) &  \(6\) \\
\hline
\end{tabular}
\end{center}
\end{table}

\noindent
In the next section, we proceed to \(p\)-groups \(G\) of coclass \(\mathrm{cc}(G)\ge 2\).



\subsection{Metabelian \(3\)-groups \(G\) of coclass \(\mathrm{cc}(G)\ge 2\) with \(G/G^\prime\simeq (3,3)\)}
\label{ss:MtabTyp33CoclGe2}

\subsubsection{Non-CF groups}
\label{sss:NonCF}

In contrast to CF groups of coclass \(1\),
metabelian \(3\)-groups \(G\) of coclass \(\mathrm{cc}(G)\ge 2\)
with abelianization \(G/G^\prime\) of type \((3,3)\)
must have at least one bicyclic factor \(\gamma_3(G)/\gamma_4(G)\)
\cite{Ne1},
and are therefore called \textit{non-CF groups}.
They are characterized by an isomorphism \textit{invariant} \(e=e(G)\), defined by
\(e+1=\min\lbrace 3\le j\le m\mid 1\le\lvert\gamma_j(G)/\gamma_{j+1}(G)\rvert\le 3\rbrace\).
This invariant \(2\le e\le m-1\) indicates
the first cyclic factor \(\gamma_{e+1}(G)/\gamma_{e+2}(G)\) of the lower central series of \(G\),
except \(\gamma_2(G)/\gamma_3(G)\), which is always cyclic.
We can calculate \(e\)
from order \(\lvert G\rvert=3^n\)
and nilpotency class \(\mathrm{cl}(G)=m-1\), resp. index \(m\) of nilpotency, of \(G\)
by the formula \(e=n-m+2\).
Since the coclass of \(G\) is given by \(\mathrm{cc}(G)=n-\mathrm{cl}(G)=n-m+1\),
we have the relation \(e=\mathrm{cc}(G)+1\).
CF groups are characterized by \(e=2\)
and non-CF groups by \(e\ge 3\).



\subsubsection{Bipolarization and defect}
\label{sss:BiplrzDfct}

For a group \(G\) of coclass \(\mathrm{cc}(G)\ge 2\)
we need a generalization of the group \(\chi_2(G)\).
Denoting by \(m\) the index of nilpotency of \(G\),
we let \(\chi_j(G)\) with \(2\le j\le m-1\)
be the centralizers
of two-step factor groups \(\gamma_j(G)/\gamma_{j+2}(G)\)
of the lower central series, that is,
the biggest subgroups of \(G\) with the property
\(\lbrack\chi_j(G),\gamma_j(G)\rbrack\le\gamma_{j+2}(G)\).
They form an ascending chain of characteristic subgroups of \(G\),
\(\gamma_2(G)\le\chi_2(G)\le\ldots\le\chi_{m-2}(G)<\chi_{m-1}(G)=G\),
which contain the commutator subgroup \(\gamma_2(G)\),
and \(\chi_j(G)\) coincides with \(G\) if and only if \(j\ge m-1\).
We characterize the smallest \textit{two-step centralizer}
different from the commutator subgroup
by an isomorphism \textit{invariant}
\(s=s(G)=\min\lbrace 2\le j\le m-1\mid\chi_j(G)>\gamma_2(G)\rbrace\).
Again, CF groups are characterized by \(s=2\)
and non-CF groups by \(s\ge 3\).

Now we can generalize the \textit{defect of commutativity} \(k=k(G)\)
to any metabelian \(3\)-group \(G\) with \(G/G^\prime\) of type \((3,3)\)
by defining \(0\le k\le 1\) such that
\(\lbrack\chi_s(G),\gamma_e(G)\rbrack=\gamma_{m-k}(G)\).

The following assumptions 
for a metabelian \(3\)-group \(G\) of coclass \(\mathrm{cc}(G)\ge 2\)
with abelianization \(G/\gamma_2(G)\) of type \((3,3)\)
can always be satisfied, according to Nebelung
\cite[Thm. 3.1.11, p. 57, and Thm. 3.4.5, p. 94]{Ne1}.

Let \(G\) be a metabelian \(3\)-group of coclass \(\mathrm{cc}(G)\ge 2\)
with abelianisation \(G/\gamma_2(G)\) of type \((3,3)\).
Assume that \(G\) has order \(\lvert G\rvert=3^n\),
class \(\mathrm{cl}(G)=m-1\), and invariant \(e=n-m+2\ge 3\),
where \(4\le m<n\le 2m-3\).
Let generators of \(G=\langle x,y\rangle\) be selected such that
the bicyclic factor \(\gamma_3(G)/\gamma_4(G)\)
is generated by their third powers,
\(\gamma_3(G)=\langle y^3,x^3,\gamma_4(G)\rangle\),
and that
\(x\in G\setminus\chi_s(G)\), if \(s<m-1\),
and \(y\in\chi_s(G)\setminus\gamma_2(G)\).
This causes a \textit{bipolarization}
among the four maximal subgroups \(H_1,\ldots,H_4\) of \(G\),
which will be standardized in Definition
\ref{dfn:NatOrdCoclGe2}.



\subsubsection{Parametrized presentation}
\label{sss:PrmPres2}

Let the main commutator of \(G\) be declared by
\(s_2=t_2=\lbrack y,x\rbrack\in\gamma_2(G)\)
and higher commutators recursively by
\(s_j=\lbrack s_{j-1},x\rbrack\), \(t_j=\lbrack t_{j-1},y\rbrack\in\gamma_j(G)\)
for \(j\ge 3\).
Starting with the powers \(\sigma_3=y^3\), \(\tau_3=x^3\in\gamma_3(G)\),
which generate \(\gamma_3(G)\) modulo \(\gamma_4(G)\),
let
\(\sigma_j=\lbrack\sigma_{j-1},x\rbrack\), \(\tau_j=\lbrack\tau_{j-1},y\rbrack\in\gamma_j(G)\)
for \(j\ge 4\).
Nilpotency of \(G\) is expressed by
\(\sigma_{m-1}=1\) and \(\tau_{e+2}=1\).
According to Nebelung
\cite{Ne1},
the group \(G\) satisfies the following relations
with certain exponents \(-1\le\alpha,\beta,\gamma,\delta,\rho\le 1\) as parameters.

\begin{equation}
\label{eqn:PwrCmtPres}
s_2^3=\sigma_4\sigma_{m-1}^{-\rho\beta}\tau_4^{-1},\quad
\ s_3\sigma_3\sigma_4=\sigma_{m-2}^{\rho\beta}\sigma_{m-1}^\gamma\tau_e^\delta,\quad
\ t_3^{-1}\tau_3\tau_4=\sigma_{m-2}^{\rho\delta}\sigma_{m-1}^\alpha\tau_e^\beta,\quad
\ \tau_{e+1}=\sigma_{m-1}^{-\rho}.
\end{equation}

\noindent
By \(G_\rho^{m,n}(\alpha,\beta,\gamma,\delta)\) we denote 
the representative of an isomorphism class of
metabelian \(3\)-groups \(G\), having \(G/G^\prime\) of type \((3,3)\),
of coclass \(\mathrm{cc}(G)=n-m+1\ge 2\),
class \(\mathrm{cl}(G)=m-1\), and order \(\lvert G\rvert=3^n\),
which satisfies the relations
(\ref{eqn:PwrCmtPres})
with a fixed system of exponents
\((\alpha,\beta,\gamma,\delta,\rho)\).
We have \(\rho=0\) if and only if \(k=0\).



\subsubsection{Two distinguished maximal subgroups}
\label{sss:DstgMaxSbgp2}

The maximal normal subgroups \(H_1,\ldots,H_4\) of \(G\)
contain the commutator subgroup \(G^\prime\)
as a normal subgroup of index \(3\)
and are thus of the shape \(H_i=\langle g_i,G^\prime\rangle\)
with suitable generators \(g_i\).
We want to arrange them in a fixed order.

\begin{definition}
\label{dfn:NatOrdCoclGe2}
The \textit{bipolarization} or \textit{natural order}
of the maximal subgroups \((H_i)_{1\le i\le 4}\) of \(G\)
is given by the \textit{distinguished first generator} \(g_1=y\in\chi_s(G)\),
the \textit{distinguished second generator} \(g_2=x\notin\chi_s(G)\),
both satisfying \(y^3,x^3\in\gamma_3(G)\setminus\gamma_4(G)\),
and the other generators \(g_i=xy^{i-2}\notin\chi_s(G)\) for \(3\le i\le 4\),
provided that \(s<m-1\).
Then, in particular, \(\chi_s(G)=H_1=\langle y,G^\prime\rangle\).
\end{definition}



\subsubsection{Parents of core and interface groups}
\label{sss:PrntCoclGe2}

\begin{definition}
\label{dfn:CoreAndIntf}
For an arbitrary prime \(p\),
let \(G\) be a finite \(p\)-group of nilpotency class \(c=\mathrm{cl}(G)\).
We call \(G\) a \textit{core group}, resp. an \textit{interface group},
if its last lower central \(\gamma_c(G)\)
is of order \(p^d\) with \(d=1\), resp. \(d\ge 2\).
\end{definition}

\noindent
If \(G\) is of order \(p^n\),
the last lower central quotient \(Q=G/\gamma_c(G)\) of \(G\)
is of order \(\lvert Q\rvert=\lvert G\rvert/p^d=p^{n-d}\) and of class \(\mathrm{cl}(Q)=\mathrm{cl}(G)-1\).
Therefore, the coclass of \(Q\) is given by
\[{cc}(Q)=n-d-\mathrm{cl}(Q)=n-d-\mathrm{cl}(G)+1=\mathrm{cc}(G)-(d-1).\]
Consequently,
the last lower central quotient \(Q\) of a core group \(G\)
is of the same coclass as \(G\), whereas
the last lower central quotient \(Q\) of an interface group \(G\)
is of lower coclass than \(G\).
Obviously, a CF group must necessarily be a core group.

Now we apply these new concepts to the case \(p=3\)
and investigate the parent \(\pi(G)\) of a metabelian \(3\)-group \(G\)
with \(G/G^\prime\) of type \((3,3)\).
Since the invariant \(e=e(G)=\mathrm{cc}(G)+1\) indicates
the first cyclic quotient \(\gamma_{e+1}(G)/\gamma_{e+2}(G)\),
\(G\) is an interface group if and only if \(e=\mathrm{cl}(G)=m-1\),
where \(m\) denotes the index of nilpotency of \(G\).
This maximal possible value of \(e\) enforces a special relation
between order \(\lvert G\rvert=3^n\) and class \(\mathrm{cl}(G)=m-1\) of \(G\),
\[n=e+m-2=m-1+m-2=2m-3.\]



Together with group counts in Nebelung's theorem
\cite[p. 178, Thm. 5.1.16]{Ne1},
the following two theorems describe the structure
of the \textit{metabelian skeleton} of those subgraphs
of the coclass graphs \(\mathcal{G}(3,r)\), \(r\ge 2\),
which are formed by isomorphism classes of metabelian \(3\)-groups \(G\)
having abelianization \(G/G^\prime\simeq (3,3)\).
This restriction concerns both,
the coclass trees \(\mathcal{T}\)
and the sporadic part \(\mathcal{G}_0(3,r)\)
of each coclass graph \(\mathcal{G}(3,r)\).
We distinguish \textit{core} groups and \textit{interface} groups
and begin with the former.

\begin{theorem}
\label{thm:CoclGe2Core}
Let \(G\) be a metabelian \(3\)-group of coclass \(r=\mathrm{cc}(G)\ge 2\)
with \(G/G^\prime\simeq (3,3)\),
such that \(G\simeq G_\rho^{m,n}(\alpha,\beta,\gamma,\delta)\in\mathcal{G}(3,r)\)
with parameters \(-1\le\alpha,\beta,\gamma,\delta,\rho\le 1\),
that is, \(G\) is of
order \(\lvert G\rvert=3^n\), class \(\mathrm{cl}(G)=m-1\), \(4\le m<n\le 2m-3\),
coclass \(2\le r=n-m+1\le m-2\), and invariant \(3\le e=n-m+2\le m-1\).
Assume additionally that \(G\) is a core group
with cyclic last lower central \(\gamma_{m-1}\) of order \(3\),
thus having \(5\le m<n\le 2m-4\) and \(e\le m-2\).
Then the parent \(\pi(G)\) of \(G\) is generally given by
\(\pi(G)\simeq G_0^{m-1,n-1}(\rho\delta,\beta,\rho\beta,\delta)\in\mathcal{G}(3,r)\),
and in particular,
\[\pi(G)\simeq
\begin{cases}
G_0^{m-1,n-1}(0,\beta,0,\delta), \text{ if } \rho=0, \\
G_0^{m-1,n-1}(\rho\delta,\beta,\rho\beta,\delta), \text{ if } \rho=\pm 1,\ (\beta,\delta)\ne (0,0), \\
G_0^{m-1,n-1}(0,0,0,0), \text{ if } \rho=\pm 1,\ (\beta,\delta)=(0,0).
\end{cases}
\]
\end{theorem}



\begin{remark}
The various cases of Theorem
\ref{thm:CoclGe2Core}
can be described as follows.

\begin{enumerate}
\item
If \(G\) is a group with parameter \(\rho=0\), or equivalently with defect \(k=0\), then
the parent \(\pi(G)\simeq G_0^{m-1,n-1}(0,\beta,0,\delta)\) is a mainline group
on one of the coclass trees,
since these groups are characterized uniquely by \(\alpha=0\), \(\gamma=0\), \(\rho=0\)
\cite{Ne2}.
A summary is given in Table
\ref{tbl:MainLines}.
\item
However, if \(G\) is a group with \(\rho=\pm 1\) or equivalently \(k=1\), then
the parent \(\pi(G)\) has defect \(\tilde{k}=0=k-1\)
but lies outside of any mainline,
either on a branch of a coclass tree \(\mathcal{T}\) or on the sporadic part \(\mathcal{G}_0(3,r)\).
\item
The only exception is the very special case that \(G\) with \(\rho=\pm 1\)
has the parameters \(\beta=0\) and \(\delta=0\).
According to
\cite{Ne2},
this uniquely characterizes groups \(G\) of transfer kernel type \(\mathrm{b}.10\), \(\varkappa=(0043)\),
outside of mainlines,
having mainline parent \(\pi(G)\) of the same TKT.
\end{enumerate}

\end{remark}

\noindent
Table
\ref{tbl:MainLines}
summarizes parametrized power-commutator presentations \(G_\rho^{m,n}(\alpha,\beta,\gamma,\delta)\)
with parameters \(\alpha=\gamma=0\), \(-1\le\beta,\delta\le 1\), \(\rho=0\), \(4\le m<n\le 2m-3\), \(n=r+m-1\),
and transfer kernel types \(\varkappa\)
of all metabelian mainline groups on coclass trees \(\mathcal{T}\)
of the coclass graphs \(\mathcal{G}(3,r)\) with given coclass \(r\ge 2\).
In any case, the metabelian root \(G_i\) of a tree \(\mathcal{T}=\mathcal{T}(G_i)\)
is given by the top vertex \(G_i=G_0^{r+2,2r+1}(0,\beta,0,\delta)\),
for which \(e=m-1\), \(r=m-2\), and thus \(n=2r+1=2(r+2)-3=2m-3\).
For the sake of comparison, the mainline of \(\mathcal{G}(3,1)\) is also included.
Total transfers \(\varkappa(i)=0\) are counted by the invariant \(\nu(G)\), cfr.
\cite[Dfn. 4.2, p. 488]{Ma1}.



\begin{table}[ht]
\caption{Metabelian mainline groups of \(\mathcal{G}(3,r)\), \(r\ge 1\), sharing \(\varkappa(1)=0\)}
\label{tbl:MainLines}
\begin{center}
\begin{tabular}{|c|ccc|ccc|}
\hline
 \multicolumn{4}{|c|}{\(3\)-group of order \(3^n\)}                          & \multicolumn{3}{|c|}{transfer kernels}              \\
 \multicolumn{4}{|c|}{\downbracefill}                                        & \multicolumn{3}{|c|}{\downbracefill}                \\
 \(G\)                   & \(\mathrm{cc}(G)\) & \(m\ge r+2\) & \(n\ge 2r+1\) & TKT                    & \(\varkappa(G)\) & \(\nu(G)\)    \\
\hline
 \(G_0^n(0,0)\)          &            \(r=1\) & \(\ge 3\)    & \(\ge 3\)     & \(\mathrm{a}.1\)       & \((0000)\)    & \(4\)      \\
\hline
 \(G_0^{m,n}(0,0,0,0)\)  &         \(r\ge 2\) & \(\ge 4\)    & \(\ge 5\)     & \(\mathrm{b}.10\)      & \((0043)\)    & \(2\)      \\
\hline
 \(G_0^{m,n}(0,-1,0,1)\) &            \(r=2\) & \(\ge 4\)    & \(\ge 5\)     & \(\mathrm{c}.18\)      & \((0313)\)    & \(1\)      \\
 \(G_0^{m,n}(0,0,0,1)\)  &            \(r=2\) & \(\ge 4\)    & \(\ge 5\)     & \(\mathrm{c}.21\)      & \((0231)\)    & \(1\)      \\
\hline
 \(G_0^{m,n}(0,1,0,1)\)  &         \(r\ge 3\) & \(\ge 5\)    & \(\ge 7\)     & \(\mathrm{d}^\ast.19\) & \((0443)\)    & \(1\)      \\
 \(G_0^{m,n}(0,-1,0,1)\) &    \(r\ge 4\) even & \(\ge 6\)    & \(\ge 9\)     & \(\mathrm{d}^\ast.19\) & \((0343)\)    & \(1\)      \\
 \(G_0^{m,n}(0,0,0,1)\)  &         \(r\ge 3\) & \(\ge 5\)    & \(\ge 7\)     & \(\mathrm{d}^\ast.23\) & \((0243)\)    & \(1\)      \\
 \(G_0^{m,n}(0,1,0,0)\)  &         \(r\ge 3\) & \(\ge 5\)    & \(\ge 7\)     & \(\mathrm{d}^\ast.25\) & \((0143)\)    & \(1\)      \\
 \(G_0^{m,n}(0,-1,0,0)\) &    \(r\ge 4\) even & \(\ge 6\)    & \(\ge 9\)     & \(\mathrm{d}^\ast.25\) & \((0143)\)    & \(1\)      \\
\hline
\end{tabular}
\end{center}
\end{table}



\begin{proof}
The assumption \(5\le m<n\le 2m-4\), and thus \(e=n-m+2\le m-2\),
ensures that \(G\) is not a top vertex of the coclass graph \(\mathcal{G}(3,r)\).
Therefore, the last lower central \(\gamma_{m-1}(G)=\langle\sigma_{m-1}\rangle\) of \(G\)
is cyclic of order \(3\).
Since \(G\simeq G_\rho^{m,n}(\alpha,\beta,\gamma,\delta)\),
\(G\) is defined by the relations
(\ref{eqn:PwrCmtPres}),
\[s_3\sigma_3\sigma_4=\sigma_{m-2}^{\rho\beta}\sigma_{m-1}^{\gamma}\tau_e^{\delta},\quad
t_3\tau_3^{-1}\tau_4^{-1}=\sigma_{m-2}^{-\rho\delta}\sigma_{m-1}^{-\alpha}\tau_e^{-\beta},\quad
\tau_{e+1}=\sigma_{m-1}^{-\rho},\]
and the relations for the parent \(\pi(G)=G/\gamma_{m-1}(G)\) of \(G\) are
\[\bar{s}_3\bar{\sigma}_3\bar{\sigma}_4=\bar{\sigma}_{m-2}^{\rho\beta}\bar{\sigma}_{m-1}^{\gamma}\bar{\tau}_e^{\delta},\quad
\bar{t}_3\bar{\tau}_3^{-1}\bar{\tau}_4^{-1}=\bar{\sigma}_{m-2}^{-\rho\delta}\bar{\sigma}_{m-1}^{-\alpha}\bar{\tau}_e^{-\beta},\quad
\bar{\tau}_{e+1}=\bar{\sigma}_{m-1}^{-\rho},\]
where the left coset of an element \(g\in G\) with respect to \(\gamma_{m-1}(G)\)
is denoted by \(\bar{g}=g\cdot\gamma_{m-1}(G)\).
In particular, we have \(\bar{\sigma}_{m-1}=1\).
Since the order of the parent is \(\lvert\pi(G)\rvert=\lvert G\rvert:\lvert\gamma_{m-1}(G)\rvert=3^n:3=3^{n-1}\)
and the nilpotency class is \(\mathrm{cl}(\pi(G))=\mathrm{cl}(G)-1=m-2\),
the coclass \(r\) and the invariant \(e\) remain the same,
and we can view the relations as
\[\bar{s}_3\bar{\sigma}_3\bar{\sigma}_4=\bar{\sigma}_{m-3}^{0}\bar{\sigma}_{m-2}^{\rho\beta}\bar{\tau}_e^{\delta},\quad
\bar{t}_3\bar{\tau}_3^{-1}\bar{\tau}_4^{-1}=\bar{\sigma}_{m-3}^{0}\bar{\sigma}_{m-2}^{-\rho\delta}\bar{\tau}_e^{-\beta},\quad
\bar{\tau}_{e+1}=1.\]
Consequently, \(\pi(G)\simeq G_0^{m-1,n-1}(\rho\delta,\beta,\rho\beta,\delta)\),
that is \(\pi(G)\simeq G_0^{m-1,n-1}(\tilde\alpha,\tilde\beta,\tilde\gamma,\tilde\delta)\)
with \(\tilde\alpha=\rho\delta\), \(\tilde\gamma=\rho\beta\),
but \(\tilde\beta=\beta\), \(\tilde\delta=\delta\) remain unchanged.
\end{proof}



The following principle,
that the kernel \(\varkappa(1)\) of the transfer from \(G\) to the first distinguished maximal subgroup \(H_1=\chi_s(G)\)
decides about the relation between depth \(\mathrm{dp}(G)\) and defect \(k=k(G)\) of \(G\),
is already known from metabelian \(p\)-groups \(G\) of coclass \(\mathrm{cc}(G)=1\).

\begin{corollary}
\label{cor:DpthCoclGe2}
For a metabelian \(3\)-group \(G\) of coclass \(r=\mathrm{cc}(G)\ge 2\)
having abelianization \(G/G^\prime\simeq(3,3)\) and defect of commutativity \(k=k(G)\),
which does not belong to the sporadic part \(\mathcal{G}_0(3,r)\),
the depth \(\mathrm{dp}(G)\) of \(G\)
on its coclass tree \(\mathcal{T}\), as a subset of \(\mathcal{G}(3,r)\), is given by
\[\mathrm{dp}(G)=
\begin{cases}
k+1, & \text{ if } \varkappa(1)\ne 0, \\
k,   & \text{ if } \varkappa(1)=0,
\end{cases}
\]
with respect to the natural order of the maximal subgroups of \(G\).
\end{corollary}



\begin{proof}
This follows immediately from Theorem
\ref{thm:CoclGe2Core}
and the remark thereafter:
The system of all groups
\(G_0^{m,n}(0,\beta,0,\delta)\) with arbitrary \(4\le m<n\le 2m-3\), \(-1\le\beta,\delta\le 1\), but \(\alpha=\gamma=\rho=0\),
consists of all mainline groups on coclass trees \(\mathcal{T}\) of \(\mathcal{G}(3,r)\), \(r=n-m+1\),
that is, of all groups \(G\) with depth \(\mathrm{dp}(G)=0=k\) equal to the defect \(k\).
According to Table
\ref{tbl:MainLines},
all these mainline groups have a total transfer \(\varkappa(1)=0\)
to the first distinguished maximal subgroup \(H_1\).

Since the defect of a group \(G_\rho^{m,n}(\alpha,\beta,\gamma,\delta)\) with parameter \(\rho=0\) is \(k=0\),
all the other groups \(G=G_0^{m,n}(\alpha,\beta,\gamma,\delta)\), \((\alpha,\gamma)\ne (0,0)\), with defect \(k=0\)
must be located at depth \(\mathrm{dp}(G)=1=k+1\) on a coclass tree \(\mathcal{T}\)
or as a top vertex on the sporadic part \(\mathcal{G}_0(3,r)\) of \(\mathcal{G}(3,r)\).
According to
\cite[Thm. 6.14, pp. 208 ff]{Ne1},
supplemented by
\cite[Thm. 3.3]{Ma2},
all these groups have a partial transfer \(\varkappa(1)\ne 0\)
to the first distinguished maximal subgroup \(H_1\).

On the other hand, Theorem
\ref{thm:CoclGe2Core}
shows that the relation between the defects of parent \(\tilde\pi(G)\) and descendant \(G\)
is given by \(\tilde k=0=k-1\) for any group \(G=G_\rho^{m,n}(\alpha,\beta,\gamma,\delta)\), \(\rho=\mp 1\),
with positive defect \(k=1\), whence the depth,
that is the number of steps required to reach the mainline
by successive construction of parents,
is given by
\[\mathrm{dp}(G)=
\begin{cases}
2=k+1, \text{ if }\varkappa(1)\ne 0, \\
1=k, \text{ if }\varkappa(1)=0.
\end{cases}\]
The groups \(G\) with positive defect \(k=1\)
are characterized by a partial transfer \(\varkappa(1)\ne 0\)
to the first distinguished maximal subgroup \(H_1\),
according to
\cite[Thm. 6.14, pp. 208 ff]{Ne1}.
The only exception are the groups \(G\) with parameters \(\beta=0\) and \(\delta=0\),
that is, those with transfer kernel type \(\mathrm{b}.10\), \(\varkappa=(0043)\), \(\varkappa(1)=0\),
outside of mainlines.
\end{proof}



We conjecture that the following property 
of mainline groups of \(\mathcal{G}(3,r)\)
might be true for mainline groups on any
coclass tree of \(\mathcal{G}(p,r)\), \(p\ge 3\) prime, \(r\ge 1\).

\begin{corollary}
\label{cor:MainLineCoclGe2}
Mainline groups on a coclass tree \(\mathcal{T}\) of \(\mathcal{G}(3,r)\), \(r\ge 1\),
that is, groups of depth \(\mathrm{dp}(G)=0\),
must have a total transfer \(\varkappa(1)=0\)
to the distinguished maximal subgroup \(H_1=\chi_s(G)\).
\end{corollary}

\begin{proof}
See Table
\ref{tbl:MainLines}.
\end{proof}

\noindent
Only the groups of TKT a.1, \(\varkappa=(0000)\), and b.10, \(\varkappa=(0043)\),
outside of mainlines,
prohibit that the converse of Corollary
\ref{cor:MainLineCoclGe2}
is also true.



Top vertices \(G\)
on coclass trees \(\mathcal{T}\)
and on the sporadic part \(\mathcal{G}_0(3,r)\)
of a coclass graph \(\mathcal{G}(3,r)\)
are groups of minimal class within their coclass \(r\ge 2\).
They are BF \textit{groups} with bicyclic factors, except \(\gamma_2(G)/\gamma_3(G)\),
in particular having a bicyclic last lower central \(\gamma_{m-1}(G)\) of type \((3,3)\),
and consequently they do not possess a parent \(\pi(G)\) on the same coclass graph.
They form the \textit{interface} between the coclass graphs
\(\mathcal{G}(3,r)\) and \(\mathcal{G}(3,r-1)\).
We call the last lower central quotient \(G/\gamma_{m-1}(G)\) of \(G\)
the \textit{generalized parent} \(\tilde\pi(G)\) of \(G\)
but we point out that there is no directed edge of depth \(1\)
from \(\tilde\pi(G)\) to \(G\).
However,
in the complete graph \(\mathcal{G}(3)\) of all finite \(3\)-groups
as defined by Leedham-Green and Newman
\cite[p. 194]{LgNm},
there is a directed edge of depth \(2\)
from \(\tilde\pi(G)\) to \(G\).
This supergraph
\(\mathcal{G}(3)\) is the disjoint union of all coclass graphs
\(\mathcal{G}(3,r)\), \(r\ge 0\).

\begin{theorem}
\label{thm:CoclGe2Intf}
Let \(G\) be a metabelian \(3\)-group of coclass \(r=\mathrm{cc}(G)\ge 2\)
with \(G/G^\prime\simeq (3,3)\),
such that 
\(\lvert G\rvert=3^n\), \(\mathrm{cl}(G)=m-1\), \(4\le m<n=2m-3\),
\(r=n-m+1=m-2\), \(e=n-m+2=m-1\), and consequently \(k=0\),
that is, \(G\) is an interface group
with bicyclic last lower central \(\gamma_{m-1}\) of type \((3,3)\).
Then the generalized parent \(\tilde\pi(G)\in\mathcal{G}(3,r-1)\) of \(G\in\mathcal{G}(3,r)\) is given by
\[\tilde\pi(G)\simeq
\begin{cases}
G_0^3(0,0)\in\mathcal{G}(3,1), & \text{ if } m=4 \text{ (and thus } n=5,\ r=2), \\
G_0^{m-1,n-2}(0,0,0,0)\in\mathcal{G}(3,r-1), & \text{ if } m\ge 5 \text{ (and thus } n\ge 7,\ r\ge 3).
\end{cases}
\]
\end{theorem}

\begin{proof}
First, we consider the very special transition from second maximal to maximal class.
The assumption \(m=4\) implies \(n=2m-3=5\).
The last lower central \(\gamma_{3}(G)=\langle\sigma_{3},\tau_{3}\rangle\)
is bicyclic of order \(3^2\),
and the generalized parent \(\tilde\pi(G)=G/\gamma_{3}(G)\) is of order \(3^3\),
of nilpotency class \(2\) and of coclass \(1\).
The group \(G\) of type \(G\simeq G_0^{4,5}(\alpha,\beta,\gamma,\delta)\)
satisfies the following special form of Nebelung's relations
(\ref{eqn:PwrCmtPres}),
\[s_2^3=1,\quad
s_3\sigma_3=\sigma_3^{\gamma}\tau_3^{\delta},\quad
t_3\tau_3^{-1}=\sigma_3^{-\alpha}\tau_3^{-\beta},\]
and since \(\bar{\sigma}_3=1\) and \(\bar{\tau}_3=1\),
the relations for the generalized parent \(\tilde\pi(G)\) can be written as Blackburn's relations
(\ref{eqn:PwrRelCocl1})
and
(\ref{eqn:CmtRelCocl1}),
\[\bar{x}^3=\bar{\tau_3}=1,\quad
\bar{y}^3\bar{s}_2^3\bar{s}_3=\bar{\sigma}_3\cdot 1\cdot\bar{\sigma}_3^{\gamma-1}\bar{\tau}_3^\delta=1,\quad
\lbrack\bar{s_2},\bar{y}\rbrack=\bar{t_3}=\bar{\sigma}_3^{-\alpha}\bar{\tau}_3^{1-\beta}=1,\]
which imply that \(\tilde\pi(G)\simeq G_0^3(0,0)\).

Now, let \(m\ge 5\).
Since \(e=m-1\),
and the last lower central \(\gamma_{m-1}(G)=\langle\sigma_{m-1},\tau_{m-1}\rangle\)
is bicyclic of type \((3,3)\),
the order of the generalized parent is \(\lvert\tilde\pi(G)\rvert=\lvert G\rvert:\lvert\gamma_{m-1}(G)\rvert=3^n:3^2=3^{n-2}\),
the nilpotency class is \(\mathrm{cl}(\tilde\pi(G))=\mathrm{cl}(G)-1=m-2\),
and the coclass \(\tilde r=n-2-(m-2)=2m-3-m=m-3=r-1\) and the invariant \(\tilde e=\tilde r+1=r=e-1\) decrease by \(1\).
Since \(k=0\), \(\rho=0\),
the group \(G\) of type \(G\simeq G_0^{m,n}(\alpha,\beta,\gamma,\delta)\),
is defined by a special form of the relations
(\ref{eqn:PwrCmtPres}),
\[s_3\sigma_3\sigma_4=\sigma_{m-1}^{\gamma}\tau_e^{\delta},\quad
t_3\tau_3^{-1}\tau_4^{-1}=\sigma_{m-1}^{-\alpha}\tau_e^{-\beta},\quad
\tau_{e+1}=1.\]
Since \(\bar{\sigma}_{m-1}=1\) and \(\bar{\tau}_{m-1}=1\),
the relations for the generalized parent \(\tilde\pi(G)\) are
\[\bar{s}_3\bar{\sigma}_3\bar{\sigma}_4=\bar{\sigma}_{m-1}^{\gamma}\bar{\tau}_{m-1}^{\delta}=1,\quad
\bar{t}_3\bar{\tau}_3^{-1}\bar{\tau}_4^{-1}=\bar{\sigma}_{m-1}^{-\alpha}\bar{\tau}_{m-1}^{-\beta}=1,\quad
\bar{\tau}_e=\bar{\tau}_{m-1}=1,\]
and therefore we have \(\tilde\pi(G)\simeq G_0^{m-1,n-2}(0,0,0,0)\).
\end{proof}



\subsection{Second \(3\)-class groups \(G=\mathrm{G}_3^2(K)\) of coclass \(\mathrm{cc}(G)\ge 2\) with \(G/G^\prime\simeq (3,3)\)}
\label{ss:ScndClgpTyp33CoclGe2}

\subsubsection{Weak transfer target type \(\tau_0(G)\) expressed by \(3\)-class numbers}
\label{sss:wTTTCoclGe2}

The group theoretic information
on the second \(3\)-class group \(G=\mathrm{G}_3^2(K)\),
that is, its class, coclass, and defect,
is contained in the \(3\)-class numbers
of the two distinguished extensions \(L_1,L_2\)
and of the Hilbert \(3\)-class field \(\mathrm{F}_3^1(K)\).
Additionally,
the principalization \(\kappa(1)\) of \(K\) in the first distinguished extension \(L_1\)
determines the connection between defect and depth of \(G\).

\begin{theorem}
\label{thm:wTTTCoclGe2}

Let \(K\) be an arbitrary number field
having \(3\)-class group \(\mathrm{Cl}_3(K)\) of type \((3,3)\).
Suppose the second \(3\)-class group
\(G=\mathrm{Gal}(\mathrm{F}_3^2(K)\vert K)\) of \(K\)
is of coclass \(\mathrm{cc}(G)\ge 2\)
with defect \(k=k(G)\),
order \(\lvert G\rvert=3^n\),
and class \(\mathrm{cl}(G)=m-1\),
where \(4\le m<n\le 2m-3\).
With respect to the natural order of
the maximal subgroups \((H_i)_{1\le i\le 4}\) of \(G\),
fixed in Definition
\ref{dfn:NatOrdCoclGe2},
the weak transfer target type \(\tau_0(G)=(\lvert H_i/H_i^{\prime}\rvert)_{1\le i\le 4}\),
that is the family of \(3\)-class numbers of the quadruplet \((L_1,\ldots,L_4)\)
of unramified cyclic cubic extension fields of \(K\), forming the first layer,
is given by

\begin{eqnarray*}
\tau_0(G) = (\mathrm{h}_3(L_1),\ldots,\mathrm{h}_3(L_4)) =
(3^{\mathrm{cl}(G)-k},3^{\mathrm{cc}(G)+1},3^3,3^3),
\end{eqnarray*}

\noindent
where,
in the case of a non-sporadic group \(G\) on some coclass tree,
defect \(k\) and depth \(\mathrm{dp}(G)\) are related by
\[k=
\begin{cases}
\mathrm{dp}(G)-1, & \text{ if } \varkappa(1)\ne 0, \\
\mathrm{dp}(G),   & \text{ if } \varkappa(1)=0.
\end{cases}
\]

\noindent
For the second layer,
consisting of the Hilbert \(3\)-class field \(\mathrm{F}_3^1(K))\) only,
the \(3\)-class number is given by
\[\mathrm{h}_3(\mathrm{F}_3^1(K)) = 3^{\mathrm{cl}(G)+\mathrm{cc}(G)-2}.\]

\end{theorem}

\begin{proof}
The statement is a succinct version of
\cite[Thm. 3.4]{Ma1},
expressed by concepts more closely related to the
position of \(G\) on the coclass graphs \(\mathcal{G}(3,r)\), \(r\ge 2\),
and to the transfer kernel type \(\varkappa(G)\) of \(G\),
using Corollary
\ref{cor:DpthCoclGe2}.
\end{proof}

\begin{remark}
Whereas \(\mathrm{h}_3(L_3)\) and \(\mathrm{h}_3(L_4)\) only indicate that \(\mathrm{cc}(G)\ge 2\),
the second distinguished \(\mathrm{h}_3(L_2)\) gives the precise coclass \(r\) of \(G\),
\(\mathrm{h}_3(\mathrm{F}_3^1(K))\) determines the order \(3^n\), \(n=\mathrm{cl}(G)+\mathrm{cc}(G)\), and class of \(G\),
and the first distinguished \(\mathrm{h}_3(L_1)\) yields the defect \(k\) of \(G\).\\
With respect to the mainline \((M_j)_{j\ge 2r+1}\) of the coclass tree \(\mathcal{T}\) containing \(G\),
the order \(\lvert M_i\rvert=3^i\) of the branch root \(M_i\) of a non-sporadic group \(G\) is given by
\(i=n-\mathrm{dp}(G)=\mathrm{cl}(G)+\mathrm{cc}(G)-\mathrm{dp}(G)\),
where
\[\mathrm{dp}(G)=
\begin{cases}
k,   & \text{ if } \varkappa(1)=0, \\
k+1, & \text{ if } \varkappa(1)\ne 0.
\end{cases}
\]
\end{remark}



\subsubsection{Selection Rules for quadratic base fields}
\label{sss:SelRuleCoclGe2}

Let \(K=\mathbb{Q}(\sqrt{D})\) be a quadratic number field with discriminant \(D\)
and \(3\)-class group \(\mathrm{Cl}_3(K)\) of type \((3,3)\).
Then the \(4\) unramified cyclic cubic extension fields \((L_1,\ldots,L_4)\) of \(K\)
have dihedral absolute Galois groups \(\mathrm{Gal}(L_i\vert K)\) of degree \(6\),
according to
\cite[Prop. 4.1]{Ma1}.
Consequently each sextic field \(L_i\) contains a cubic subfield \(K_i\),
whose invariants can be computed easier than those of \(L_i\)
and are also sufficient to determine complete information on the group \(G=\mathrm{G}_3^2(K)\).

\begin{theorem}
\label{thm:SelRuleCoclGe2}

Let \(G=\mathrm{Gal}(\mathrm{F}_3^2(K)\vert K)\)
be the second \(3\)-class group of \(K=\mathbb{Q}(\sqrt{D})\).
If \(G\in\mathcal{G}(3,r)\) for some coclass \(r\ge 2\), then
the family of \(3\)-class numbers of the non-Galois subfields \(K_i\) of \(L_i\),
with respect to the natural order fixed in Definition
\ref{dfn:NatOrdCoclGe2}
is given by

\begin{eqnarray*}
(\mathrm{h}_3(K_1),\ldots,\mathrm{h}_3(K_4)) = 
\begin{cases}
(3^{\frac{\mathrm{cl}(G)-(k+1)}{2}},3^{\frac{\mathrm{cc}(G)}{2}},3,3)            & \text{ for sporadic }G,\text{ (where always }\varkappa(2)\ne 0,) \\
(3^{\frac{\mathrm{cl}(G)-\mathrm{dp}(G)}{2}},3^{\frac{\mathrm{cc}(G)+1}{2}},3,3) & \text{ otherwise, if }\varkappa(2)=0, \\
(3^{\frac{\mathrm{cl}(G)-\mathrm{dp}(G)}{2}},3^{\frac{\mathrm{cc}(G)}{2}},3,3)   & \text{ otherwise, if }\varkappa(2)\ne 0.
\end{cases} \\
\end{eqnarray*}

\noindent
The order \(\lvert M_i\rvert=3^i\) of the root \(M_i\)
for a non-sporadic group \(G\) on branch \(\mathcal{B}(M_i)\) of some coclass tree
is given by
\[i=\mathrm{cl}(G)+\mathrm{cc}(G)-\mathrm{dp}(G)\equiv
\begin{cases}
 1\pmod{2}, & \text{ if }\varkappa(2)=0, \\
 0\pmod{2}, & \text{ if }\varkappa(2)\ne 0.
\end{cases}\]

\end{theorem}

\begin{remark}
While \(\mathrm{h}_3(K_3)\) and \(\mathrm{h}_3(K_4)\) do not provide any information,
the second distinguished \(\mathrm{h}_3(K_2)\) indicates the coclass of \(G\) and enforces the parity
\[\mathrm{cc}(G)\equiv
\begin{cases}
 1\pmod{2}, & \text{ if }\varkappa(2)=0, \\
 0\pmod{2}, & \text{ if }\varkappa(2)\ne 0,
\end{cases}\]
in dependence on the principalization \(\kappa(2)\) of \(K\) in the second distinguished extension \(L_2\),
and the first distinguished \(\mathrm{h}_3(K_1)\) demands
\(\mathrm{cl}(G)-\mathrm{dp}(G)\equiv 0\pmod{2}\),
for non-sporadic \(G\).
\end{remark}

\begin{proof}
See 
\cite[Thm. 4.2.]{Ma1}.
For the branch root order \(\lvert M_i\rvert=3^i\)
of a non-sporadic vertex \(G\) of order \(\lvert G\rvert=3^n\),
we use the relations
\(n=\mathrm{cl}(G)+\mathrm{cl}(G)\) and \(i=n-\mathrm{dp}(G)\),
the depth being the number of successive steps on the path between \(G\) and \(M_i\),
each decreasing order and class by \(1\) and keeping the coclass constant.
\end{proof}



\subsubsection{Identifying densely populated vertices by fast algorithms}
\label{sss:DensePopulation}

The top vertices \(G\) on \(\mathcal{G}(3,2)\) with \(G/G^\prime\) of type \((3,3)\)
in Figure
\ref{fig:Typ33Cocl2}
can be identified by the fast algorithm given in
\cite[\S\ 5.2--5.3]{Ma3},
using the TTT and the counter \(\varepsilon\) of \(\tau(i)=(3,3,3)\), \(1\le i\le 4\), in Table
\ref{tbl:TttTop33}.

\renewcommand{\arraystretch}{1.0}

\begin{table}[ht]
\caption{TTT and \(\varepsilon\) of the top vertices of type \((3,3)\) on \(\mathcal{G}(3,2)\)}
\label{tbl:TttTop33}
\begin{center}
\begin{tabular}{|ll|lc|r|rrrr|c|}
\hline
 Id of                     & isoclinism    &                   &               & Hilbert       & \multicolumn{4}{|c|}{TTT \(\tau(G)\)}                     &                 \\
 \(3\)-group \(G\)         & family        & TKT               & \(\varkappa\) & TTT           & \multicolumn{4}{|c|}{\downbracefill}                      &                 \\
                           &               &                   &               & \(\tau(0)\)   & \(\tau(1)\)  & \(\tau(2)\)  & \(\tau(3)\)  & \(\tau(4)\)  & \(\varepsilon\) \\
\hline
 \(\langle 243,5\rangle\)  & \(\Phi_6\)    & \(\mathrm{D}.10\) & \((2241)\)    & \((3,3,3)\)   &   \((9,3)\) &   \((9,3)\) & \((3,3,3)\) &   \((9,3)\) &           \(1\) \\
 \(\langle 243,7\rangle\)  & \(\Phi_6\)    & \(\mathrm{D}.5\)  & \((4224)\)    & \((3,3,3)\)   & \((3,3,3)\) &   \((9,3)\) & \((3,3,3)\) &   \((9,3)\) &           \(2\) \\
 \(\langle 243,9\rangle\)  & \(\Phi_6\)    & \(\mathrm{G}.19\) & \((2143)\)    & \((3,3,3)\)   &   \((9,3)\) &   \((9,3)\) &   \((9,3)\) &   \((9,3)\) &           \(0\) \\
 \(\langle 729,57\rangle\) & \(\Phi_{43}\) & \(\mathrm{G}.19\) & \((2143)\)    & \((3,3,3,3)\) &   \((9,3)\) &   \((9,3)\) &   \((9,3)\) &   \((9,3)\) &           \(0\) \\
 \(\langle 243,4\rangle\)  & \(\Phi_6\)    & \(\mathrm{H}.4\)  & \((4443)\)    & \((3,3,3)\)   & \((3,3,3)\) & \((3,3,3)\) &   \((9,3)\) & \((3,3,3)\) &           \(3\) \\
 \(\langle 729,45\rangle\) & \(\Phi_{42}\) & \(\mathrm{H}.4\)  & \((4443)\)    & \((9,3,3)\)   & \((3,3,3)\) & \((3,3,3)\) &   \((9,3)\) & \((3,3,3)\) &           \(3\) \\
\hline
 \(\langle 243,3\rangle\)  & \(\Phi_6\)    & \(\mathrm{b}.10\) & \((0043)\)    & \((3,3,3)\)   &   \((9,3)\) &   \((9,3)\) & \((3,3,3)\) & \((3,3,3)\) &           \(2\) \\
 \(\langle 243,6\rangle\)  & \(\Phi_6\)    & \(\mathrm{c}.18\) & \((0313)\)    & \((3,3,3)\)   &   \((9,3)\) &   \((9,3)\) & \((3,3,3)\) &   \((9,3)\) &           \(1\) \\
 \(\langle 729,49\rangle\) & \(\Phi_{23}\) & \(\mathrm{c}.18\) & \((0313)\)    & \((9,3,3)\)   &   \((9,9)\) &   \((9,3)\) & \((3,3,3)\) &   \((9,3)\) &           \(1\) \\
 \(\langle 243,8\rangle\)  & \(\Phi_6\)    & \(\mathrm{c}.21\) & \((0231)\)    & \((3,3,3)\)   &   \((9,3)\) &   \((9,3)\) &   \((9,3)\) &   \((9,3)\) &           \(0\) \\
 \(\langle 729,54\rangle\) & \(\Phi_{23}\) & \(\mathrm{c}.21\) & \((0231)\)    & \((9,3,3)\)   &   \((9,9)\) &   \((9,3)\) &   \((9,3)\) &   \((9,3)\) &           \(0\) \\
\hline
\end{tabular}
\end{center}
\end{table}

For quadratic fields \(K=\mathbb{Q}(\sqrt{D})\), the following metabelian \(3\)-groups \(G\in\Phi_6\)
cannot be realized as second \(3\)-class groups \(\mathrm{G}_3^2(K)\).

\begin{itemize}
\item
\(\langle 243,9\rangle\), \(\langle 243,4\rangle\),
since they are not Schur \(\sigma\)-groups
(Thm.
\ref{thm:SpecWeakLeaf}).
\item
\(\langle 243,6\rangle\), \(\langle 243,8\rangle\),
due to the selection rule for branches in Theorem
\ref{thm:SelRuleCoclGe2}.
\item
\(\langle 243,3\rangle\),
according to the remark after Theorem
\ref{thm:SelRuleCoclGe2},
since \(\varkappa(2)=0\) enforces odd coclass.
\end{itemize}



\subsubsection{Top vertices of type \((3,3)\) on \(\mathcal{G}(3,2)\)}
\label{sss:Typ33Cocl2}

Figure
\ref{fig:Typ33Cocl2}
shows the interface between the coclass graphs
\(\mathcal{G}(3,1)\) and \(\mathcal{G}(3,2)\).
The extra special group \(G_0^3(0,0)\) of order \(27\) and exponent \(3\),
which is the second member of the unique mainline of \(\mathcal{G}(3,1)\),
is the generalized parent \(\tilde\pi(G)\) of all top vertices \(G\) of \(\mathcal{G}(3,2)\).
We point out that the connecting edges of depth \(2\)
neither belong to \(\mathcal{G}(3,1)\) nor to \(\mathcal{G}(3,2)\).
The metabelian skeleton of this graph is also shown in
\cite[p. 189 ff]{Ne1}
and the complete graph, including the non-metabelian leaves, was first drawn in
\cite[Tbl. 1--2, pp. 265--266]{AHL}
and
\cite[Fig. 4.6--4.7, p. 74]{As1}.
Among the non-CF groups \(G\) with abelianization \(G/G^\prime\) of type \((3,3)\)
at the top of coclass graph \(\mathcal{G}(3,2)\),
which form the stem of isoclinism family \(\Phi_6\),
we have, from the left to the right:

\begin{itemize}
\item
two terminal vertices \(\langle 243,5\rangle\), \(\langle 243,7\rangle\),
the only Schur \(\sigma\)-groups of order \(3^5\)
\cite[Thm. 4.2, p. 14]{BBH},
\item
two roots \(\langle 243,9\rangle\) and \(\langle 243,4\rangle\)
of finite trees of depth \(2\),
whose metabelian descendants belong to the
stem of the isoclinism families \(\Phi_{43}\) and \(\Phi_{42}\),
\item
a root \(\langle 243,3\rangle\) of an infinite tree,
which is not a coclass tree and is shown in Figure
\ref{fig:TreeBTyp33Cocl2},
\item
and two roots \(\langle 243,6\rangle\), \(\langle 243,8\rangle\)
of coclass trees, shown in detail in Figures
\ref{fig:TreeQTyp33Cocl2}
and
\ref{fig:TreeUTyp33Cocl2}.
\end{itemize}

The sporadic part \(\mathcal{G}_0(3,2)\) of \(\mathcal{G}(3,2)\) consists of
the terminal vertices \(\langle 243,5\rangle\) and \(\langle 243,7\rangle\),
the finite trees \(\mathcal{T}(\langle 243,9\rangle)\) and \(\mathcal{T}(\langle 243,4)\rangle\),
and a certain finite subset of the difference
\(\mathcal{T}(\langle 243,3\rangle)\setminus\mathcal{T}(\langle 729,40\rangle)\).

\begin{figure}[ht]
\caption{Sporadic groups and roots of coclass trees on the coclass graph \(\mathcal{G}(3,2)\)}
\label{fig:Typ33Cocl2}

\setlength{\unitlength}{1cm}
\begin{picture}(16,14)(-1,-11)

\put(0,2.5){\makebox(0,0)[cb]{Order \(3^n\)}}
\put(0,2){\line(0,-1){10}}
\multiput(-0.1,2)(0,-2){6}{\line(1,0){0.2}}
\put(-0.2,2){\makebox(0,0)[rc]{\(9\)}}
\put(0.2,2){\makebox(0,0)[lc]{\(3^2\)}}
\put(-0.2,0){\makebox(0,0)[rc]{\(27\)}}
\put(0.2,0){\makebox(0,0)[lc]{\(3^3\)}}
\put(-0.2,-2){\makebox(0,0)[rc]{\(81\)}}
\put(0.2,-2){\makebox(0,0)[lc]{\(3^4\)}}
\put(-0.2,-4){\makebox(0,0)[rc]{\(243\)}}
\put(0.2,-4){\makebox(0,0)[lc]{\(3^5\)}}
\put(-0.2,-6){\makebox(0,0)[rc]{\(729\)}}
\put(0.2,-6){\makebox(0,0)[lc]{\(3^6\)}}
\put(-0.2,-8){\makebox(0,0)[rc]{\(2\,187\)}}
\put(0.2,-8){\makebox(0,0)[lc]{\(3^7\)}}

\put(2.2,2.2){\makebox(0,0)[lc]{\(C_3\times C_3\)}}
\put(1.9,1.9){\framebox(0.2,0.2){}}
\put(2,2){\line(0,-1){2}}
\put(2,0){\circle{0.2}}
\put(2.2,0.2){\makebox(0,0)[lc]{\(G^3_0(0,0)\)}}

\put(2,0){\line(1,-4){1}}
\put(2,0){\line(1,-2){2}}
\put(2,0){\line(1,-1){4}}
\put(2,0){\line(3,-2){6}}
\put(2,0){\line(2,-1){8}}
\put(2,0){\line(5,-2){10}}
\put(2,0){\line(3,-1){12}}
\put(7,-1){\makebox(0,0)[lc]{Edges of depth \(2\) forming the interface}}
\put(8,-1.5){\makebox(0,0)[lc]{between \(\mathcal{G}(3,1)\) and \(\mathcal{G}(3,2)\)}}

\put(2,-3.6){\makebox(0,0)[cc]{\(\Phi_6\)}}
\multiput(3,-4)(1,0){2}{\circle*{0.2}}
\put(3.1,-3.9){\makebox(0,0)[lb]{\(\langle 5\rangle\)}}
\put(4.1,-3.9){\makebox(0,0)[lb]{\(\langle 7\rangle\)}}
\multiput(6,-4)(2,0){5}{\circle*{0.2}}
\put(6.1,-3.9){\makebox(0,0)[lb]{\(\langle 9\rangle\)}}
\put(8.1,-3.9){\makebox(0,0)[lb]{\(\langle 4\rangle\)}}
\put(10.1,-3.9){\makebox(0,0)[lb]{\(\langle 3\rangle\)}}
\put(12.1,-3.9){\makebox(0,0)[lb]{\(\langle 6\rangle\)}}
\put(14.1,-3.9){\makebox(0,0)[lb]{\(\langle 8\rangle\)}}

\put(3.2,-4){\oval(0.8,1)}
\put(2.2,-4.5){\makebox(0,0)[cc]{\underbar{\textbf{667}}/\underbar{\textbf{93}}}}
\put(4.2,-4){\oval(0.8,1)}
\put(5.2,-4.5){\makebox(0,0)[cc]{\underbar{\textbf{269}}/\underbar{\textbf{47}}}}

\put(2,-5){\makebox(0,0)[cc]{\textbf{TKT:}}}
\put(3,-5){\makebox(0,0)[cc]{D.10}}
\put(3,-5.5){\makebox(0,0)[cc]{\((2241)\)}}
\put(4,-5){\makebox(0,0)[cc]{D.5}}
\put(4,-5.5){\makebox(0,0)[cc]{\((4224)\)}}
\put(1.5,-5.75){\framebox(3,1){}}

\put(6.5,-6.5){\makebox(0,0)[cc]{\(\Phi_{43}\)}}
\put(6,-4){\line(0,-1){2}}
\put(5.9,-5.9){\makebox(0,0)[rc]{\(\langle 57\rangle\)}}
\put(6,-4){\line(1,-4){0.5}}
\multiput(6,-6)(0.5,0){2}{\circle*{0.1}}
\multiput(6,-6)(0.5,0){2}{\line(0,-1){2}}
\multiput(5.95,-8.05)(0.5,0){2}{\framebox(0.1,0.1){}}

\put(9,-6.5){\makebox(0,0)[cc]{\(\Phi_{42}\)}}
\put(8,-4){\line(0,-1){2}}
\put(7.9,-5.9){\makebox(0,0)[rc]{\(\langle 45\rangle\)}}
\put(8,-4){\line(1,-4){0.5}}
\put(8,-4){\line(1,-2){1}}
\put(8,-4){\line(3,-4){1.5}}
\multiput(8,-6)(0.5,0){4}{\circle*{0.1}}
\multiput(8,-6)(0.5,0){2}{\line(0,-1){2}}
\multiput(7.95,-8.05)(0.5,0){2}{\framebox(0.1,0.1){}}
\multiput(7.9,-7.9)(0.5,0){2}{\makebox(0,0)[rc]{\(4*\)}}

\put(5.7,-6){\oval(1,1)}
\put(5.4,-5.3){\makebox(0,0)[cc]{\underbar{\textbf{94}}/\underbar{\textbf{11}}}}
\put(7.7,-6){\oval(1,1)}
\put(7.3,-5.3){\makebox(0,0)[cc]{\underbar{\textbf{297}}/\underbar{\textbf{27}}}}

\put(5,-8.5){\makebox(0,0)[cc]{\textbf{TKT:}}}
\put(6.25,-8.5){\makebox(0,0)[cc]{G.19}}
\put(6.25,-9){\makebox(0,0)[cc]{\((2143)\)}}
\put(8.75,-8.5){\makebox(0,0)[cc]{H.4}}
\put(8.75,-9){\makebox(0,0)[cc]{\((4443)\)}}
\put(4.5,-9.25){\framebox(5,1){}}


\put(11.25,-6.5){\makebox(0,0)[cc]{\(\Phi_{40},\Phi_{41}\)}}
\multiput(10,-4)(0,-2){2}{\line(0,-1){2}}
\multiput(10,-6)(0,-2){2}{\circle*{0.2}}
\put(10.1,-5.9){\makebox(0,0)[lc]{\(\langle 40\rangle\)}}
\put(10,-8){\vector(0,-1){2}}
\put(10,-10.2){\makebox(0,0)[ct]{\(\mathcal{T}(\langle 729,40\rangle)\)}}
\put(10,-4){\line(3,-4){1.5}}
\put(11.5,-6){\circle*{0.1}}
\put(11.3,-5.9){\makebox(0,0)[rc]{\(6*\)}}

\multiput(12,-4)(0,-2){2}{\line(0,-1){2}}
\multiput(12,-6)(0,-2){2}{\circle*{0.2}}
\put(12.1,-5.9){\makebox(0,0)[lc]{\(\langle 49\rangle\)}}
\put(12,-8){\vector(0,-1){2}}
\put(12,-10.2){\makebox(0,0)[ct]{\(\mathcal{T}(\langle 243,6\rangle)\)}}

\put(13.1,-6.5){\makebox(0,0)[cc]{\(\Phi_{23}\)}}
\put(12,-9){\makebox(0,0)[cc]{\(3\) mainlines}}

\multiput(14,-4)(0,-2){2}{\line(0,-1){2}}
\multiput(14,-6)(0,-2){2}{\circle*{0.2}}
\put(14.1,-5.9){\makebox(0,0)[lc]{\(\langle 54\rangle\)}}
\put(14,-8){\vector(0,-1){2}}
\put(14,-10.2){\makebox(0,0)[ct]{\(\mathcal{T}(\langle 243,8\rangle)\)}}

\put(9,-11){\makebox(0,0)[cc]{\textbf{TKT:}}}
\put(10,-11){\makebox(0,0)[cc]{b.10}}
\put(10,-11.5){\makebox(0,0)[cc]{\((0043)\)}}
\put(12,-11){\makebox(0,0)[cc]{c.18}}
\put(12,-11.5){\makebox(0,0)[cc]{\((0313)\)}}
\put(14,-11){\makebox(0,0)[cc]{c.21}}
\put(14,-11.5){\makebox(0,0)[cc]{\((0231)\)}}
\put(8.5,-11.75){\framebox(6,1){}}

\put(12.3,-6){\oval(1,1)}
\put(12.5,-5.3){\makebox(0,0)[cc]{\underbar{\textbf{0}}/\underbar{\textbf{29}}}}
\put(14.3,-6){\oval(1,1)}
\put(14.5,-5.3){\makebox(0,0)[cc]{\underbar{\textbf{0}}/\underbar{\textbf{25}}}}

\end{picture}

\end{figure}
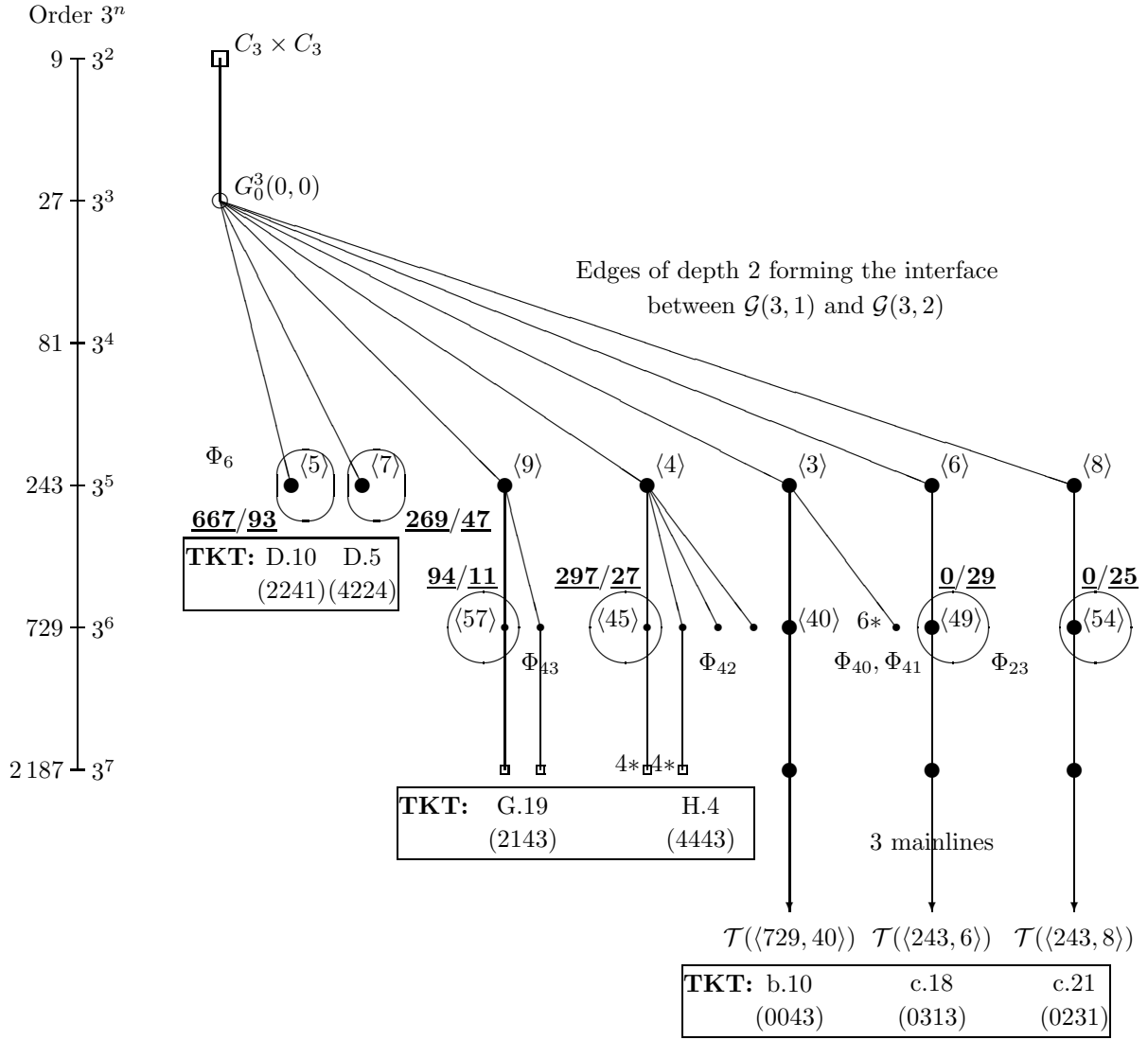

\noindent
The vertices of the coclass graph \(\mathcal{G}(3,2)\) in Figure
\ref{fig:Typ33Cocl2}
are classified by using different symbols:

\begin{enumerate}
\item
a large contour square \(\square\) represents an abelian group,
\item
a big contour circle {\Large \(\circ\)} represents a metabelian group containing an abelian maximal subgroup,
all other metabelian groups do not possess abelian subgroups,
\item
big full discs {\Large \(\bullet\)} represent metabelian groups with bicyclic centre of type \((3,3)\) and defect \(k=0\),
\item
small full discs {\tiny \(\bullet\)} represent metabelian groups with cyclic centre of order \(3\) and defect \(k=1\),
\item
small contour squares {\tiny \(\square\)} represent non-metabelian groups.
\end{enumerate}

\noindent
Groups of particular importance are labelled by a number in angles.
This is the identifier in the SmallGroups library
\cite{BEO}
of GAP
\cite{GAP}
and
\cite{MAGMA},
where we omit the order, which is given on the left hand scale.

\noindent
The actual distribution of the
\(2020\), resp. \(2576\), second \(3\)-class groups \(G_3^2(K)\)
of complex, resp. real, quadratic number fields \(K=\mathbb{Q}(\sqrt{D})\) of type \((3,3)\)
with discriminant \(-10^6<D<10^7\) is represented by
underlined boldface counters (in the format complex/real)
of the hits of vertices surrounded by the adjacent oval.

It is illuminating to compare these frequencies,
which we have computed in
\cite[\S\ 6, Tbl. 3--5]{Ma1}
and
\cite[\S\ 6, Tbl. 13--15,17]{Ma3}
with the non-abelian generalization of the asymptotic Cohen-Lenstra-Martinet probability,
which is given for complex quadratic fields by
\cite[p. 18, Tbl. 2]{BBH}
with respect to all \(3190\) discriminants of \(3\)-rank \(2\)
instead to the \(2020\) discriminants of type \((3,3)\)
in the range \(-10^6<D<0\).
In three cases of Table
\ref{tbl:NonAbCohenLenstra},
the actual percentage exceeds the conjectural asymptotic probability.
In the first case the excess is significant.
A possible interpretation is that the population of vertices of higher order
will become more probable in ranges of considerably bigger absolute values of discriminants
so that the percentage of hits of the low order vertices in the table will decrease.

\renewcommand{\arraystretch}{1.0}

\begin{table}[ht]
\caption{Asymptotic probability for densely populated vertices of \(\mathcal{G}(3,2)\)}
\label{tbl:NonAbCohenLenstra}
\begin{center}
\begin{tabular}{|c|cc|c|}
\hline
 Id of \(G\)               & frequency & percentage & probability \\
\hline
 \(\langle 243,5\rangle\)  & \(667\)   &  \(20.91\) & \(17.56\)   \\
 \(\langle 243,7\rangle\)  & \(269\)   &   \(8.43\) &  \(8.78\)   \\
 \(\langle 729,57\rangle\) &  \(94\)   &   \(2.95\) &  \(2.19\)   \\
 \(\langle 729,45\rangle\) & \(297\)   &   \(9.31\) &  \(8.78\)   \\
\hline
\end{tabular}
\end{center}
\end{table}

Identification of the vertex
\(\langle 729,57\rangle\),
resp.
\(\langle 729,45\rangle\),
among two, resp. four, closely related vertices
in isoclinism family \(\Phi_{43}\), resp. \(\Phi_{42}\),
was possible by means of the following \textit{Artin criterion}
for second \(p\)-class groups of quadratic base fields,
which can be verified by testing for
a suitable automorphism of order \(2\).

\begin{theorem}
Let \(p\ge 3\) be an odd prime.
The second \(p\)-class group \(G=\mathrm{G}_p^2(K)\)
of a quadratic field \(K=\mathbb{Q}(\sqrt{D})\)
admits an extension by the cyclic group \(C_2\),
\(1\to G\to H\to C_2\to 1\),
such that \(H^\prime\simeq G\).
\end{theorem}

\begin{proof}
See the letter of E. Artin to H. Hasse from November 19, 1928
\cite{FRL}.
\end{proof}



\subsubsection{Coclass trees of type \((3,3)\) and \((9,3)\) on \(\mathcal{G}(3,2)\)}
\label{sss:Trees33Cocl2}

\begin{definition}
\label{d:ForbTrees}
Let \(p\ge 3\) be an odd prime.
A rooted subtree \(\mathcal{T}\) of a coclass graph \(\mathcal{G}(p,r)\), \(r\ge 1\),
is called \textit{forbidden},
if none of its vertices \(G\in\mathcal{T}\) can be realized as
the second \(p\)-class group \(\mathrm{G}_p^2(K)\) of a quadratic field \(K=\mathbb{Q}(\sqrt{D})\).
Otherwise \(\mathcal{T}\) is called \textit{admissible}.
\end{definition}

\begin{theorem}
\label{thm:ForbTrees}
Let \(\mathcal{T}\) be a coclass tree of a coclass graph \(\mathcal{G}(3,r)\), \(r\ge 1\),
such that all its mainline groups \(M\) are metabelian with abelianization of type \((3,3)\).

\begin{enumerate}
\item
The unique tree of \(\mathcal{G}(3,1)\)
and the trees of \(\mathcal{G}(3,2)\),
whose mainline groups \(M\)
are of transfer kernel type
either \(\mathrm{c}.18\), \(\varkappa=(0313)\),
or \(\mathrm{c}.21\), \(\varkappa=(0231)\),
are admissible.
\item
If all mainline groups \(M\)
are of transfer kernel type \(\mathrm{b}.10\), \(\varkappa=(0043)\),
with distinguished second member \(\varkappa(2)=0\),
then
\(\mathcal{T}\) is forbidden if and only if the coclass \(r\ge 2\) is even.
\item
If all mainline groups \(M\)
are of transfer kernel type
either \(\mathrm{d}^\ast.23\), \(\varkappa=(0243)\),
or \(\mathrm{d}^\ast.19\), \(\varkappa=(0443)\),
or \(\mathrm{d}^\ast.25\), \(\varkappa=(0143)\),
with distinguished second member \(\varkappa(2)\ne 0\),
then
\(\mathcal{T}\) is forbidden if and only if the coclass \(r\ge 3\) is odd.
\end{enumerate}

\end{theorem}

Referring to
\cite[p. 189 ff]{Ne1}
we point out the following details.

\begin{enumerate}
\item
There is a periodic pattern of period length \(2\)
of rooted subtrees with metabelian mainlines of type \((3,3)\)
among the coclass graphs \(\mathcal{G}(p,r)\), setting in with \(r=3\).
The roots of the trees with fixed coclass \(r\)
are of order \(3^{2r+1}\).
The metabelian skeletons of the trees
with common transfer kernel type and coclass of the same parity
are isomorphic as graphs.
The same is true for the metabelian skeletons of sporadic groups
with coclass of the same parity.
\item
For odd coclass  \(r\ge 3\),
there are \(4\) trees with mainlines of transfer kernel types
\(\mathrm{b}.10\), \(\varkappa=(0043)\), \(\mathrm{d}^\ast.23\), \(\varkappa=(0243)\),
\(\mathrm{d}^\ast.19\), \(\varkappa=(0443)\), and \(\mathrm{d}^\ast.25\), \(\varkappa=(0143)\).
\item
For even coclass  \(r\ge 4\),
there are \(6\) trees with mainlines of transfer kernel types
\(\mathrm{b}.10\), \(\varkappa=(0043)\), and \(\mathrm{d}^\ast.23\), \(\varkappa=(0243)\),
each occurring only once,
and on the other hand \(\mathrm{d}^\ast.19\), \(\varkappa=(0443)\), and \(\mathrm{d}^\ast.25\), \(\varkappa=(0143)\),
each occurring in two instances, isomorphic as graphs.
\end{enumerate}

\begin{proof}
Theorem
\ref{thm:ForbTrees}
is an immediate consequence of
Theorem
\ref{thm:SelRuleCoclGe2},
due to the second distinguished member \(\varkappa(2)\)
of the TKT.
See the diagrams on the pages without numbers, following
\cite[p. 189]{Ne1}.
These diagrams were constructed by means of the
lists of representatives \(G_\rho^{m,n}(\alpha,\beta,\gamma,\delta)\) for isomorphism classes,
given in the appendix
\cite{Ne2}
of Nebelung's thesis.
The connection with the transfer kernel types \(\varkappa\)
is established in
\cite[Thm. 6.14, pp. 208 ff]{Ne1}.
\end{proof}

\renewcommand{\arraystretch}{1.0}

\begin{table}[ht]
\caption{Infinite trees of types \((3,3)\), \((9,3)\) with metabelian mainline on \(\mathcal{G}(3,2)\)}
\label{tbl:TreesBQUAG}
\begin{center}
\begin{tabular}{|cc|cc|cc|c|c|}
\hline
 SmallGroups Id of root \(G\) & \(G/G^\prime\) & TKT resp. pTKT    &\(\varkappa\) & \(\varepsilon\) & \(\eta\in\)            & \((f,g,h)\) & population \\
\hline
 \(\langle 729,40\rangle\)    & \((3,3)\)      & \(\mathrm{b}.10\) & \((0043)\)   & \(2\)           & \(\lbrace 1,2\rbrace\) & \((0,1,0)\) & forbidden  \\
 \(\langle 243,6\rangle\)     & \((3,3)\)      & \(\mathrm{c}.18\) & \((0313)\)   & \(1\)           & \(\lbrace 0,1\rbrace\) & \((0,1,2)\) & admissible \\
 \(\langle 243,8\rangle\)     & \((3,3)\)      & \(\mathrm{c}.21\) & \((0231)\)   & \(0\)           & \(\lbrace 2,3\rbrace\) & \((1,1,2)\) & admissible \\
\hline
 \(\langle 243,17\rangle\)    & \((9,3)\)      & \(\mathrm{a}.1\)  & \((000;0)\)  & \(2\)           & \(\lbrace 3,4\rbrace\) & \((0,0,1)\) & admissible \\
 \(\langle 243,15\rangle\)    & \((9,3)\)      & \(\mathrm{a}.1\)  & \((000;0)\)  & \(1\)           & \(\lbrace 3,4\rbrace\) & \((1,0,0)\) & admissible \\
\hline
\end{tabular}
\end{center}
\end{table}

Aside from the single forbidden coclass tree
with metabelian mainline of transfer kernel type \(\mathrm{b}.10\), \(\varkappa=(0043)\),
there exist \(4\) admissible coclass trees with metabelian mainline on \(\mathcal{G}(3,2)\)
which are populated quite densely by second \(3\)-class groups \(\mathrm{G}_3^2(K)\)
of quadratic fields \(K\) with \(3\)-class group of type \((3,3)\), resp. \((9,3)\).
They can be characterized by the number of members of the TTT with \(3\)-rank bigger than \(2\),
\(\varepsilon=\#\lbrace 1\le i\le 4\mid\mathrm{r}_3(H_i/H_i^\prime)\ge 3\rbrace\),
or by the number of members of the TKT, resp. punctured TKT (pTKT), having Taussky's type A,
\(\eta=\#\lbrace 1\le i\le 4\mid\kappa(i)=\mathrm{A}\rbrace\),
as shown in Table
\ref{tbl:TreesBQUAG}.
Groups along their mainlines arise as quotients of
infinite pro-\(3\)-groups of coclass \(2\) having a non-trivial centre,
whose pro-\(3\) presentations are defined
by suitable triplets \((f,g,h)\) of relational exponents in
\cite[Thm. 4.1]{ELNO}.



\begin{figure}[ht]
\caption{Distribution of \(G_3^2(K)\) on the coclass \(2\) tree \(\mathcal{T}(\langle 243,6\rangle)\)}
\label{fig:TreeQTyp33Cocl2}

\setlength{\unitlength}{1cm}
\begin{picture}(14,14.5)(-6,-13.5)

\put(-5,0.5){\makebox(0,0)[cb]{Order \(3^n\)}}
\put(-5,0){\line(0,-1){12}}
\multiput(-5.1,0)(0,-2){7}{\line(1,0){0.2}}
\put(-5.2,0){\makebox(0,0)[rc]{\(243\)}}
\put(-4.8,0){\makebox(0,0)[lc]{\(3^5\)}}
\put(-5.2,-2){\makebox(0,0)[rc]{\(729\)}}
\put(-4.8,-2){\makebox(0,0)[lc]{\(3^6\)}}
\put(-5.2,-4){\makebox(0,0)[rc]{\(2\,187\)}}
\put(-4.8,-4){\makebox(0,0)[lc]{\(3^7\)}}
\put(-5.2,-6){\makebox(0,0)[rc]{\(6\,561\)}}
\put(-4.8,-6){\makebox(0,0)[lc]{\(3^8\)}}
\put(-5.2,-8){\makebox(0,0)[rc]{\(19\,683\)}}
\put(-4.8,-8){\makebox(0,0)[lc]{\(3^9\)}}
\put(-5.2,-10){\makebox(0,0)[rc]{\(59\,049\)}}
\put(-4.8,-10){\makebox(0,0)[lc]{\(3^{10}\)}}
\put(-5.2,-12){\makebox(0,0)[rc]{\(177\,147\)}}
\put(-4.8,-12){\makebox(0,0)[lc]{\(3^{11}\)}}

\multiput(0,0)(0,-2){6}{\circle*{0.2}}
\multiput(0,0)(0,-2){5}{\line(0,-1){2}}
\multiput(-1,-2)(0,-2){6}{\circle*{0.2}}
\multiput(-2,-2)(0,-2){6}{\circle*{0.2}}
\multiput(1.95,-4.05)(0,-2){5}{\framebox(0.1,0.1){}}
\multiput(3,-2)(0,-2){6}{\circle*{0.2}}
\multiput(0,0)(0,-2){6}{\line(-1,-2){1}}
\multiput(0,0)(0,-2){6}{\line(-1,-1){2}}
\multiput(0,-2)(0,-2){5}{\line(1,-1){2}}
\multiput(0,0)(0,-2){6}{\line(3,-2){3}}
\multiput(-3.05,-4.05)(-1,0){2}{\framebox(0.1,0.1){}}
\multiput(3.95,-4.05)(0,-2){5}{\framebox(0.1,0.1){}}
\multiput(5,-4)(0,-2){5}{\circle*{0.1}}
\multiput(6,-4)(0,-2){5}{\circle*{0.1}}
\multiput(-1,-2)(-1,0){2}{\line(-1,-1){2}}
\multiput(3,-2)(0,-2){5}{\line(1,-2){1}}
\multiput(3,-2)(0,-2){5}{\line(1,-1){2}}
\multiput(3,-2)(0,-2){5}{\line(3,-2){3}}
\multiput(6.95,-6.05)(0,-2){4}{\framebox(0.1,0.1){}}
\multiput(6,-4)(0,-2){4}{\line(1,-2){1}}

\put(2,-1){\makebox(0,0)[lc]{\(\mathcal{B}_5\)}}
\put(2,-3){\makebox(0,0)[lc]{\(\mathcal{B}_6\)}}
\put(2,-5){\makebox(0,0)[lc]{\(\mathcal{B}_7\)}}
\put(2,-7){\makebox(0,0)[lc]{\(\mathcal{B}_8\)}}

\put(-0.1,0.7){\makebox(0,0)[rc]{\(\langle 6\rangle\)}}
\put(-2.1,-1.8){\makebox(0,0)[rc]{\(\langle 50\rangle\)}}
\put(-1.1,-1.8){\makebox(0,0)[rc]{\(\langle 51\rangle\)}}
\put(0.1,-1.3){\makebox(0,0)[lc]{\(\langle 49\rangle\)}}
\put(3.1,-1.8){\makebox(0,0)[lc]{\(\langle 48\rangle\)}}
\put(-4.1,-3.6){\makebox(0,0)[cc]{\(\langle 292\rangle\)}}
\put(-3.1,-3.6){\makebox(0,0)[cc]{\(\langle 293\rangle\)}}
\put(-2.1,-2.6){\makebox(0,0)[cc]{\(\langle 289\rangle\)}}
\put(-2.1,-3.1){\makebox(0,0)[cc]{\(\langle 290\rangle\)}}
\put(-1.1,-3.1){\makebox(0,0)[cc]{\(\langle 288\rangle\)}}
\put(0.1,-3.5){\makebox(0,0)[lc]{\(\langle 285\rangle\)}}
\put(3,-3){\makebox(0,0)[cc]{\(\langle 286\rangle\)}}
\put(3,-3.5){\makebox(0,0)[cc]{\(\langle 287\rangle\)}}
\put(2.1,-3.8){\makebox(0,0)[lc]{\(*2\)}}
\multiput(-2.1,-3.8)(0,-4){3}{\makebox(0,0)[rc]{\(2*\)}}
\multiput(3.1,-3.8)(0,-4){3}{\makebox(0,0)[lc]{\(*2\)}}
\put(4.1,-3.8){\makebox(0,0)[lc]{\(*2\)}}
\multiput(5.1,-5.8)(0,-4){2}{\makebox(0,0)[lc]{\(*2\)}}
\multiput(6.1,-5.8)(0,-4){2}{\makebox(0,0)[lc]{\(*2\)}}
\multiput(5.1,-3.8)(0,-4){3}{\makebox(0,0)[lc]{\(*3\)}}
\multiput(6.1,-3.8)(0,-4){3}{\makebox(0,0)[lc]{\(*3\)}}
\multiput(7.1,-5.8)(0,-2){4}{\makebox(0,0)[lc]{\(*3\)}}
\put(-3,-13){\makebox(0,0)[cc]{\textbf{TKT:}}}
\put(-2,-13){\makebox(0,0)[cc]{E.14}}
\put(-1,-13){\makebox(0,0)[cc]{E.6}}
\put(0,-13){\makebox(0,0)[cc]{c.18}}
\put(3,-13){\makebox(0,0)[cc]{H.4}}
\put(5,-13){\makebox(0,0)[cc]{H.4}}
\put(6,-13){\makebox(0,0)[cc]{H.4}}
\put(-2,-13.5){\makebox(0,0)[cc]{\((2313)\)}}
\put(-1,-13.5){\makebox(0,0)[cc]{\((1313)\)}}
\put(0,-13.5){\makebox(0,0)[cc]{\((0313)\)}}
\put(3,-13.5){\makebox(0,0)[cc]{\((3313)\)}}
\put(5,-13.5){\makebox(0,0)[cc]{\((3313)\)}}
\put(6,-13.5){\makebox(0,0)[cc]{\((3313)\)}}
\put(-3.8,-13.7){\framebox(10.6,1){}}

\put(0,-10){\vector(0,-1){2}}
\put(0.2,-11.5){\makebox(0,0)[lc]{main}}
\put(0.2,-12){\makebox(0,0)[lc]{line}}
\put(1.8,-12.5){\makebox(0,0)[rc]{\(\mathcal{T}(\langle 243,6\rangle)\)}}

\multiput(-1.5,-4)(0,-4){2}{\oval(2,1)}
\put(-1.5,-4.8){\makebox(0,0)[cc]{\underbar{\textbf{186}}/\underbar{\textbf{7}}}}
\put(-1.5,-8.8){\makebox(0,0)[cc]{\underbar{\textbf{15}}/\underbar{\textbf{0}}}}
\multiput(5.5,-6)(0,-4){2}{\oval(2,1)}
\put(5.5,-6.8){\makebox(0,0)[cc]{\underbar{\textbf{63}}/\underbar{\textbf{3}}}}
\put(5.5,-10.8){\makebox(0,0)[cc]{\underbar{\textbf{6}}/\underbar{\textbf{0}}}}
\put(0.2,-2){\oval(1.5,1)}
\put(1.1,-2.4){\makebox(0,0)[lc]{\underbar{\textbf{0}}/\underbar{\textbf{29}}}}

\end{picture}

\end{figure}

\begin{theorem}
\label{t:TreeQ}
The structure of the complete coclass tree \(\mathcal{T}(\langle 243,6\rangle)\)
as part of the coclass graph \(\mathcal{G}(3,2)\),
restricted to \(3\)-groups \(G\) with abelianization \(G/G^\prime\simeq (3,3)\),
is globally characterized by the tree invariant \(\varepsilon(G)=1\)
and given up to order \(3^{11}=177\,147\) by Figure
\ref{fig:TreeQTyp33Cocl2}.
The branches are of depth \(3\) and periodic of length \(2\).
The pre-period consists of \(\mathcal{B}_5,\mathcal{B}_6\),
the primitive period of \(\mathcal{B}_7,\mathcal{B}_8\)

\end{theorem}

\noindent
In Figure
\ref{fig:TreeQTyp33Cocl2},
we have
\(G_3^2(\mathbb{Q}(\sqrt{D}))\in\mathcal{T}(\langle 243,6\rangle)\)
for \(270\) \((13.4\%)\)
of the \(2020\) discriminants \(-10^6<D<0\)
and for \(39\) \((1.5\%)\)
of the \(2576\) discriminants \(0<D<10^7\),
investigated in
\cite[\S\ 6]{Ma1},
\cite[\S\ 6]{Ma3}.\\
Since the TKT \(\mathrm{c}.18\), \(\varkappa=(0313)\),
of the mainline is \textit{total} with \(\varkappa(1)=0\),
there only occur \(G_3^2(K)\) of \textit{real} quadratic fields
\(K=\mathbb{Q}(\sqrt{D})\), \(D>0\),
on the mainline.\\
Due to the \textit{Selection Rule} in Theorem
\ref{thm:SelRuleCoclGe2},
the \(G_3^2(K)\) are distributed on \textit{even branches} only,
since the second distinguished transfer kernel \(\varkappa(2)\ne 0\).\\
Underpinning the weak leaf conjecture,
there is no actual hit of the vertices at depth \(1\)
with TKT \(\mathrm{H}.4\), \(\varkappa=(3313)\).



\begin{figure}[ht]
\caption{Distribution of \(G_3^2(K)\) on the coclass \(2\) tree \(\mathcal{T}(\langle 243,8\rangle)\)}
\label{fig:TreeUTyp33Cocl2}

\setlength{\unitlength}{1cm}
\begin{picture}(14,14.5)(-6,-13.5)

\put(-5,0.5){\makebox(0,0)[cb]{Order \(3^n\)}}
\put(-5,0){\line(0,-1){12}}
\multiput(-5.1,0)(0,-2){7}{\line(1,0){0.2}}
\put(-5.2,0){\makebox(0,0)[rc]{\(243\)}}
\put(-4.8,0){\makebox(0,0)[lc]{\(3^5\)}}
\put(-5.2,-2){\makebox(0,0)[rc]{\(729\)}}
\put(-4.8,-2){\makebox(0,0)[lc]{\(3^6\)}}
\put(-5.2,-4){\makebox(0,0)[rc]{\(2\,187\)}}
\put(-4.8,-4){\makebox(0,0)[lc]{\(3^7\)}}
\put(-5.2,-6){\makebox(0,0)[rc]{\(6\,561\)}}
\put(-4.8,-6){\makebox(0,0)[lc]{\(3^8\)}}
\put(-5.2,-8){\makebox(0,0)[rc]{\(19\,683\)}}
\put(-4.8,-8){\makebox(0,0)[lc]{\(3^9\)}}
\put(-5.2,-10){\makebox(0,0)[rc]{\(59\,049\)}}
\put(-4.8,-10){\makebox(0,0)[lc]{\(3^{10}\)}}
\put(-5.2,-12){\makebox(0,0)[rc]{\(177\,147\)}}
\put(-4.8,-12){\makebox(0,0)[lc]{\(3^{11}\)}}

\multiput(0,0)(0,-2){6}{\circle*{0.2}}
\multiput(0,0)(0,-2){5}{\line(0,-1){2}}
\multiput(-1,-2)(0,-2){6}{\circle*{0.2}}
\multiput(-2,-2)(0,-2){6}{\circle*{0.2}}
\multiput(1.95,-4.05)(0,-2){5}{\framebox(0.1,0.1){}}
\multiput(3,-2)(0,-2){6}{\circle*{0.2}}
\multiput(0,0)(0,-2){6}{\line(-1,-2){1}}
\multiput(0,0)(0,-2){6}{\line(-1,-1){2}}
\multiput(0,-2)(0,-2){5}{\line(1,-1){2}}
\multiput(0,0)(0,-2){6}{\line(3,-2){3}}
\multiput(-3.05,-4.05)(-1,0){2}{\framebox(0.1,0.1){}}
\multiput(3.95,-6.05)(0,-2){4}{\framebox(0.1,0.1){}}
\multiput(5,-6)(0,-2){4}{\circle*{0.1}}
\multiput(6,-4)(0,-2){5}{\circle*{0.1}}
\multiput(-1,-2)(-1,0){2}{\line(-1,-1){2}}
\multiput(3,-4)(0,-2){4}{\line(1,-2){1}}
\multiput(3,-4)(0,-2){4}{\line(1,-1){2}}
\multiput(3,-2)(0,-2){5}{\line(3,-2){3}}
\multiput(6.95,-6.05)(0,-2){4}{\framebox(0.1,0.1){}}
\multiput(6,-4)(0,-2){4}{\line(1,-2){1}}

\put(2,-1){\makebox(0,0)[lc]{\(\mathcal{B}_5\)}}
\put(2,-3){\makebox(0,0)[lc]{\(\mathcal{B}_6\)}}
\put(2,-5){\makebox(0,0)[lc]{\(\mathcal{B}_7\)}}
\put(2,-7){\makebox(0,0)[lc]{\(\mathcal{B}_8\)}}

\put(-0.1,0.7){\makebox(0,0)[rc]{\(\langle 8\rangle\)}}
\put(-2.1,-1.8){\makebox(0,0)[rc]{\(\langle 53\rangle\)}}
\put(-1.1,-1.8){\makebox(0,0)[rc]{\(\langle 55\rangle\)}}
\put(0.1,-1.3){\makebox(0,0)[lc]{\(\langle 54\rangle\)}}
\put(3.1,-1.8){\makebox(0,0)[lc]{\(\langle 52\rangle\)}}
\put(-4.1,-3.6){\makebox(0,0)[cc]{\(\langle 300\rangle\)}}
\put(-3.1,-3.6){\makebox(0,0)[cc]{\(\langle 309\rangle\)}}
\put(-2.1,-2.6){\makebox(0,0)[cc]{\(\langle 302\rangle\)}}
\put(-2.1,-3.1){\makebox(0,0)[cc]{\(\langle 306\rangle\)}}
\put(-1.1,-3.1){\makebox(0,0)[cc]{\(\langle 304\rangle\)}}
\put(0.1,-3.5){\makebox(0,0)[lc]{\(\langle 303\rangle\)}}
\put(3,-3){\makebox(0,0)[cc]{\(\langle 301\rangle\)}}
\put(3,-3.5){\makebox(0,0)[cc]{\(\langle 305\rangle\)}}
\put(2.1,-3.8){\makebox(0,0)[lc]{\(*2\)}}
\multiput(-2.1,-3.8)(0,-4){3}{\makebox(0,0)[rc]{\(2*\)}}
\multiput(3.1,-3.8)(0,-4){3}{\makebox(0,0)[lc]{\(*2\)}}
\put(6.1,-3.8){\makebox(0,0)[lc]{\(*6\)}}
\multiput(5.1,-5.8)(0,-4){2}{\makebox(0,0)[lc]{\(*2\)}}
\multiput(6.1,-5.8)(0,-4){2}{\makebox(0,0)[lc]{\(*2\)}}
\multiput(5.1,-7.8)(0,-4){2}{\makebox(0,0)[lc]{\(*3\)}}
\multiput(6.1,-7.8)(0,-4){2}{\makebox(0,0)[lc]{\(*3\)}}
\multiput(7.1,-7.8)(0,-2){3}{\makebox(0,0)[lc]{\(*2\)}}
\put(-3,-13){\makebox(0,0)[cc]{\textbf{TKT:}}}
\put(-2,-13){\makebox(0,0)[cc]{E.9}}
\put(-1,-13){\makebox(0,0)[cc]{E.8}}
\put(0,-13){\makebox(0,0)[cc]{c.21}}
\put(3,-13){\makebox(0,0)[cc]{G.16}}
\put(5,-13){\makebox(0,0)[cc]{G.16}}
\put(6,-13){\makebox(0,0)[cc]{G.16}}
\put(-2,-13.5){\makebox(0,0)[cc]{\((2231)\)}}
\put(-1,-13.5){\makebox(0,0)[cc]{\((1231)\)}}
\put(0,-13.5){\makebox(0,0)[cc]{\((0231)\)}}
\put(3,-13.5){\makebox(0,0)[cc]{\((4231)\)}}
\put(5,-13.5){\makebox(0,0)[cc]{\((4231)\)}}
\put(6,-13.5){\makebox(0,0)[cc]{\((4231)\)}}
\put(-3.8,-13.7){\framebox(10.6,1){}}

\put(0,-10){\vector(0,-1){2}}
\put(0.2,-11.5){\makebox(0,0)[lc]{main}}
\put(0.2,-12){\makebox(0,0)[lc]{line}}
\put(1.8,-12.5){\makebox(0,0)[rc]{\(\mathcal{T}(\langle 243,8\rangle)\)}}

\multiput(-1.5,-4)(0,-4){2}{\oval(2,1)}
\put(-1.5,-4.8){\makebox(0,0)[cc]{\underbar{\textbf{197}}/\underbar{\textbf{14}}}}
\put(-1.5,-8.8){\makebox(0,0)[cc]{\underbar{\textbf{13}}/\underbar{\textbf{0}}}}
\multiput(5.5,-6)(0,-4){2}{\oval(2,1)}
\put(5.5,-6.8){\makebox(0,0)[cc]{\underbar{\textbf{79}}/\underbar{\textbf{2}}}}
\put(5.5,-10.8){\makebox(0,0)[cc]{\underbar{\textbf{2}}/\underbar{\textbf{0}}}}
\multiput(0.2,-2)(0,-4){2}{\oval(1.5,1)}
\put(1.1,-2.4){\makebox(0,0)[lc]{\underbar{\textbf{0}}/\underbar{\textbf{25}}}}
\put(1.1,-6.4){\makebox(0,0)[lc]{\underbar{\textbf{0}}/\underbar{\textbf{2}}}}

\end{picture}

\end{figure}

\begin{theorem}
\label{t:TreeU}
The structure of the complete coclass tree \(\mathcal{T}(\langle 243,8\rangle)\)
as part of the coclass graph \(\mathcal{G}(3,2)\),
restricted to \(3\)-groups \(G\) with abelianization \(G/G^\prime\simeq (3,3)\),
is globally characterized by \(\varepsilon(G)=0\)
and given up to order \(3^{11}=177\,147\) by Figure
\ref{fig:TreeUTyp33Cocl2}.
The branches are of depth \(3\) and periodic of length \(2\).
The pre-period consists of \(\mathcal{B}_5,\mathcal{B}_6\),
the primitive period of \(\mathcal{B}_7,\mathcal{B}_8\)
\end{theorem}

\noindent
In Figure
\ref{fig:TreeUTyp33Cocl2},
we have
\(G_3^2(\mathbb{Q}(\sqrt{D}))\in\mathcal{T}(\langle 243,8\rangle)\)
for \(291\) \((14.4\%)\)
of the \(2020\) discriminants \(-10^6<D<0\)
and for \(43\) \((1.7\%)\)
of the \(2576\) discriminants \(0<D<10^7\),
investigated in
\cite[\S\ 6]{Ma1},
\cite[\S\ 6]{Ma3}.\\
Since the TKT \(\mathrm{c}.21\), \(\varkappa=(0231)\),
of the mainline is \textit{total} with \(\varkappa(1)=0\),
there only occur \(G_3^2(K)\) of \textit{real} quadratic fields
\(K=\mathbb{Q}(\sqrt{D})\), \(D>0\), on the mainline.\\
Due to the \textit{Selection Rule} in Theorem
\ref{thm:SelRuleCoclGe2},
the \(G_3^2(K)\) are distributed on \textit{even branches} only,
since the second distinguished transfer kernel \(\varkappa(2)\ne 0\).\\
Underpinning the weak leaf conjecture,
there is no actual hit of the vertices at depth \(1\)
with TKT \(\mathrm{G}.16\), \(\varkappa=(4231)\).

The vertices of the coclass trees in both Figures
\ref{fig:TreeQTyp33Cocl2}
and
\ref{fig:TreeUTyp33Cocl2}
are classified by using different symbols:

\begin{enumerate}
\item
big full discs {\Large \(\bullet\)} represent metabelian groups with bicyclic centre of type \((3,3)\) and defect \(k=0\),
\item
small full discs {\footnotesize \(\bullet\)} represent metabelian groups with cyclic centre of order \(3\) and defect \(k=1\),
\item
small contour squares {\scriptsize \(\square\)} represent non-metabelian groups.
\end{enumerate}

\noindent
A number adjacent to a vertex denotes the multiplicity of a batch
of immediate descendants sharing a common parent.
The groups of particular importance are labelled by a number in angles,
which is the identifier in the SmallGroups library
\cite{BEO}
of GAP
\cite{GAP}
and MAGMA
\cite{MAGMA}.
The metabelian skeletons were drawn in
\cite[p. 189 ff]{Ne1},
the complete trees were given in
\cite[p. 76, Fig. 4.8 and p. 123, Fig. 6.1]{As1}.

\noindent
The actual distribution of the
\(2020\), resp. \(2576\), second \(3\)-class groups \(G_3^2(K)\)
of complex, resp. real, quadratic number fields \(K=\mathbb{Q}(\sqrt{D})\) of type \((3,3)\)
with discriminant \(-10^6<D<10^7\) is represented by
underlined boldface counters (in the format complex/real)
of the hits of vertices surrounded by the adjacent oval.
See
\cite[\S\ 6, tbl. 3--5]{Ma1}
and
\cite[\S\ 6, tbl. 15--18]{Ma3}.\\
The realization of mainline vertices
with TKT \(\mathrm{c}.18\) and \(\mathrm{c}.21\)
as \(G_3^2(K)\)
is no violation of the weak leaf conjecture
\ref{cnj:WeakLeafCnj},
since these vertices do not possess metabelian immediate descendants
of the same TKT.

When we had completed our extensive investigation of
second \(3\)-class groups \(\mathrm{G}_3^2(K)\)
of all \(4596\) quadratic fields \(K=\mathbb{Q}(\sqrt{D})\) of type \((3,3)\)
in the range \(-10^6<D<10^7\),
we wondered whether the distribution of second \(3\)-class groups \(\mathrm{G}_3^2(K)\)
for other sequences of base fields \(K\) of type \((3,3)\) shows similarities or differences.

Since fields of degree \(4\) are still within the reach of numerical computations,
we are able to present the results for bicyclic biquadratic fields
of Gauss-Dirichlet-Hilbert type \(K=\mathbb{Q}\left(\sqrt{\strut -1},\sqrt{\strut d}\right)\)
\cite{Hi}
in section \S\
\ref{ss:NewRsltGDH}.
These fields reveal strong similarities to quadratic fields.
In section \S\
\ref{ss:NewRsltESR},
however, we show that bicyclic biquadratic fields
of Eisenstein-Scholz-Reichardt type \(K=\mathbb{Q}\left(\sqrt{\strut -3},\sqrt{\strut d}\right)\)
\cite{So,Re}
exhibit a completely different behavior.



\renewcommand{\arraystretch}{1.1}

\begin{table}[ht]
\caption{\(11\) inherited variants of \(G=\mathrm{G}_3^2(K)\) for \(K=\mathbb{Q}\left(\sqrt{\strut -1},\sqrt{\strut d}\right)\)}
\label{tbl:LiftedVariants}
\begin{center}
\begin{tabular}{|r|r||c||c|c||c|c|c|}
\hline
  \(d\)   &\(d^\prime\)& \(\tau(G)=(\mathrm{Cl}_3(L_i))_{1\le i\le 4}\) & Type  & \(\varkappa(G)\) & \(G\)    & \(\mathrm{cc}(G)\) & \(\Phi\) \\
\hline
 \(3896\) &  \(-3896\) & \((3,3,3),(3,3,3),(9,3),(3,3,3)\) & H.4                & \((4443)\) & \(\langle 729,45\rangle\)   & \(2\) & \(\Phi_{42}\) \\
 \(5069\) & \(-20276\) & \((3,3,3),(9,3),(3,3,3),(9,3)\)   & D.5                & \((4224)\) & \(\langle 243,7\rangle\)    & \(2\) & \(\Phi_{6}\) \\
\(10173\) & \(-40692\) & \((3,3,3),(9,3),(9,3),(9,3)\)     & D.10               & \((2241)\) & \(\langle 243,5\rangle\)    & \(2\) & \(\Phi_{6}\) \\
\(12481\) & \(-49924\) & \((9,3),(9,3),(9,3),(9,3)\)       & G.19               & \((2143)\) & \(\langle 729,57\rangle\)   & \(2\) & \(\Phi_{43}\) \\
\hline
 \(2437\) &  \(-9748\) & \((27,9),(9,3),(9,3),(9,3)\)      & E.9                & \((2231)\) & \(G_0^{6,7}(0,0,\pm 1,1)\)  & \(2\) & \\
 \(5417\) & \(-21668\) & \((27,9),(9,3),(3,3,3),(9,3)\)    & H.4\(\uparrow\)    & \((3313)\) & \(G_{\pm 1}^{7,8}(0,-1,\pm 1,1)\) & \(2\) & \\
 \(6221\) & \(-24884\) & \((27,9),(9,3),(9,3),(9,3)\)      & G.16               & \((4231)\) & \(G_1^{7,8}(0,0,\pm 1,1)\)        & \(2\) & \\
\(12837\) & \(-51348\) & \((27,9),(9,3),(3,3,3),(9,3)\)    & E.14               & \((2313)\) & \(G_0^{6,7}(0,-1,\pm 1,1)\) & \(2\) & \\
\(15544\) & \(-15544\) & \((27,9),(9,3),(3,3,3),(9,3)\)    & E.6                & \((1313)\) & \(G_0^{6,7}(1,-1,1,1)\)     & \(2\) & \\
\hline
 \(6789\) & \(-27156\) & \((27,9),(27,9),(3,3,3),(3,3,3)\) & F.11               & \((1143)\) & \(G_0^{6,9}(0,0,\pm 1,1)\)     & \(4\) & \\
 \(7977\) & \(-31908\) & \((27,9),(27,9),(3,3,3),(3,3,3)\) & F.12               & \((1343)\) & \(G_0^{6,9}(\mp 1,0,\pm 1,1)\) & \(4\) & \\
\hline
\end{tabular}
\end{center}
\end{table}

\renewcommand{\arraystretch}{1.1}

\begin{table}[ht]
\caption{\(11\) genuine variants of \(G=\mathrm{G}_3^2(K)\) for \(K=\mathbb{Q}\left(\sqrt{\strut -1},\sqrt{\strut d}\right)\)}
\label{tbl:GenuineVariants}
\begin{center}
\begin{tabular}{|r||c||c|c||c|c|c|}
\hline
  \(d\)   & \(\tau(G)=(\mathrm{Cl}_3(L_i))_{1\le i\le 4}\) & Type  & \(\varkappa(G)\) & \(G\)    & \(\mathrm{cc}(G)\) & \(\Phi\) \\
\hline
  \(473\) & \((9,3),(3,3),(3,3),(3,3)\)       & a.3                & \((2000)\) & \(\langle 81,8\rangle\)     & \(1\) & \(\Phi_3\) \\
 \(1937\) & \((3,3,3),(3,3),(3,3),(3,3)\)     & a.3*               & \((2000)\) & \(\langle 81,7\rangle\)     & \(1\) & \(\Phi_3\) \\
 \(2993\) & \((27,9),(3,3),(3,3),(3,3)\)      & a.3\(\uparrow\)    & \((2000)\) & \(\langle 729,97\vert 98\rangle\) & \(1\) & \(\Phi_{35}\) \\
\hline
 \(2713\) & \((3,3,3),(3,3,3),(9,3),(3,3,3)\) & H.4                & \((4443)\) & \(\langle 729,45\rangle\)   & \(2\) & \(\Phi_{42}\) \\
 \(3305\) & \((9,9),(9,3),(9,3),(9,3)\)       & c.21               & \((0231)\) & \(\langle 729,54\rangle\)   & \(2\) & \(\Phi_{23}\) \\
 \(3941\) & \((9,3),(9,3),(9,3),(9,3)\)       & G.19               & \((2143)\) & \(\langle 729,57\rangle\)   & \(2\) & \(\Phi_{43}\) \\
\(13153\) & \((3,3,3),(9,3),(9,9),(9,3)\)     & c.18               & \((0313)\) & \(\langle 729,49\rangle\)   & \(2\) & \(\Phi_{23}\) \\
\hline
 \(7665\) & \((27,9),(9,3),(9,3),(9,3)\)      & G.16               & \((4231)\) & \(G_1^{7,8}(0,0,\pm 1,1)\)        & \(2\) & \\
\(15265\) & \((27,9),(9,3),(3,3,3),(9,3)\)    & H.4\(\uparrow\)    & \((3313)\) & \(G_{\pm 1}^{7,8}(0,-1,\pm 1,1)\) & \(2\) & \\
\hline
 \(5912\) & \((27,9),(27,9),(3,3,3),(3,3,3)\) & G.16\(\uparrow\uparrow\) & \((1243)\) & \(G_{\pm 1}^{7,10}(0,0,\pm 1,1)\) & \(4\) & \\
\(12685\) & \((27,9),(27,9),(3,3,3),(3,3,3)\) & G.19\(\uparrow\uparrow\) & \((2143)\) & \(G_{\pm 1}^{7,10}(0,1,0,0)\)     & \(4\) & \\
\hline
\end{tabular}
\end{center}
\end{table}



\subsection{Bicyclic biquadratic Dirichlet fields of type \((3,3)\)}
\label{ss:NewRsltGDH}

In the range \(0<d<2\cdot 10^4\) of real quadratic discriminants \(d\),
we discovered \(22\) variants
of the second \(3\)-class group \(G=\mathrm{G}_3^2(K)\)
of bicyclic biquadratic fields \(K=\mathbb{Q}\left(\sqrt{\strut -1},\sqrt{\strut d}\right)\)
containing the fourth roots of unity and
having a \(3\)-class group of type \((3,3)\).
In Table
\ref{tbl:LiftedVariants},
resp. Table
\ref{tbl:GenuineVariants},
we present the smallest discriminants \(d\)
for which these \(22\) variants occur,
divided into \(11\) \textit{lifted} variants \textit{inherited} from the complex quadratic subfield of \(K\),
resp. \(11\) \textit{intrinsic} or \textit{genuine} variants of \(K\) itself.

About \(20\%\) of these bicyclic biquadratic fields \(K\)
are composita of a real quadratic field \(k_1=\mathbb{Q}\left(\sqrt{\strut d}\right)\) of \(3\)-class rank \(0\)
and its dual complex quadratic field \(k_2=\mathbb{Q}\left(\sqrt{\strut -d}\right)\) of \(3\)-class rank \(2\).
In this case,
the second \(3\)-class group \(\mathrm{G}_3^2(K)\simeq\mathrm{G}_3^2(k_2)\)
is \textit{inherited} from the complex quadratic subfield \(k_2\)
by \textit{lifting} the entire \(3\)-class field tower isomorphically from \(k_2\) to \(K\). 
The discriminants of the quadratic subfields are denoted by
\(\mathrm{d}(k_1)=d\) and \(\mathrm{d}(k_2)=d^\prime\).

Roughly \(80\%\) of these bicyclic biquadratic fields \(K\)
are composita of dual quadratic fields \(k_1\) and \(k_2\) of equal \(3\)-class rank \(1\).
Their second \(3\)-class groups \(\mathrm{G}_3^2(K)\) are intrinsic, genuine
invariants of the bicyclic biquadratic fields \(K\). 



\subsection{Bicyclic biquadratic Eisenstein fields of type \((3,3)\)}
\label{ss:NewRsltESR}

In cooperation with A. Azizi, M. Talbi, and A. Derhem
\cite{ATDM},
and based on
\cite{AAIT1,Tb}
we have completely determined all possibilities
for the isomorphism type of the second \(3\)-class group
\(G=\mathrm{G}_3^2(K)\)
of a bicyclic biquadratic base field
\(K=\mathbb{Q}\left(\sqrt{\strut -3},\sqrt{\strut d}\right)\),
containing the third roots of unity,
of type \((3,3)\)
and we are able to draw an impressive resum\'e in comparison to a quadratic base field
\(K=\mathbb{Q}(\sqrt{\strut D})\).
The possibilities are totally disjoint.

\begin{itemize}
\item
For odd coclass \(\mathrm{cc}(G)=2n+1\equiv 1\pmod{2}\),
the groups of biquadratic fields are exclusively mainline vertices of depth \(0\),
whereas the groups of real quadratic fields are vertices of depth \(1\) on branches,
and odd coclass is impossible at all for compex quadratic fields.
However, in the case of coclass \(2n+1\ge 3\)
both kinds of fields have their groups on the same tree
with mainline of transfer kernel type b.10, \(\varkappa=(0,0,4,3)\),
and the other three trees of coclass graph \(\mathcal{G}(3,2n+1)\) are
populated by the groups of neither biquadratic fields nor quadratic fields.
\item
For even coclass \(\mathrm{cc}(G)=2n\equiv 0\pmod{2}\),
the separation is even more striking.
While the groups of biquadratic fields are restricted to the single tree
whose mainline vertices share the transfer kernel type b.10, \(\varkappa=(0,0,4,3)\),
exactly this tree is entirely forbidden for any quadratic field
and the groups of real and complex quadratic fields are located
on all the other two, resp. five, trees and on the sporadic part of
coclass graph \(\mathcal{G}(3,2n)\), where \(2n=2\), resp. \(2n\ge 4\).
\end{itemize}



\begin{figure}[ht]
\caption{Distribution of \(\mathrm{G}_3^2(K)\) on the coclass graph \(\mathcal{G}(3,1)\)}
\label{fig:WimanBlackburn1}

\setlength{\unitlength}{1cm}
\begin{picture}(16,15)(-8,-14)

\put(-8,0.5){\makebox(0,0)[cb]{Order \(3^n\)}}
\put(-8,0){\line(0,-1){12}}
\multiput(-8.1,0)(0,-2){7}{\line(1,0){0.2}}
\put(-8.2,0){\makebox(0,0)[rc]{\(9\)}}
\put(-7.8,0){\makebox(0,0)[lc]{\(3^2\)}}
\put(-8.2,-2){\makebox(0,0)[rc]{\(27\)}}
\put(-7.8,-2){\makebox(0,0)[lc]{\(3^3\)}}
\put(-8.2,-4){\makebox(0,0)[rc]{\(81\)}}
\put(-7.8,-4){\makebox(0,0)[lc]{\(3^4\)}}
\put(-8.2,-6){\makebox(0,0)[rc]{\(243\)}}
\put(-7.8,-6){\makebox(0,0)[lc]{\(3^5\)}}
\put(-8.2,-8){\makebox(0,0)[rc]{\(729\)}}
\put(-7.8,-8){\makebox(0,0)[lc]{\(3^6\)}}
\put(-8.2,-10){\makebox(0,0)[rc]{\(2\,187\)}}
\put(-7.8,-10){\makebox(0,0)[lc]{\(3^7\)}}
\put(-8.2,-12){\makebox(0,0)[rc]{\(6\,561\)}}
\put(-7.8,-12){\makebox(0,0)[lc]{\(3^8\)}}

\put(-0.1,-0.1){\framebox(0.2,0.2){}}
\put(-2.1,-0.1){\framebox(0.2,0.2){}}
\multiput(0,-2)(0,-2){5}{\circle*{0.2}}
\multiput(-2,-2)(0,-2){6}{\circle*{0.2}}
\multiput(-4,-4)(0,-2){5}{\circle*{0.2}}
\multiput(-6,-4)(0,-4){3}{\circle*{0.2}}
\multiput(2,-6)(0,-2){4}{\circle*{0.1}}
\multiput(4,-6)(0,-2){4}{\circle*{0.1}}
\multiput(6,-6)(0,-2){4}{\circle*{0.1}}

\multiput(0,0)(0,-2){5}{\line(0,-1){2}}
\multiput(0,0)(0,-2){6}{\line(-1,-1){2}}
\multiput(0,-2)(0,-2){5}{\line(-2,-1){4}}
\multiput(0,-2)(0,-4){3}{\line(-3,-1){6}}
\multiput(0,-4)(0,-2){4}{\line(1,-1){2}}
\multiput(0,-4)(0,-2){4}{\line(2,-1){4}}
\multiput(0,-4)(0,-2){4}{\line(3,-1){6}}

\put(0,-10){\vector(0,-1){2}}
\put(0.2,-11.5){\makebox(0,0)[lc]{main}}
\put(0.2,-12){\makebox(0,0)[lc]{line}}
\put(0,-12.3){\makebox(0,0)[rc]{\(\mathcal{T}(\langle 9,2\rangle)\)}}

\put(0.2,0){\makebox(0,0)[lt]{\(C_3\times C_3\)}}
\put(-2.2,0){\makebox(0,0)[rt]{\(C_9\)}}
\put(0.2,-2){\makebox(0,0)[lt]{\(G^3_0(0,0)\)}}
\put(-2.2,-2){\makebox(0,0)[rt]{\(G^3_0(0,1)\)}}
\put(-4,-4.5){\makebox(0,0)[cc]{\(\mathrm{Syl}_3(A_9)\)}}

\put(-3,0){\makebox(0,0)[cc]{\(\Phi_1\)}}
\put(-2.1,0.1){\makebox(0,0)[rb]{\(\langle 1\rangle\)}}
\put(-0.1,0.1){\makebox(0,0)[rb]{\(\langle 2\rangle\)}}

\put(2,-2){\makebox(0,0)[cc]{\(\Phi_2\)}}
\put(-2.1,-1.9){\makebox(0,0)[rb]{\(\langle 4\rangle\)}}
\put(-0.1,-1.9){\makebox(0,0)[rb]{\(\langle 3\rangle\)}}

\put(2,-4){\makebox(0,0)[cc]{\(\Phi_3\)}}
\put(-6.1,-3.9){\makebox(0,0)[rb]{\(\langle 8\rangle\)}}
\put(-4.1,-3.9){\makebox(0,0)[rb]{\(\langle 7\rangle\)}}
\put(-2.1,-3.9){\makebox(0,0)[rb]{\(\langle 10\rangle\)}}
\put(-0.1,-3.9){\makebox(0,0)[rb]{\(\langle 9\rangle\)}}

\put(-4,-6.5){\makebox(0,0)[cc]{\(\Phi_9\)}}
\put(-4.1,-5.9){\makebox(0,0)[rb]{\(\langle 25\rangle\)}}
\put(-2.1,-5.9){\makebox(0,0)[rb]{\(\langle 27\rangle\)}}
\put(-0.1,-5.9){\makebox(0,0)[rb]{\(\langle 26\rangle\)}}

\put(4,-6.5){\makebox(0,0)[cc]{\(\Phi_{10}\)}}
\put(2.1,-5.9){\makebox(0,0)[lb]{\(\langle 28\rangle\)}}
\put(4.1,-5.9){\makebox(0,0)[lb]{\(\langle 30\rangle\)}}
\put(6.1,-5.9){\makebox(0,0)[lb]{\(\langle 29\rangle\)}}

\put(-4,-8.5){\makebox(0,0)[cc]{\(\Phi_{35}\)}}
\put(-6.1,-7.9){\makebox(0,0)[rb]{\(\langle 98\rangle\)}}
\put(-4.1,-7.9){\makebox(0,0)[rb]{\(\langle 97\rangle\)}}
\put(-2.1,-7.9){\makebox(0,0)[rb]{\(\langle 96\rangle\)}}
\put(-0.1,-7.9){\makebox(0,0)[rb]{\(\langle 95\rangle\)}}

\put(4,-8.5){\makebox(0,0)[cc]{\(\Phi_{36}\)}}
\put(2.1,-7.9){\makebox(0,0)[lb]{\(\langle 100\rangle\)}}
\put(4.1,-7.9){\makebox(0,0)[lb]{\(\langle 99\rangle\)}}
\put(6.1,-7.9){\makebox(0,0)[lb]{\(\langle 101\rangle\)}}

\put(-0.1,-9.9){\makebox(0,0)[rb]{\(\langle 386\rangle\)}}

\put(0,-13){\makebox(0,0)[cc]{\(G^n_0(0,0)\)}}
\put(-2,-13){\makebox(0,0)[cc]{\(G^n_0(0,1)\)}}
\put(-4,-13){\makebox(0,0)[cc]{\(G^n_0(1,0)\)}}
\put(-6,-13){\makebox(0,0)[cc]{\(G^n_0(-1,0)\)}}
\put(2,-13){\makebox(0,0)[cc]{\(G^n_1(0,-1)\)}}
\put(4,-13){\makebox(0,0)[cc]{\(G^n_1(0,0)\)}}
\put(6,-13){\makebox(0,0)[cc]{\(G^n_1(0,1)\)}}

\put(2.5,0){\makebox(0,0)[cc]{\textbf{TKT:}}}
\put(3.5,0){\makebox(0,0)[cc]{a.1}}
\put(3.5,-0.5){\makebox(0,0)[cc]{\((0000)\)}}
\put(1.8,-0.7){\framebox(2.9,1){}}
\put(-6,-2){\makebox(0,0)[cc]{\textbf{TKT:}}}
\put(-5,-2){\makebox(0,0)[cc]{A.1}}
\put(-5,-2.5){\makebox(0,0)[cc]{\((1111)\)}}
\put(-6.7,-2.7){\framebox(2.9,1){}}

\put(-8,-14){\makebox(0,0)[cc]{\textbf{TKT:}}}
\put(0,-14){\makebox(0,0)[cc]{a.1}}
\put(-2,-14){\makebox(0,0)[cc]{a.2}}
\put(-4,-14){\makebox(0,0)[cc]{a.3}}
\put(-6,-14){\makebox(0,0)[cc]{a.3}}
\put(2,-14){\makebox(0,0)[cc]{a.1}}
\put(4,-14){\makebox(0,0)[cc]{a.1}}
\put(6,-14){\makebox(0,0)[cc]{a.1}}
\put(0,-14.5){\makebox(0,0)[cc]{\((0000)\)}}
\put(-2,-14.5){\makebox(0,0)[cc]{\((1000)\)}}
\put(-4,-14.5){\makebox(0,0)[cc]{\((2000)\)}}
\put(-6,-14.5){\makebox(0,0)[cc]{\((2000)\)}}
\put(2,-14.5){\makebox(0,0)[cc]{\((0000)\)}}
\put(4,-14.5){\makebox(0,0)[cc]{\((0000)\)}}
\put(6,-14.5){\makebox(0,0)[cc]{\((0000)\)}}
\put(-8.7,-14.7){\framebox(15.4,1){}}

\put(0,0){\oval(3,1.5)}
\put(0,-4){\oval(3,1.5)}
\put(0,-8){\oval(3,1.5)}
\put(0.5,-1.1){\makebox(0,0)[cc]{\underbar{\textbf{605}}}}
\put(0.5,-5.1){\makebox(0,0)[cc]{\underbar{\textbf{197}}}}
\put(0.5,-9.1){\makebox(0,0)[cc]{\underbar{\textbf{42}}}}

\end{picture}

\end{figure}

Since the behavior of these biquadratic fields \(K\)
with respect to second \(3\)-class groups
is totally different from quadratic fields,
the following Figures
\ref{fig:WimanBlackburn1}--\ref{fig:TreeTyp33Cocl3}
visualize the distribution of their second \(3\)-class group \(G=\mathrm{G}_3^2(K)\)
on the coclass graphs \(\mathcal{G}(3,r)\), \(1\le r\le 3\).

\noindent
In Figures
\ref{fig:WimanBlackburn1}--\ref{fig:TreeTyp33Cocl3}
the actual distribution of the \(930\) second \(3\)-class groups \(G=G_3^2(K)\)
of bicyclic biquadratic number fields \(K=\mathbb{Q}\left(\sqrt{\strut -3},\sqrt{\strut d}\right)\)
of type \((3,3)\) with discriminant \(0<d<5\cdot 10^4\) is represented by
underlined boldface counters of hits of the vertices surrounded by the adjacent oval.
Isomorphisms among the extensions \(L_i\vert K\), \(1\le i\le 4\),
cause severe constraints on the group \(G\).

\noindent
We point out that only every other mainline vertex of \(\mathcal{G}(3,1)\)
is populated by second \(3\)-class groups \(\mathrm{G}_3^2(K)\) of quartic fields \(K\) in Figure
\ref{fig:WimanBlackburn1}
in contrast to the distribution of the groups \(\mathrm{G}_3^2(K)\) of quadratic fields \(K\) in Figure
\ref{fig:Distr3Cocl1}.



\begin{figure}[ht]
\caption{Distribution of \(\mathrm{G}_3^2(K)\) on the tree \(\mathcal{T}(\langle 243,3\rangle)\) of \(\mathcal{G}(3,2)\)}
\label{fig:TreeBTyp33Cocl2}

\setlength{\unitlength}{1cm}
\begin{picture}(14,14.5)(-6,-13.5)

\put(-5,0.5){\makebox(0,0)[cb]{Order \(3^n\)}}
\put(-5,0){\line(0,-1){12}}
\multiput(-5.1,0)(0,-2){7}{\line(1,0){0.2}}
\put(-5.2,0){\makebox(0,0)[rc]{\(243\)}}
\put(-4.8,0){\makebox(0,0)[lc]{\(3^5\)}}
\put(-5.2,-2){\makebox(0,0)[rc]{\(729\)}}
\put(-4.8,-2){\makebox(0,0)[lc]{\(3^6\)}}
\put(-5.2,-4){\makebox(0,0)[rc]{\(2\,187\)}}
\put(-4.8,-4){\makebox(0,0)[lc]{\(3^7\)}}
\put(-5.2,-6){\makebox(0,0)[rc]{\(6\,561\)}}
\put(-4.8,-6){\makebox(0,0)[lc]{\(3^8\)}}
\put(-5.2,-8){\makebox(0,0)[rc]{\(19\,683\)}}
\put(-4.8,-8){\makebox(0,0)[lc]{\(3^9\)}}
\put(-5.2,-10){\makebox(0,0)[rc]{\(59\,049\)}}
\put(-4.8,-10){\makebox(0,0)[lc]{\(3^{10}\)}}
\put(-5.2,-12){\makebox(0,0)[rc]{\(177\,147\)}}
\put(-4.8,-12){\makebox(0,0)[lc]{\(3^{11}\)}}

\multiput(0,0)(0,-2){6}{\circle*{0.2}}
\multiput(0,0)(0,-2){5}{\line(0,-1){2}}

\put(0,-10){\vector(0,-1){2}}
\put(0.2,-11.5){\makebox(0,0)[lc]{main}}
\put(0.2,-12){\makebox(0,0)[lc]{line}}
\put(0,-12.3){\makebox(0,0)[rc]{\(\mathcal{T}(\langle 729,40\rangle)\)}}
\put(7.2,-5.8){\makebox(0,0)[lc]{\(*3\)}}
\multiput(7,-4)(1,0){2}{\vector(0,-1){2}}
\put(7.5,-6.5){\makebox(0,0)[cc]{non-}}
\put(7.5,-7){\makebox(0,0)[cc]{metabelian}}
\put(7.5,-7.5){\makebox(0,0)[cc]{main}}
\put(7.5,-8){\makebox(0,0)[cc]{lines}}

\multiput(-1,-2)(0,-2){6}{\circle*{0.2}}
\multiput(-2,-2)(0,-2){6}{\circle*{0.2}}
\multiput(-3,-2)(0,-2){6}{\circle*{0.2}}
\multiput(1.95,-4.05)(0,-2){5}{\framebox(0.1,0.1){}}
\multiput(3,-2)(0,-2){6}{\circle*{0.1}}
\multiput(4,-2)(0,-2){6}{\circle*{0.1}}
\multiput(5,-2)(1,0){2}{\circle*{0.1}}
\multiput(1,-2)(7,0){2}{\circle*{0.1}}
\multiput(0,0)(0,-2){6}{\line(-1,-2){1}}
\multiput(0,0)(0,-2){6}{\line(-1,-1){2}}
\multiput(0,0)(0,-2){6}{\line(-3,-2){3}}
\multiput(0,-2)(0,-2){5}{\line(1,-1){2}}
\multiput(0,0)(0,-2){6}{\line(3,-2){3}}
\multiput(0,0)(0,-2){6}{\line(2,-1){4}}
\put(0,0){\line(1,-2){1}}
\put(0,0){\line(5,-2){5}}
\put(0,0){\line(3,-1){6}}
\put(0,0){\line(4,-1){8}}

\put(-4.05,-4.05){\framebox(0.1,0.1){}}
\put(5.95,-4.05){\framebox(0.1,0.1){}}
\multiput(6.9,-4.1)(1,0){2}{\framebox(0.2,0.2){}}
\multiput(4.95,-4.05)(0,-2){5}{\framebox(0.1,0.1){}}
\put(-1,-2){\line(-3,-2){3}}
\multiput(3,-2)(1,0){4}{\line(1,-1){2}}
\multiput(4,-4)(0,-2){4}{\line(1,-2){1}}

\put(-2.5,0.5){\makebox(0,0)[cc]{\(\Phi_{6}\)}}
\put(-2.5,-2.5){\makebox(0,0)[cc]{\(\Phi_{23}\)}}
\put(2,-2.5){\makebox(0,0)[cc]{\(\Phi_{41}\)}}
\put(7,-2.5){\makebox(0,0)[cc]{\(\Phi_{40}\)}}
\put(-0.1,0.5){\makebox(0,0)[rc]{\(\langle 3\rangle\)}}
\put(-3.1,-1.8){\makebox(0,0)[rc]{\(\langle 42\rangle\)}}
\put(-2.1,-1.8){\makebox(0,0)[rc]{\(\langle 43\rangle\)}}
\put(-1.1,-1.8){\makebox(0,0)[rc]{\(\langle 41\rangle\)}}
\put(-0.1,-1.8){\makebox(0,0)[rc]{\(\langle 40\rangle\)}}
\put(1.1,-1.8){\makebox(0,0)[lc]{\(\langle 39\rangle\)}}
\put(3.1,-1.8){\makebox(0,0)[lc]{\(\langle 38\rangle\)}}
\put(4.1,-1.8){\makebox(0,0)[lc]{\(\langle 37\rangle\)}}
\put(5.1,-1.8){\makebox(0,0)[lc]{\(\langle 34\rangle\)}}
\put(6.1,-1.8){\makebox(0,0)[lc]{\(\langle 35\rangle\)}}
\put(8.1,-1.8){\makebox(0,0)[lc]{\(\langle 36\rangle\)}}
\put(-3.9,-3.8){\makebox(0,0)[rc]{\(3*\)}}
\multiput(-2.1,-3.8)(0,-4){3}{\makebox(0,0)[rc]{\(2*\)}}
\multiput(-1.1,-3.8)(0,-4){3}{\makebox(0,0)[rc]{\(2*\)}}
\put(2.1,-3.8){\makebox(0,0)[lc]{\(*4\)}}
\put(5.1,-3.8){\makebox(0,0)[lc]{\(*5\)}}
\put(6.1,-3.8){\makebox(0,0)[lc]{\(*6\)}}
\multiput(7.2,-3.8)(1,0){2}{\makebox(0,0)[lc]{\(*6\)}}
\multiput(2.1,-5.8)(0,-4){2}{\makebox(0,0)[lc]{\(*3\)}}
\multiput(2.1,-7.8)(0,-4){2}{\makebox(0,0)[lc]{\(*2\)}}
\multiput(3,-3.8)(0,-4){3}{\makebox(0,0)[lc]{\(*5\)}}
\multiput(3,-5.8)(0,-4){2}{\makebox(0,0)[lc]{\(*6\)}}
\multiput(4.1,-5.8)(0,-4){2}{\makebox(0,0)[lc]{\(*2\)}}
\multiput(5.1,-5.8)(0,-4){2}{\makebox(0,0)[lc]{\(*12\)}}
\multiput(5.1,-7.8)(0,-4){2}{\makebox(0,0)[lc]{\(*8\)}}
\put(-4,-13){\makebox(0,0)[cc]{\textbf{TKT:}}}
\put(-3,-13){\makebox(0,0)[cc]{d.23}}
\put(-2,-13){\makebox(0,0)[cc]{d.25}}
\put(-1,-13){\makebox(0,0)[cc]{d.19}}
\put(0,-13){\makebox(0,0)[cc]{b.10}}
\put(4,-13){\makebox(0,0)[cc]{b.10}}
\put(-3,-13.5){\makebox(0,0)[cc]{\((1043)\)}}
\put(-2,-13.5){\makebox(0,0)[cc]{\((2043)\)}}
\put(-1,-13.5){\makebox(0,0)[cc]{\((4043)\)}}
\put(0,-13.5){\makebox(0,0)[cc]{\((0043)\)}}
\put(4,-13.5){\makebox(0,0)[cc]{\((0043)\)}}
\put(-4.8,-13.7){\framebox(9.6,1){}}

\put(0,-2){\oval(1.5,1.5)}
\put(4.8,-2){\oval(2,1.5)}
\put(4,-6){\oval(1,1.5)}
\put(4,-10){\oval(1,1.5)}
\put(0.3,-3){\makebox(0,0)[cc]{\underbar{\textbf{3}}}}
\put(4.5,-3){\makebox(0,0)[cc]{\underbar{\textbf{59}}}}
\put(3.5,-6.9){\makebox(0,0)[cc]{\underbar{\textbf{17}}}}
\put(3.5,-10.9){\makebox(0,0)[cc]{\underbar{\textbf{1}}}}

\end{picture}

\end{figure}
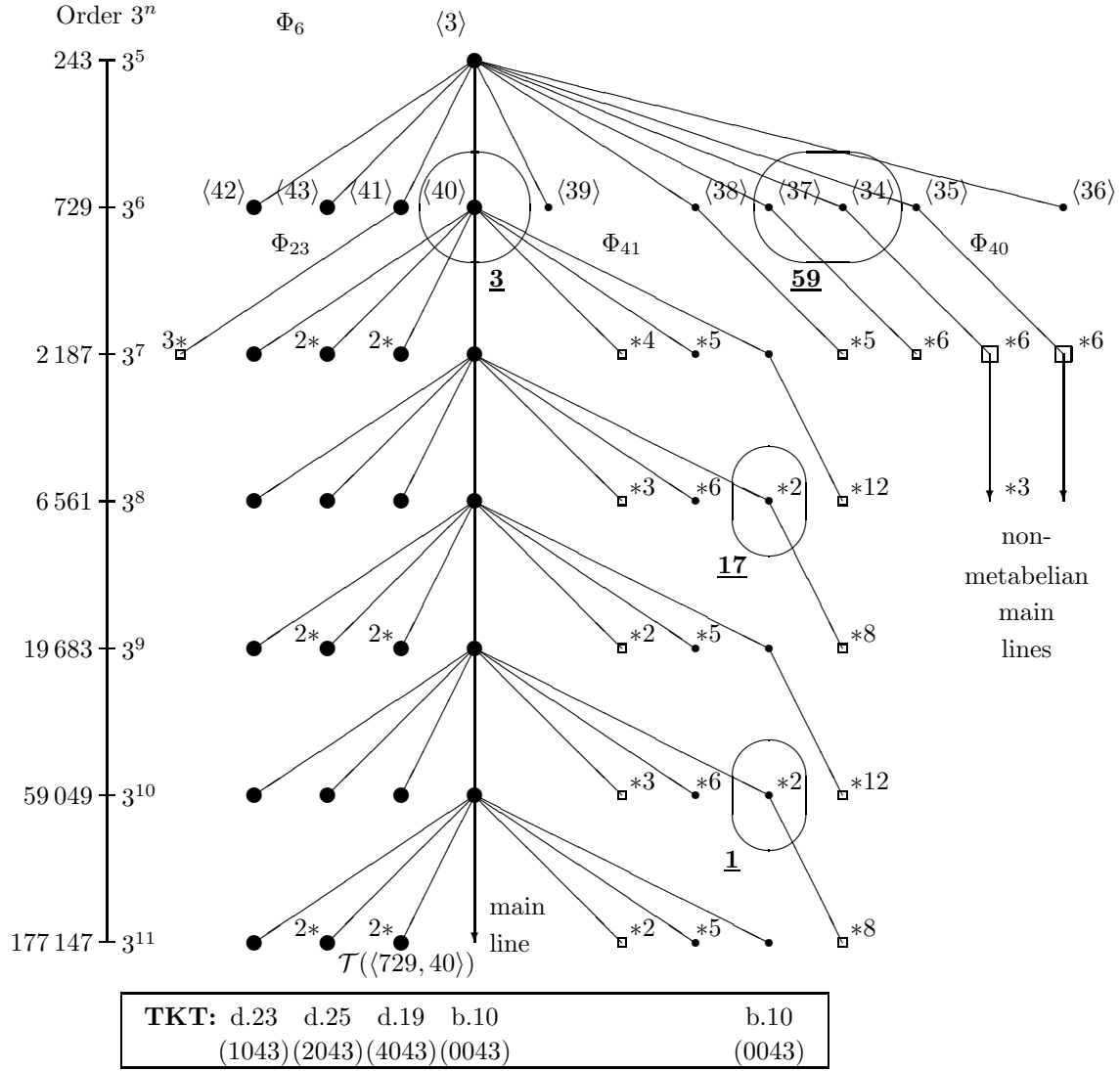



\noindent
Vertices of the tree \(\mathcal{T}(\langle 243,3\rangle)\)
of coclass graph \(\mathcal{G}(3,2)\) in Figure
\ref{fig:TreeBTyp33Cocl2}
are classified according to their defect \(k(G)\)
by using different symbols:

\begin{enumerate}
\item
big full discs {\Large \(\bullet\)} denote metabelian groups with defect \(k(G)=0\) and centre of type \((3,3)\),
\item
small full discs {\footnotesize \(\bullet\)} denote metabelian groups with \(k(G)=1\) and cyclic centre of order \(3\),
\item
small contour squares {\tiny \(\square\)} denote terminal non-metabelian groups,
\item
large contour squares \(\square\) denote infinitely capable non-metabelian groups,
giving rise to non-metabelian mainlines
\cite{AHL,ELNO}.
\end{enumerate}



\noindent
This tree is completely forbidden for quadratic fields.
A symbol \(n\ast\) adjacent to a vertex denotes the multiplicity of a batch
of \(n\) immediate descendants of a common parent.
Numbers in angles denote identifiers in the SmallGroups library
\cite{BEO,GAP},
where we omit the orders, which are given on the left hand scale.
The symbols \(\Phi_s\) denote isoclinism families
\cite{Hl,Ef,Jm}.
Transfer kernel types, briefly TKT,
\cite[Thm. 2.5, Tbl. 6--7]{Ma2}
in the bottom rectangle concern all vertices located vertically above.
The periodicity with length \(2\) of branches,
\(\mathcal{B}(j)\simeq\mathcal{B}(j+2)\) for \(j\ge 7\),
sets in with branch \(\mathcal{B}(7)\), having root of order \(3^7\).



\begin{figure}[ht]
\caption{Distribution of \(\mathrm{G}_3^2(K)\) on the tree \(\mathcal{T}(G_0^{5,7}(0,0,0,0))\) of \(\mathcal{G}(3,3)\)}
\label{fig:TreeTyp33Cocl3}

\setlength{\unitlength}{1cm}
\begin{picture}(14,14.5)(-7,-13.5)

\put(-5,0.5){\makebox(0,0)[cb]{Order \(3^n\)}}
\put(-5,0){\line(0,-1){12}}
\multiput(-5.1,0)(0,-2){7}{\line(1,0){0.2}}
\put(-5.2,0){\makebox(0,0)[rc]{\(2\,187\)}}
\put(-4.8,0){\makebox(0,0)[lc]{\(3^7\)}}
\put(-5.2,-2){\makebox(0,0)[rc]{\(6\,561\)}}
\put(-4.8,-2){\makebox(0,0)[lc]{\(3^8\)}}
\put(-5.2,-4){\makebox(0,0)[rc]{\(19\,683\)}}
\put(-4.8,-4){\makebox(0,0)[lc]{\(3^9\)}}
\put(-5.2,-6){\makebox(0,0)[rc]{\(59\,049\)}}
\put(-4.8,-6){\makebox(0,0)[lc]{\(3^{10}\)}}
\put(-5.2,-8){\makebox(0,0)[rc]{\(177\,147\)}}
\put(-4.8,-8){\makebox(0,0)[lc]{\(3^{11}\)}}
\put(-5.2,-10){\makebox(0,0)[rc]{\(531\,441\)}}
\put(-4.8,-10){\makebox(0,0)[lc]{\(3^{12}\)}}
\put(-5.2,-12){\makebox(0,0)[rc]{\(1\,594\,323\)}}
\put(-4.8,-12){\makebox(0,0)[lc]{\(3^{13}\)}}

\multiput(0,0)(0,-2){6}{\circle*{0.2}}
\multiput(0,0)(0,-2){5}{\line(0,-1){2}}

\put(0,-10){\vector(0,-1){2}}
\put(0.2,-11.5){\makebox(0,0)[lc]{main}}
\put(0.2,-12){\makebox(0,0)[lc]{line}}
\put(0,-12.4){\makebox(0,0)[rc]{\(\mathcal{T}(G_0^{5,7}(0,0,0,0))\)}}

\multiput(-1,-2)(0,-2){6}{\circle*{0.2}}
\multiput(-2,-2)(0,-2){6}{\circle*{0.2}}
\multiput(-3,-2)(0,-2){6}{\circle*{0.2}}
\multiput(3,-2)(0,-2){6}{\circle*{0.1}}
\multiput(0,0)(0,-2){6}{\line(-1,-2){1}}
\multiput(0,0)(0,-2){6}{\line(-1,-1){2}}
\multiput(0,0)(0,-2){6}{\line(-3,-2){3}}
\multiput(0,0)(0,-2){6}{\line(3,-2){3}}

\put(0,0.5){\makebox(0,0)[cc]{\(G_0^{5,7}(0,0,0,0)\simeq\langle 2187,64\rangle\)}}
\put(3.1,-1.8){\makebox(0,0)[lc]{\(*6\)}}
\multiput(-2.1,-1.8)(0,-4){3}{\makebox(0,0)[rc]{\(2*\)}}
\multiput(-1.1,-1.8)(0,-4){3}{\makebox(0,0)[rc]{\(2*\)}}
\multiput(3.1,-3.8)(0,-4){3}{\makebox(0,0)[lc]{\(*6\)}}
\multiput(3.1,-5.8)(0,-4){2}{\makebox(0,0)[lc]{\(*9\)}}
\put(-4,-13){\makebox(0,0)[cc]{\textbf{TKT:}}}
\put(-3,-13){\makebox(0,0)[cc]{d.23}}
\put(-2,-13){\makebox(0,0)[cc]{d.25}}
\put(-1,-13){\makebox(0,0)[cc]{d.19}}
\put(0,-13){\makebox(0,0)[cc]{b.10}}
\put(3,-13){\makebox(0,0)[cc]{b.10}}
\put(-3,-13.5){\makebox(0,0)[cc]{\((1043)\)}}
\put(-2,-13.5){\makebox(0,0)[cc]{\((2043)\)}}
\put(-1,-13.5){\makebox(0,0)[cc]{\((4043)\)}}
\put(0,-13.5){\makebox(0,0)[cc]{\((0043)\)}}
\put(3,-13.5){\makebox(0,0)[cc]{\((0043)\)}}
\put(-4.8,-13.7){\framebox(8.6,1){}}

\put(0,-2){\oval(1,1.5)}
\put(0.4,-3){\makebox(0,0)[cc]{\underbar{\textbf{5}}}}

\end{picture}

\end{figure}
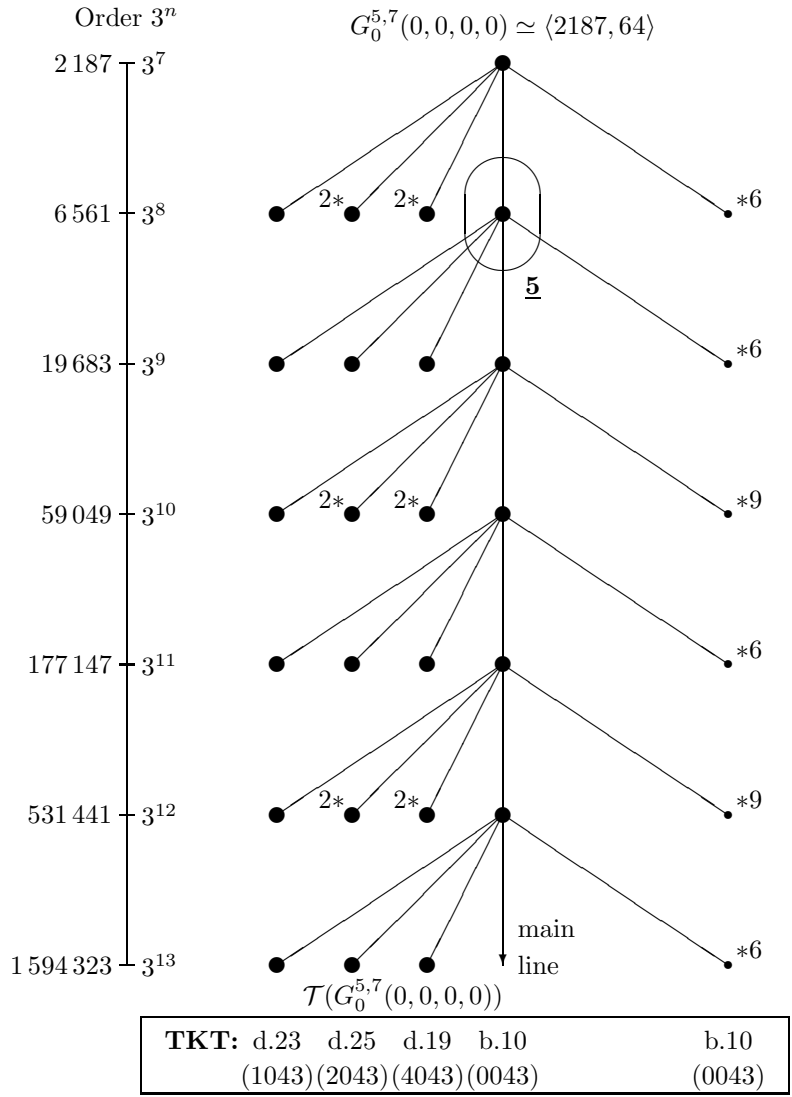



\bigskip
\noindent
Vertices of the metabelian skeleton of tree \(\mathcal{T}(G_0^{5,7}(0,0,0,0))\)
of coclass graph \(\mathcal{G}(3,3)\) in Figure
\ref{fig:TreeTyp33Cocl3}
are classified according to their defect \(k(G)\)
by using different symbols:

\begin{enumerate}
\item
big full discs {\Large \(\bullet\)} denote metabelian groups with defect \(k(G)=0\) and centre of type \((3,3)\),
\item
small full discs {\footnotesize \(\bullet\)} denote metabelian groups with \(k(G)=1\) and cyclic centre of order \(3\).
\end{enumerate}



\noindent
The symbol \(n\ast\) denotes a batch of \(n\) siblings of a common parent.
Transfer kernel types, briefly TKT,
\cite[Thm. 2.5, Tbl. 6--7]{Ma2}
in the bottom rectangle concern all vertices located vertically above.
Metabelian periodicity with length \(2\) of branches,
\(\mathcal{B}(j)\simeq\mathcal{B}(j+2)\) for \(j\ge 8\),
sets in with branch \(\mathcal{B}(8)\), having root of order \(3^8\).



\renewcommand{\arraystretch}{1.1}

\begin{table}[ht]
\caption{\(11\) variants of \(G=\mathrm{G}_3^2(K)\) for \(930\) \(K=\mathbb{Q}\left(\sqrt{\strut -3},\sqrt{\strut d}\right)\)}
\label{tbl:11Variants}
\begin{center}
\begin{tabular}{|r||c||c|c||c|c|c|}
\hline
  \(d\)   & \(\tau(G)=(\mathrm{Cl}_3(L_i))_{1\le i\le 4}\) & Type  & \(\varkappa(G)\) & \(G\)    & \(\mathrm{cc}(G)\) & \(\Phi\) \\
\hline
  \(229\) & \((3),(3),(3),(3)\)               & a.1                & \((0000)\) & \(\langle 9,2\rangle\)      & \(1\) & \(\Phi_1\) \\
  \(469\) & \((9,3),(3,3),(3,3),(3,3)\)       & a*.1               & \((0000)\) & \(\langle 81,9\rangle\)     & \(1\) & \(\Phi_3\) \\
 \(7453\) & \((27,9),(3,3),(3,3),(3,3)\)      & a*.1\(\uparrow\)   & \((0000)\) & \(\langle 729,95\rangle\)   & \(1\) & \(\Phi_{35}\) \\
\hline
 \(2177\) & \((9,3),(9,3),(3,3,3),(3,3,3)\)   & b.10               & \((0043)\) & \(\langle 729,37\rangle\)   & \(2\) & \(\Phi_{41}\) \\
 \(2589\) & \((9,3),(9,3),(3,3,3),(3,3,3)\)   & b.10               & \((0043)\) & \(\langle 729,34\rangle\)   & \(2\) & \(\Phi_{40}\) \\
\(17609\) & \((9,3),(9,9),(3,3,3),(3,3,3)\)   & b*.10              & \((0043)\) & \(\langle 729,40\rangle\)   & \(2\) & \(\Phi_{23}\) \\
\(14056\) & \((27,9),(9,3),(3,3,3),(3,3,3)\)  & b.10\(\uparrow\)   & \((0043)\) & \(G_{\pm 1}^{7,8}(0,0,0,0)\)  & \(2\) & \\
\(20521\) & \((9,3),(27,9),(3,3,3),(3,3,3)\)  & b.10\(\uparrow\)   & \((0043)\) & \(G_{\pm 1}^{7,8}(0,0,0,0)\)  & \(2\) & \\
\(44581\) & \((81,27),(9,3),(3,3,3),(3,3,3)\) & b.10\(\uparrow^2\) & \((0043)\) & \(G_{\pm 1}^{9,10}(0,0,0,0)\) & \(2\) & \\
\hline
 \(4933\) & \((27,9),(9,9),(3,3,3),(3,3,3)\)  & b*.10\(\uparrow\uparrow\)  & \((0043)\) & \(G_{0}^{6,8}(0,0,0,0)\)  & \(3\) & \\
\hline
\(47597\) & \((27,9),(27,9),(3,3,3),(3,3,3)\) & b.10\(\uparrow\uparrow^2\) & \((0043)\) & \(G_{\pm 1}^{7,10}(0,0,0,0)\) & \(4\) & \\
\hline
\end{tabular}
\end{center}
\end{table}

In the range \(0<d<5\cdot 10^4\) of real quadratic discriminants \(d\),
we discovered \(11\) variants
of the second \(3\)-class group \(G=\mathrm{G}_3^2(K)\)
of bicyclic biquadratic fields \(K=\mathbb{Q}\left(\sqrt{\strut -3},\sqrt{\strut d}\right)\)
having \(3\)-class group of type \((3,3)\).
In Table
\ref{tbl:11Variants}
we present the smallest discriminants \(d\)
for which these \(11\) variants occur.

\noindent
The invariants listed are
the discriminant \(d\) of the real quadratic subfield
\(\mathbb{Q}\left(\sqrt{\strut d}\right)\) of \(K\),
the TTT \(\tau(G)\) of \(G\),
the TKT \(\varkappa(G)\) of \(G\),
with arrows \(\uparrow\) denoting excited states,
the GAP \(4\) identifier of \(G\) in the SmallGroups library
\cite{BEO,GAP},
provided that \(\lvert G\rvert\le 3^6\), otherwise
the symbol \(G_\varrho^{m,n}(\alpha,\beta,\gamma,\delta)\)
for the isomorphism type of \(G\) defined in section \S\
\ref{sss:PrmPres2},
if \(\lvert G\rvert\ge 3^8\),
the coclass \(\mathrm{cc}(G)\) of \(G\),
and the isoclinism family \(\Phi\) to which \(G\) belongs,
as far as it is defined in
\cite{Hl,Ef,Jm}.



\subsection{Stem of isoclinism family \(\Phi_6\)}
\label{ss:StemPhi6}

In this section,
we provide group theoretic foundations
for determining second \(5\)-class groups \(\mathrm{G}_5^2(K)\)
of coclass \(\mathrm{cc}(G)\ge 2\)
for quadratic and quartic number fields \(K\) of type \((5,5)\).
The stem of Hall's isoclinism family \(\Phi_6\) is the key for
a deeper understanding of the \(5\)-principalization of these base fields
in their six unramified cyclic quintic extensions \(L_1,\ldots,L_6\),
which has partially but not completely been investigated by
Heider and Schmithals
\cite{HeSm}
and by Bembom
\cite{Bm}.

The \textit{stem groups} \(G\) of Hall's isoclinism family \(\Phi_6\)
\cite[p. 139]{Hl}
are \(p\)-groups of order \(\lvert G\rvert=p^5\) with odd prime \(p\),
nilpotency class \(\mathrm{cl}(G)=3\), and coclass \(\mathrm{cc}(G)=2\).
They were discovered in 1898 by Bagnera
\cite[pp. 182--183]{Bg},
and were constructed as extensions of \(C_p^3\) by \(C_p^2\), for \(p\ge 5\),
in 1926 by Schreier \cite[pp. 341--345]{Sr2}.
Bagnera also pointed out that these groups do not have an analog for \(p=2\).

Every stem group of isoclinism family \(\Phi_6\) is a \(2\)-generator group \(G=\langle x,y\rangle\)
with main commutator \(s_2=\lbrack y,x\rbrack\) in \(\gamma_2(G)\)
and higher commutators \(s_3=\lbrack s_2,x\rbrack\), \(t_3=\lbrack s_2,y\rbrack\) in \(\gamma_3(G)\),
satisfying the power relations \(s_2^p=s_3^p=t_3^p=1\).
The lower central series of \(G\) is given by
\begin{center}
\(\gamma_2(G)=\langle s_2,s_3,t_3\rangle\) of type \((p,p,p)\),\quad
\(\gamma_3(G)=\langle s_3,t_3\rangle\) of type \((p,p)\),\quad
\(\gamma_4(G)=1\),
\end{center}
and the center by \(\zeta_1(G)=\gamma_3(G)\).
The central quotient \(G/\zeta_1(G)\) is of type \(\Phi_2(1^3)\simeq G_0^3(0,0)\),
the extra special \(p\)-group of order \(p^3\) and exponent \(p\),
and the abelianization \(G/\gamma_2(G)\) is of type \((p,p)\).
Therefore, the lower central structure of these groups uniformly consists of
two bicyclic factors, the \textit{head} \(G/\gamma_2(G)\), and the \textit{tail} \(\gamma_3(G)/\gamma_4(G)\),
separated by the cyclic factor \(\gamma_2(G)/\gamma_3(G)\).

For any stem group \(G\) in \(\Phi_6\),
there exists a nice \(1\)-to-\(1\) correspondence
between the two bicyclic factors,
the head and the tail,
by taking the derived subgroups.

\begin{lemma}
\label{lmm:StemIcl6}
The maximal normal subgroups \(H_i\) of \(G\)
contain the commutator subgroup \(G^\prime=\gamma_2(G)\)
and are given by
\begin{center}
\(H_i=\langle g_i,G^\prime\rangle\) with generators \(g_1=y\) and \(g_i=xy^{i-2}\) for \(2\le i\le p+1\).
\end{center}
Their derived subgroups \(H_i^\prime=(G^\prime)^{g_i-1}\) are given by
\begin{center}
\(H_1^\prime=\langle t_3\rangle\) and
\(H_i^\prime=\langle s_3t_3^{i-2}\rangle\) for \(2\le i\le p+1\).
\end{center}
\end{lemma}

\noindent
As a consequence of Lemma
\ref{lmm:StemIcl6},
we only have trivial two-step centralizers
\(G^\prime=\chi_2(G)<\chi_3(G)=G\) and the invariants \(e(G)\) and \(s(G)\)
of section \S\
\ref{ss:MtabTyp33CoclGe2}
take the same value \(e=s=3\).

Individual relations for isomorphism classes by James
\cite[pp. 620--621]{Jm}
are given in Table
\ref{tab:RelStemIcl6},
where \(\nu\) denotes the smallest positive quadratic non-residue modulo \(p\)
and \(g\) denotes the smallest positive primitive root modulo \(p\).

\renewcommand{\arraystretch}{1.2}
\begin{table}[ht]
\caption{Relations for the stem groups of \(\Phi_6\)}
\label{tab:RelStemIcl6}
\begin{center}
\begin{tabular}{|l|l|l|l|}
\hline
 stem group             & \(x^p\)        & \(y^p\)                & parameters                                             \\
\hline
 \(\Phi_6(221)_a\)      & \(s_3\)        & \(t_3\)                &                                                        \\
 \(\Phi_6(221)_{b_r}\)  & \(s_3\)        & \(t_3^k\)              & \(1\le r\le\frac{p-1}{2}\), \(k=g^r\)                  \\
 \(\Phi_6(221)_{c_r}\)  & \(s_3^rt_3^r\) & \(s_3^{-\frac{r}{4}}\) & \(r\in\lbrace 1,\nu\rbrace\)                           \\
 \(\Phi_6(221)_{d_0}\)  & \(t_3^\nu\)    & \(s_3\)                &                                                        \\
 \(\Phi_6(221)_{d_r}\)  & \(s_3t_3\)     & \(s_3^k\)              & \(1\le r\le\frac{p-1}{2}\), \(k=\frac{g^{2r+1}-1}{4}\) \\
 \(\Phi_6(21^3)_a\)     & \(1\)          & \(t_3\)                & \(p\ge 5\)                                             \\
 \(\Phi_6(21^3)_{b_r}\) & \(t_3^r\)      & \(1\)                  & \(r\in\lbrace 1,\nu\rbrace\), \(p\ge 5\)               \\
 \(\Phi_6(1^5)\)        & \(1\)          & \(1\)                  &                                                        \\      
\hline
\end{tabular}
\end{center}
\end{table}

These presentations
for \(7\) isomorphism classes of \(3\)-groups,
resp. \(12\) isomorphism classes of \(5\)-groups,
among the stem of \(\Phi_6\)
are now used to calculate the kernels of all transfers
\(\mathrm{T}_i:G/G^\prime\to H_i/H_i^\prime\), \(1\le i\le p+1\),
whose images are given for \(p=5\) in very convenient form by Lemma
\ref{lmm:TransferIcl6},
since the expressions for \textit{inner transfers} are well-behaved \(p\)th powers.
\textit{Outer transfers} always map to \(p\)th powers, anyway.

\begin{lemma}
\label{lmm:TransferIcl6}
For any \(1\le i\le 6\), the image of an arbitrary element
\(gG^\prime\in G/G^\prime\) with representation \(x^jy^\ell G^\prime\), \(0\le j,\ell\le 4\),
under the transfer \(\mathrm{T}_i\) is given by
\(\mathrm{T}_i(x^jy^\ell G^\prime)=x^{pj}y^{p\ell}H_i^\prime\).
\end{lemma}

\renewcommand{\arraystretch}{1.0}

\begin{table}[ht]
\caption{TKT of corresponding \(p\)-groups in \(\Phi_6\) for \(p\in\lbrace 3,5\rbrace\)}
\label{tab:TrfKerStemIcl6}
\begin{center}
\begin{tabular}{|cl|lc|ll|ccc|}
\hline
 \multicolumn{4}{|c|}{\(p=3\)}                                                    & \multicolumn{5}{|c|}{\(p=5\)}                       \\
\hline
 \multicolumn{2}{|c|}{\(3\)-group}                     & TKT      & \(\varkappa\) & \multicolumn{2}{|c|}{\(5\)-group} & \(\eta\) & \(\varkappa\) & property           \\
\hline
 \(\langle 243,7\rangle\) & \(G_0^{4,5}(1,1,-1,1)\)    & D.\(5\)  & \((4224)\)    & \(\langle 3125,14\rangle\) & \(\Phi_6(221)_a\)      & 6 & \((123456)\)  & identity           \\
 \(\langle 243,4\rangle\) & \(G_0^{4,5}(1,1,1,1)\)     & H.\(4\)  & \((4443)\)    & \(\langle 3125,11\rangle\) & \(\Phi_6(221)_{b_1}\)  & 2 & \((125364)\)  & \(4\)-cycle        \\
 \multicolumn{2}{|c|}{no analog}                       &          &               & \(\langle 3125,7\rangle\)  & \(\Phi_6(221)_{b_2}\)  & 2 & \((126543)\)  & two transpos.      \\
 \(\langle 243,8\rangle\) & \(G_0^{4,5}(0,0,0,1)\)     & c.\(21\) & \((0231)\)    & \(\langle 3125,8\rangle\)  & \(\Phi_6(221)_{c_1}\)  & 1 & \((612435)\)  & \(5\)-cycle        \\
 \(\langle 243,5\rangle\) & \(G_0^{4,5}(0,0,-1,1)\)    & D.\(10\) & \((2241)\)    & \(\langle 3125,13\rangle\) & \(\Phi_6(221)_{c_2}\)  & 1 & \((612435)\)  & \(5\)-cycle        \\
 \(\langle 243,9\rangle\) & \(G_0^{4,5}(0,-1,-1,0)\)   & G.\(19\) & \((2143)\)    & \(\langle 3125,10\rangle\) & \(\Phi_6(221)_{d_0}\)  & 0 & \((214365)\)  & three transpos.    \\
 \(\langle 243,6\rangle\) & \(G_0^{4,5}(0,-1,0,1)\)    & c.\(18\) & \((0313)\)    & \(\langle 3125,12\rangle\) & \(\Phi_6(221)_{d_1}\)  & 0 & \((512643)\)  & \(6\)-cycle        \\
 \multicolumn{2}{|c|}{no analog}                       &          &               & \(\langle 3125,9\rangle\)  & \(\Phi_6(221)_{d_2}\)  & 0 & \((312564)\)  & two \(3\)-cycles   \\
 \multicolumn{2}{|c|}{no analog}                       &          &               & \(\langle 3125,4\rangle\)  & \(\Phi_6(21^3)_a\)     & 2 & \((022222)\)  & nrl.const.with fp. \\
 \multicolumn{2}{|c|}{no analog}                       &          &               & \(\langle 3125,5\rangle\)  & \(\Phi_6(21^3)_{b_1}\) & 1 & \((011111)\)  & nearly constant    \\
 \multicolumn{2}{|c|}{no analog}                       &          &               & \(\langle 3125,6\rangle\)  & \(\Phi_6(21^3)_{b_2}\) & 1 & \((011111)\)  & nearly constant    \\
 \(\langle 243,3\rangle\) & \(G_0^{4,5}(0,0,0,0)\)     & b.\(10\) & \((0043)\)    & \(\langle 3125,3\rangle\)  & \(\Phi_6(1^5)\)        & 6 & \((000000)\)  & constant           \\
\hline
\end{tabular}
\end{center}
\end{table}

In Table \ref{tab:TrfKerStemIcl6},
TKTs \(\varkappa\) of \(3\)-groups in the notation of
\cite[\S\ 3.3]{Ma2}
were determined by Nebelung
\cite[p. 208, Thm. 6.14]{Ne1}
already, using different presentations in equation
(\ref{eqn:PwrCmtPres}),
section \S\
\ref{sss:PrmPres2}.
For \(5\)-groups the TKTs are given here for the first time.
The only exception is the group \(\langle 3125,14\rangle\),
which was discussed in the well-known paper by Taussky
\cite[p. 436, Thm.2]{Ta2}
as an example to show that the coarse TKT \(\kappa=(\mathrm{AAAAAA})\) can occur for \(p=5\),
and also for primes \(p\ge 7\).
A convenient partial characterization is provided by
counters of fixed point transfer kernels, resp. abelianizations of type \((5,5,5)\),
\(\eta=\#\lbrace 1\le i\le 6\mid\kappa(i)=\mathrm{A}\rbrace
=\#\lbrace 1\le i\le 6\mid\tau(i)=(5,5,5)\rbrace\),
which must coincide, according to
\cite[Thm. 7, p. 11]{HeSm}.

\noindent
The correspondence between \(p=3\) and \(p=5\) is due to the formally identical power-commutator presentation.
However, it is partially rather shallow, since corresponding groups can have different properties
with respect to their role on the coclass graphs \(\mathcal{G}(3,2)\) and \(\mathcal{G}(5,2)\).
For example, the \(3\)-groups \(\langle 243,6\rangle\) and \(\langle 243,8\rangle\) are mainline vertices
having the mandatory total first transfer kernel \(\varkappa(1)=0\)
whereas the \(5\)-groups \(\langle 3125,12\rangle\) and \(\langle 3125,8\rangle\) are terminal without total transfer.



\subsubsection{Top vertices of type \((5,5)\) on \(\mathcal{G}(5,2)\)}
\label{sss:Typ55Cocl2}

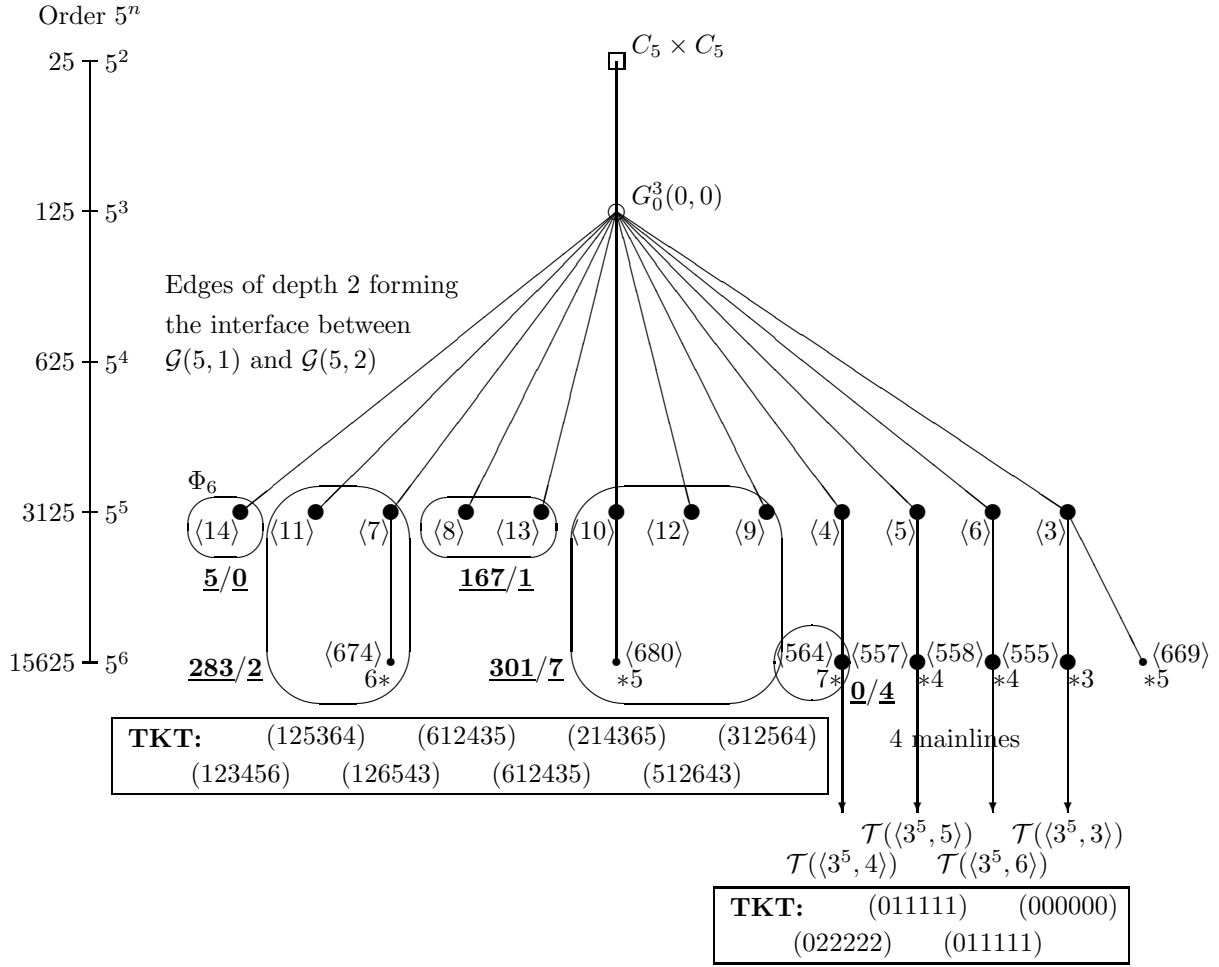
\begin{figure}[ht]
\caption{Sporadic groups and roots of conjectural coclass trees on \(\mathcal{G}(5,2)\)}
\label{fig:Typ55Cocl2}

\setlength{\unitlength}{1cm}
\begin{picture}(16,12.5)(-1,-9.5)

\put(0,2.5){\makebox(0,0)[cb]{Order \(5^n\)}}
\put(0,2){\line(0,-1){8}}
\multiput(-0.1,2)(0,-2){5}{\line(1,0){0.2}}
\put(-0.2,2){\makebox(0,0)[rc]{\(25\)}}
\put(0.2,2){\makebox(0,0)[lc]{\(5^2\)}}
\put(-0.2,0){\makebox(0,0)[rc]{\(125\)}}
\put(0.2,0){\makebox(0,0)[lc]{\(5^3\)}}
\put(-0.2,-2){\makebox(0,0)[rc]{\(625\)}}
\put(0.2,-2){\makebox(0,0)[lc]{\(5^4\)}}
\put(-0.2,-4){\makebox(0,0)[rc]{\(3125\)}}
\put(0.2,-4){\makebox(0,0)[lc]{\(5^5\)}}
\put(-0.2,-6){\makebox(0,0)[rc]{\(15625\)}}
\put(0.2,-6){\makebox(0,0)[lc]{\(5^6\)}}

\put(7.2,2.2){\makebox(0,0)[lc]{\(C_5\times C_5\)}}
\put(6.9,1.9){\framebox(0.2,0.2){}}
\put(7,2){\line(0,-1){2}}
\put(7,0){\circle{0.2}}
\put(7.2,0.2){\makebox(0,0)[lc]{\(G^3_0(0,0)\)}}

\put(7,0){\line(-5,-4){5}}
\put(7,0){\line(-1,-1){4}}
\put(7,0){\line(-3,-4){3}}
\put(7,0){\line(-1,-2){2}}
\put(7,0){\line(-1,-4){1}}
\put(7,0){\line(0,-1){4}}
\put(7,0){\line(1,-4){1}}
\put(7,0){\line(1,-2){2}}
\put(7,0){\line(3,-4){3}}
\put(7,0){\line(1,-1){4}}
\put(7,0){\line(5,-4){5}}
\put(7,0){\line(3,-2){6}}
\put(1,-1){\makebox(0,0)[lc]{Edges of depth \(2\) forming}}
\put(1,-1.5){\makebox(0,0)[lc]{the interface between}}
\put(1,-2){\makebox(0,0)[lc]{\(\mathcal{G}(5,1)\) and \(\mathcal{G}(5,2)\)}}

\put(1.5,-3.6){\makebox(0,0)[cc]{\(\Phi_6\)}}
\multiput(2,-4)(1,0){12}{\circle*{0.2}}
\put(2,-4.1){\makebox(0,0)[rt]{\(\langle 14\rangle\)}}
\put(3,-4.1){\makebox(0,0)[rt]{\(\langle 11\rangle\)}}
\put(4,-4.1){\makebox(0,0)[rt]{\(\langle 7\rangle\)}}
\put(5,-4.1){\makebox(0,0)[rt]{\(\langle 8\rangle\)}}
\put(6,-4.1){\makebox(0,0)[rt]{\(\langle 13\rangle\)}}
\put(7,-4.1){\makebox(0,0)[rt]{\(\langle 10\rangle\)}}
\put(8,-4.1){\makebox(0,0)[rt]{\(\langle 12\rangle\)}}
\put(9,-4.1){\makebox(0,0)[rt]{\(\langle 9\rangle\)}}
\put(10,-4.1){\makebox(0,0)[rt]{\(\langle 4\rangle\)}}
\put(11,-4.1){\makebox(0,0)[rt]{\(\langle 5\rangle\)}}
\put(12,-4.1){\makebox(0,0)[rt]{\(\langle 6\rangle\)}}
\put(13,-4.1){\makebox(0,0)[rt]{\(\langle 3\rangle\)}}

\multiput(4,-4)(3,0){2}{\line(0,-1){2}}
\multiput(10,-4)(1,0){4}{\line(0,-1){2}}
\multiput(4,-6)(3,0){2}{\circle*{0.1}}
\multiput(10,-6)(1,0){4}{\circle*{0.2}}
\put(13,-4){\line(1,-2){1}}
\put(14,-6){\circle*{0.1}}
\put(4,-6.1){\makebox(0,0)[rt]{\(6*\)}}
\put(3.9,-5.9){\makebox(0,0)[rc]{\(\langle 674\rangle\)}}
\put(7,-6.1){\makebox(0,0)[lt]{\(*5\)}}
\put(7.1,-5.9){\makebox(0,0)[lc]{\(\langle 680\rangle\)}}
\put(10,-6.1){\makebox(0,0)[rt]{\(7*\)}}
\put(11,-6.1){\makebox(0,0)[lt]{\(*4\)}}
\put(12,-6.1){\makebox(0,0)[lt]{\(*4\)}}
\put(13,-6.1){\makebox(0,0)[lt]{\(*3\)}}
\put(14,-6.1){\makebox(0,0)[lt]{\(*5\)}}


\put(9.9,-5.9){\makebox(0,0)[rc]{\(\langle 564\rangle\)}}
\put(10,-6){\vector(0,-1){2}}
\put(10,-8.5){\makebox(0,0)[ct]{\(\mathcal{T}(\langle 3^5,4\rangle)\)}}

\put(10.9,-5.9){\makebox(0,0)[rc]{\(\langle 557\rangle\)}}
\put(11,-6){\vector(0,-1){2}}
\put(11,-8.1){\makebox(0,0)[ct]{\(\mathcal{T}(\langle 3^5,5\rangle)\)}}

\put(11.5,-7){\makebox(0,0)[cc]{\(4\) mainlines}}

\put(11.9,-5.9){\makebox(0,0)[rc]{\(\langle 558\rangle\)}}
\put(12,-6){\vector(0,-1){2}}
\put(12,-8.5){\makebox(0,0)[ct]{\(\mathcal{T}(\langle 3^5,6\rangle)\)}}

\put(12.9,-5.9){\makebox(0,0)[rc]{\(\langle 555\rangle\)}}
\put(14.1,-5.9){\makebox(0,0)[lc]{\(\langle 669\rangle\)}}
\put(13,-6){\vector(0,-1){2}}
\put(13,-8.1){\makebox(0,0)[ct]{\(\mathcal{T}(\langle 3^5,3\rangle)\)}}

\put(1,-7){\makebox(0,0)[cc]{\textbf{TKT:}}}
\put(2,-7.5){\makebox(0,0)[cc]{\((123456)\)}}
\put(3,-7){\makebox(0,0)[cc]{\((125364)\)}}
\put(4,-7.5){\makebox(0,0)[cc]{\((126543)\)}}
\put(5,-7){\makebox(0,0)[cc]{\((612435)\)}}
\put(6,-7.5){\makebox(0,0)[cc]{\((612435)\)}}
\put(7,-7){\makebox(0,0)[cc]{\((214365)\)}}
\put(8,-7.5){\makebox(0,0)[cc]{\((512643)\)}}
\put(9,-7){\makebox(0,0)[cc]{\((312564)\)}}
\put(0.3,-7.75){\framebox(9.5,1){}}

\put(9,-9.25){\makebox(0,0)[cc]{\textbf{TKT:}}}
\put(10,-9.75){\makebox(0,0)[cc]{\((022222)\)}}
\put(11,-9.25){\makebox(0,0)[cc]{\((011111)\)}}
\put(12,-9.75){\makebox(0,0)[cc]{\((011111)\)}}
\put(13,-9.25){\makebox(0,0)[cc]{\((000000)\)}}
\put(8.3,-10){\framebox(5.5,1){}}

\put(1.8,-4.2){\oval(1,0.8)}
\put(1.8,-4.9){\makebox(0,0)[cc]{\(\underbar{\textbf{5}}/\underbar{\textbf{0}}\)}}
\put(3.3,-5.1){\oval(1.9,2.9)}
\put(1.8,-6.1){\makebox(0,0)[cc]{\(\underbar{\textbf{283}}/\underbar{\textbf{2}}\)}}

\put(5.3,-4.2){\oval(1.8,0.8)}
\put(5.4,-4.9){\makebox(0,0)[cc]{\(\underbar{\textbf{167}}/\underbar{\textbf{1}}\)}}
\put(7.8,-5.1){\oval(2.8,2.9)}
\put(5.8,-6.1){\makebox(0,0)[cc]{\(\underbar{\textbf{301}}/\underbar{\textbf{7}}\)}}

\put(9.6,-6){\oval(1,1)}
\put(10.4,-6.4){\makebox(0,0)[cc]{\(\underbar{\textbf{0}}/\underbar{\textbf{4}}\)}}

\end{picture}

\end{figure}

Figure
\ref{fig:Typ55Cocl2}
shows the non-CF groups at the top of coclass graph \(\mathcal{G}(5,2)\).
It was constructed by means of the SmallGroups library
\cite{BEO}
of GAP
\cite{GAP}
and MAGMA
\cite{MAGMA}.
The groups are labelled by a number in angles,
which is their identifier in that library.
Additional confirmation was obtained
by explicit descendant calculation
with the aid of the ANUPQ package
\cite{GNO}.

\noindent
The vertices of the coclass graph \(\mathcal{G}(5,2)\) in Figure
\ref{fig:Typ55Cocl2}
are classified by using different symbols:

\begin{enumerate}
\item
a large contour square \(\square\) represents an abelian group,
\item
a big contour circle {\Large \(\circ\)} represents a metabelian group with abelian maximal subgroup,
\item
big full discs {\Large \(\bullet\)} represent metabelian groups with bicyclic centre of type \((5,5)\),
\item
small full discs {\tiny \(\bullet\)} represent metabelian groups with cyclic centre of order \(5\).
\end{enumerate}

\noindent
The actual distribution of the
\(959\), resp. \(377\), second \(5\)-class groups \(G_5^2(K)\)
of complex, resp. real, quadratic number fields \(K=\mathbb{Q}(\sqrt{D})\) of type \((5,5)\)
with discriminant \(-2\,270\,831<D<26\,695\,193\) is represented by
underlined boldface counters (in the format complex/real)
of the hits of vertices surrounded by the adjacent oval.



\subsubsection{Fixed point principalization problem}
\label{sss:TausskyProblem}

We are pleased to present the solution of a problem posed in 1970 by Taussky
\cite[p. 438, Rem. 1]{Ta2}.
It concerns the lack of realizations, in the form of
second \(5\)-class groups \(\mathrm{G}_5^2(K)\) of number fields \(K\),
of the unique metabelian \(5\)-group \(\langle 5^5,14\rangle\)
with \(6\) fixed point transfer kernels,
that is with coarse TKT \(\kappa=(\mathrm{AAAAAA})\),
but without total transfer kernels \(\varkappa(i)=0\).
Actually, we now have \(5\) realizations
of this very special TKT \(\varkappa=(123456)\) (identity permutation)
for quadratic fields \(K=\mathbb{Q}(\sqrt{D})\),
\(D\in\lbrace -89\,751,-235\,796,-1\,006\,931,-1\,996\,091,\)
\(-2\,187\,064\rbrace\),
in Table
\ref{tbl:CompQuad5x5Details},
and \(4\) further realizations
in Table
\ref{tbl:CyclQrt5x5},
for certain cyclic quartic fields
\(K=\mathbb{Q}\left((\zeta_5-\zeta_5^{-1})\sqrt{D}\right)\),
\(D\in\lbrace 581,753,2\,296,4\,553\rbrace\).



\subsubsection{Statistical evaluation of second \(5\)-class groups \(\mathrm{G}_5^2(K)\)}
\label{sss:StatScnd5ClgpCocl2}

The possibilities for \(5\)-groups of coclass \(2\)
are more extensive than those for coclass \(1\).

For the \(377\) real quadratic fields \(K=\mathbb{Q}(\sqrt{D})\), \(0<D<26\,695\,193\),
in Table
\ref{tbl:RealQuad5x5},
there occur \(7\) cases of coarse TKT \(\kappa=(\mathrm{BBBBBB})\), for
\(D\in\lbrace 4\,954\,652,7\,216\,401,12\,562\,849,16\,434\,245,18\,434\,456,\)
\(19\,115\,293,20\,473\,841\rbrace\),
a single case of TKT \(\varkappa=(612435)\), for \(D=18\,070\,649\),
\(2\) cases of coarse TKT\(\kappa=(\mathrm{AABBBB})\), for
\(D\in\lbrace 10\,486\,805,18\,834\,493\rbrace\),
and \(4\) cases of the first excited state of TKT \(\varkappa=(022222)\),
for \(D\in\lbrace 7\,306\,081,11\,545\,953,14\,963\,612,22\,042\,632\rbrace\).



\renewcommand{\arraystretch}{1.1}

\begin{table}[ht]
\caption{\(9\) variants of \(G=\mathrm{G}_5^2(K)\) for \(959\) \(K=\mathbb{Q}(\sqrt{D})\), \(-2\,270\,831\le D<0\)}
\label{tbl:CompQuad5x5}
\begin{center}
\begin{tabular}{|r||c|c||c||c|c||c|c|}
\hline
           \(D\) & \(\tau(K)\)                       & \(\tau(0)\) & \(\kappa(K)\)    & \(G\)              & \(\mathrm{cc}(G)\) &  \(\#\) &   \(\%\) \\
\hline
    \(-11\,199\) & \((5,5^2)^6\)                     & \((5,5,5)\) & \((B^6)\)        &\(\langle 3125,9\vert 10\vert 12\rangle\)& \(2\) & \(301\) & \(31.4\) \\
    \(-12\,451\) & \((5,5,5),(5,5^2)^5\)             & \((5,5,5)\) & \((A,B^5)\)      &\(\langle 3125,8\vert 13\rangle*\)& \(2\) & \(167\) & \(17.4\) \\
    \(-30\,263\) & \((5,5,5)^2,(5,5^2)^4\)           & \((5,5,5)\) & \((A^2,B^4)\)    &\(\langle 3125,7\vert 11\rangle\)& \(2\) & \(283\) & \(29.5\) \\
    \(-89\,751\) & \((5,5,5)^6\)                     & \((5,5,5)\) & \((A^6)\)        & \(\langle 3125,14\rangle*\)      & \(2\) &   \(5\) &   \(  \) \\
\hline
    \(-62\,632\) & \((5,5,5,5^2),(5,5,5),(5,5^2)^4\) &             & \((?,A,B^4)\)    & \(\langle 78125,\#\rangle\)     & \(2\) & \(124\) & \(12.9\) \\
    \(-67\,031\) & \((5,5,5,5,5),(5,5^2)^5\)         &             & \((B^6)\)        & \(\langle 78125,\#\rangle\)     & \(2\) &   \(6\) &   \(  \) \\
    \(-67\,063\) & \((5,5,5,5,5),(5,5,5),(5,5^2)^4\) &             & \((B,A,B^4)\)    & \(\langle 78125,\#\rangle\)     & \(2\) &  \(37\) &  \(3.9\) \\
   \(-280\,847\) & \((5,5,5,5^2),(5,5^2)^5\)         &             & \((?,B^5)\)      & \(\langle 78125,\#\rangle\)     & \(2\) &  \(32\) &  \(3.3\) \\
\hline
   \(-181\,752\) & \((5,5^2,5^2,5^2),(5,5,5),(5,5^2)^4\) &         & \((?,A,B^4)\)    & \(\langle 1953125,\#\rangle\)   & \(2\) &   \(4\) &   \(  \) \\
\hline
\end{tabular}
\end{center}
\end{table}

Among the \(959\) complex quadratic fields \(K=\mathbb{Q}(\sqrt{D})\), \(-2\,270\,831\le D<0\),
in Table
\ref{tbl:CompQuad5x5},
ground states (GS) appear exclusively with sporadic, and mostly terminal, top vertices of \(\mathcal{G}(5,2)\).
The \(5\) cases of TKT \(\varkappa=(123456)\) have been presented separately
in section \S\
\ref{sss:TausskyProblem}
as solutions of Taussky's problem of 1970.
Further, there are \(167\) cases \((17.4\%)\) of TKT \(\varkappa=(612435)\)
(5-cycle with coarse TKT \(\kappa=(\mathrm{BBBABB})\)) starting with \(D=-12\,451\),
which was attempted but not analyzed completely in 1982 by Heider and Schmithals
\cite{HeSm}
and \(283\) cases \((29.5\%)\) of coarse TKT \(\kappa=(\mathrm{AABBBB})\) starting with \(D=-30\,263\).
The remaining \(301\) cases \((31.4\%)\) of coarse TKT \(\kappa=(\mathrm{BBBBBB})\), starting with \(D=-11\,199\) are slightly dominating.

For excited states (ES) of coclass \(2\) as well as of coclass \(1\),
the distinguished first \(5\)-class group \(\mathrm{Cl}_5(K_1)\)
of the non-Galois absolute quintic subfield \(K_1\) of the unramified extension \(L_1\vert K\)
is of \(5\)-rank \(\mathrm{r}_5(K_1)=2\),
which shows impressively that the rank equation for \(p=3\),
\(\mathrm{r}_3(K_i)=\mathrm{r}_3(K)-1\),
by Gras
\cite{Gr}
and Gerth
\cite{Ge}
generalizes to a double inequality for \(p\ge 5\),
\[\mathrm{r}_p(K)-1\le\mathrm{r}_p(K_i)\le\frac{p-1}{2}\cdot(\mathrm{r}_p(K)-1),\]
as predicted, and partially proved, by B\"olling
\cite{Boe}
and Lemmermeyer
\cite{Lm}.

\renewcommand{\arraystretch}{1.1}

\begin{table}[ht]
\caption{\(9\) subvariants of \(G=\mathrm{G}_5^2(K)\) for \(31\) fields \(K=\mathbb{Q}(\sqrt{D})\), \(-89\,751\le D<0\)}
\label{tbl:CompQuad5x5Details}
\begin{center}
\begin{tabular}{|r||c|c||c||c|c||c|}
\hline
           \(D\) & \(\tau(K)\)                       & \(\tau(0)\) & \(\varkappa(K)\) & \(G\)              & \(\mathrm{cc}(G)\) & \(\#\) \\
\hline
    \(-11\,199\) & \((5,5^2)^6\)                     & \((5,5,5)\) & \((512643)\)     & \(\langle 3125,12\rangle*\)     & \(2\) &  \(7\) \\
    \(-17\,944\) & \((5,5^2)^6\)                     & \((5,5,5)\) & \((312564)\)     & \(\langle 3125,9\rangle*\)      & \(2\) &  \(2\) \\
    \(-42\,871\) & \((5,5^2)^6\)                     & \((5,5,5)\) & \((214365)\)     & \(\langle 3125,10\rangle\)      & \(2\) &  \(3\) \\
                 &                                   &             &                  & or \(\langle 15625,680\rangle\) &       &        \\
\hline
    \(-12\,451\) & \((5,5,5),(5,5^2)^5\)             & \((5,5,5)\) & \((612435)\)     & \(\langle 3125,8\rangle*\)      & \(2\) &  \(5\) \\
                 &                                   &             &                  & or \(\langle 3125,13\rangle*\)  &       &        \\
\hline
    \(-30\,263\) & \((5,5,5)^2,(5,5^2)^4\)           & \((5,5,5)\) & \((126543)\)     & \(\langle 3125,7\rangle\)       & \(2\) &  \(4\) \\
                 &                                   &             &                  & or \(\langle 15625,647\rangle\) &       &        \\
    \(-37\,363\) & \((5,5,5)^2,(5,5^2)^4\)           & \((5,5,5)\) & \((125364)\)     & \(\langle 3125,11\rangle*\)     & \(2\) &  \(2\) \\
\hline
    \(-89\,751\) & \((5,5,5)^6\)                     & \((5,5,5)\) & \((123456)\)     & \(\langle 3125,14\rangle*\)     & \(2\) &  \(5\) \\
\hline
    \(-62\,632\) & \((5,5,5,5^2),(5,5,5),(5,5^2)^4\) &             & \((322222)\)     & \(\langle 78125,\#\rangle\)     & \(2\) &  \(1\) \\
    \(-67\,031\) & \((5,5,5,5,5),(5,5^2)^5\)         &             & \((211111)\)     & \(\langle 78125,\#\rangle\)     & \(2\) &  \(1\) \\
    \(-67\,063\) & \((5,5,5,5,5),(5,5,5),(5,5^2)^4\) &             & \((322222)\)     & \(\langle 78125,\#\rangle\)     & \(2\) &  \(2\) \\
\hline
\end{tabular}
\end{center}
\end{table}

\noindent
The transfer target type (TTT) \(\tau(G)\)
of second \(5\)-class groups \(\mathrm{G}_5^2(K)\)
has been computed
for all quadratic number fields \(K=\mathbb{Q}(\sqrt{D})\), having
discriminant \(-2\,270\,831<D<26\,695\,193\) and
\(5\)-class group \(\mathrm{Cl}_5(K)\) of type \((5,5)\),
with the aid of MAGMA
\cite{MAGMA}
As a refinement, we calculated
the transfer kernel type (TKT) \(\varkappa(G)\)
for \(31\) fields \(K=\mathbb{Q}(\sqrt{D})\), \(-89\,751\le D<0\),
as given in Table
\ref{tbl:CompQuad5x5Details}.
This also refines results of Bembom in
\cite[p. 129]{Bm}.
Observe that Bembom does not give TKTs in our sense
and consequently was not able to discover the
distinguished role of \(D=-89\,751\) with respect
to the Taussky problem. 



\subsubsection{Statistical evaluation of second \(7\)-class groups \(\mathrm{G}_7^2(K)\)}
\label{sss:StatScnd7Clgp}

Among the \(94\) complex quadratic fields \(K\) of type \((7,7)\)
with discriminants \(-10^6\le D<0\),
we found \(7\) variants of the second \(7\)-class group \(G=\mathrm{G}_7^2(K)\),
characterized by different TTT \(\tau(K)\) and Taussky's coarse TKT \(\kappa(K)\),
which are related by
\cite[Thm. 7, p. 11]{HeSm}.

\renewcommand{\arraystretch}{1.1}

\begin{table}[ht]
\caption{\(7\) variants of \(G=\mathrm{G}_7^2(K)\) for \(70\) fields \(K=\mathbb{Q}(\sqrt{D})\), \(-751\,288\le D<0\)}
\label{tbl:CompQuad7x7}
\begin{center}
\begin{tabular}{|r||c|c||c||c|c||c|c|}
\hline
           \(D\) & \(\tau(K)\)                       & \(\tau(0)\) & \(\kappa(K)\) & \(G\)         & \(\mathrm{cc}(G)\) & \(\#\) & \(\%\) \\
\hline
    \(-63\,499\) & \((7,7^2)^8\)                     & \((7,7,7)\) & \((B^8)\)     & \(\langle 16807,10\vert 14\vert 15\vert 16\rangle\)& \(2\) & \(40\) & \(43\) \\
   \(-183\,619\) & \((7,7,7)^2,(7,7^2)^6\)           & \((7,7,7)\) & \((A^2,B^6)\) & \(\langle 16807,11\vert 12\vert 13\rangle\)& \(2\) & \(29\) & \(31\) \\
   \(-227\,860\) & \((7,7,7),(7,7^2)^7\)             & \((7,7,7)\) & \((A,B^7)\)   & \(\langle 16807,8\vert 9\rangle\)& \(2\) &  \(9\) & \(9.6\) \\
         unknown & \((7,7,7)^8\)                     & \((7,7,7)\) & \((A^8)\)     & \(\langle 16807,7\rangle\)& \(2\) &  \( \) & \(  \) \\
\hline
   \(-159\,592\) & \((7,7,7,7,7),(7,7,7),(7,7^2)^6\) &             & \((A^2,B^6)\) &\(\langle 823543,\#\rangle\)& \(2\) &  \(3\) & \(  \) \\
   \(-227\,387\) & \((7,7,7,7^2),(7,7,7),(7,7^2)^6\) &             & \((B,A,B^6)\) &\(\langle 823543,\#\rangle\)& \(2\) &  \(10\) & \(  \) \\
   \(-272\,179\) & \((7,7,7,7^2),(7,7^2)^7\)         &             & \((B^8)\)     &\(\langle 823543,\#\rangle\)& \(2\) &  \(2\) & \(  \) \\
\hline
   \(-673\,611\) & \((7,7,7,7,7,7^2),(7,7,7),(7,7^2)^6\) &         & \((?,A,B^6)\) &\(\langle 40353607,\#\rangle\)& \(2\) &  \(1\) & \(  \) \\
\hline
\end{tabular}
\end{center}
\end{table}

\noindent
In Table
\ref{tbl:CompQuad7x7}
we present the discriminants with smallest absolute values,
corresponding to these variants.
\(\tau(0)\) denotes the \(7\)-class group of \(\mathrm{F}_7^1(K)\).
Using the SmallGroups library
\cite{BEO},
we identified \(78\) \textit{ground states} \((83\%)\)
having their \(G\) among the sporadic top vertices of \(\mathcal{G}(7,2)\)
in the stem of isoclinism family \(\Phi_6\).
Unfortunately, there didn't occur a solution
of Taussky's 1970 fixed point capitulation problem for \(p=7\),
in form of a realization of \(\langle 16807,7\rangle\).
However, there appeared \(15\) \textit{first excited states} \((16\%)\)
with \(G\) located on coclass trees of \(\mathcal{G}(7,2)\),
where the non-Galois subfield \(K_1\) of the distinguished extension \(L_1\)
has a \(7\)-class group of type \((7,7)\),
and, particularly remarkable,
a single \textit{second excited state} for \(D=-673\,611\),
where the maximal \(7\)-rank \(3\) in B\"olling's inequality
\cite{Boe,Lm}
is attained in form of \(\mathrm{Cl}_7(K_1)\simeq(7,7,7)\).



\subsection{Cyclic quartic fields of type \((5,5)\)}
\label{ss:NewRslt5Mirror}

In cooperation with A. Azizi and M. Talbi
\cite{ATM},
and based on the quintic reflection theorem
\cite{Ks},
we have computed
the isomorphism type of the second \(5\)-class group
\(G=\mathrm{G}_5^2(K)\)
of \(41\) cyclic quartic fields
\(K=\mathbb{Q}\left((\zeta_5-\zeta_5^{-1})\sqrt{D}\right)\),
\(\zeta_5=\exp(\frac{2\pi i}{5})\), \(-15\,419\le D<5\,000\), \(5\nmid D\),
of type \((5,5)\).
Such a field is the \(5\)-dual \lq mirror image\rq\
of the quadratic fields \(\mathbb{Q}(\sqrt{D})\) and \(\mathbb{Q}(\sqrt{5D})\).
Isomorphisms among the extensions \(L_i\vert K\), \(1\le i\le 6\),
cause severe constraints on the group \(G\), as Table
\ref{tbl:CyclQrt5x5},
visualized by Figures
\ref{fig:Distr5Cocl1}
and
\ref{fig:Typ55Cocl2},
shows.

\renewcommand{\arraystretch}{1.1}

\begin{table}[ht]
\caption{\(7\) variants of \(G=\mathrm{G}_5^2(K)\) for \(41\) fields \(K=\mathbb{Q}\left((\zeta_5-\zeta_5^{-1})\sqrt{D}\right)\)}
\label{tbl:CyclQrt5x5}
\begin{center}
\begin{tabular}{|r||c|c||c||c|c||c|c|}
\hline
           \(D\) & \(\tau(K)\)                       & \(\tau(0)\) & \(\varkappa(K)\) & \(G\)              & \(\mathrm{cc}(G)\) & \(\#\) &   \(\%\) \\
\hline
    \(-12\,883\) & \((5,5)^6\)                       &             & \((000000)\)     & \(\langle 125,3\rangle\)        & \(1\) &  \(4\) &\(100\%\) \\
\hline
\hline
         \(257\) & \((5,5,5)^2,(5,5^2)^4\)           & \((5,5,5)\) & \((022222)\)     & \(\langle 3125,4\rangle\)       & \(2\) &  \(9\) & \(24\%\) \\
         \(457\) & \((5,5,5)^2,(5,5^2)^4\)           & \((5,5,5)\) & \((125364)\)     & \(\langle 3125,11\rangle*\)     & \(2\) &  \( \) &   \(  \) \\
         \(508\) & \((5,5,5)^2,(5,5^2)^4\)           & \((5,5,5)\) & \((126543)\)     & \(\langle 3125,7\rangle\)       & \(2\) &  \( \) &   \(  \) \\
                 &                                   &             &                  & or \(\langle 15625,647\rangle\) &       &        &          \\
\hline
         \(581\) & \((5,5,5)^6\)                     & \((5,5,5)\) & \((123456)\)     & \(\langle 3125,14\rangle*\)     & \(2\) &  \(4\) & \(11\%\) \\
\hline
      \(1\,137\) & \((5,5,5,5^2),(5,5,5),(5,5^2)^4\) &             & \((111111)\)     & \(\langle 78125,\#\rangle\)     & \(2\) &  \(3\) &  \(8\%\) \\
\hline
      \(4\,357\) & \((5)^6\)                         &             & \((000000)\)     & \(\langle 25,2\rangle\)     & \(1\) &  \(3\) &  \(8\%\) \\
\hline
\end{tabular}
\end{center}
\end{table}



\section{\(p\)-Groups with double layered metabelianization of type \((p^2,p)\) or \((p,p,p)\)}
\label{s:DoubleLayer}

For a number field \(K\) with \(p\)-class group \(\mathrm{Cl}_p(K)\)
of type \((p^2,p)\), resp. \((p,p,p)\),
there exist \textit{two layers} of unramified abelian extensions \(L\vert K\),
each containing \(p+1\), resp. \(p^2+p+1\), members.
Extensions in the \textit{first layer} are of relative degree \(p\),
those in the \textit{second layer} are of relative degree \(p^2\).
Consequently,
the second layer tends to be out of the scope of actual computations.
However, there are some exceptions of modest degree.

\subsection{Quadratic fields of type \((9,3)\)}
\label{ss:Qdr9x3}

On the one hand, there is the case \(p=3\)
for quadratic fields \(K=\mathbb{Q}(\sqrt{D})\) with
\(3\)-class group \(\mathrm{Cl}_3(K)\)
of type \((9,3)\) or \((3,3,3)\),
where extensions in the second layer are of absolute degree \(18\).
From the viewpoint of \(3\)-towers,
there are no open problems for
complex quadratic fields \(K\) of type \((3,3,3)\),
since it is known that \(\ell_3(K)=\infty\)
\cite{KoVe,mL},
that is, \(\mathrm{G}_3^\infty(K)\) is always an infinite pro-\(3\) group.

Thus, we focussed on \(3\)-class rank \(2\) and computed
the first layer of the TTT and TKT of all
\(875\) complex quadratic fields \(K\) of type \((9,3)\) with discriminant \(-10^6<D<0\) 
and of all
\(271\) real quadratic fields \(K\) of type \((9,3)\) with discriminant \(0<D<10^7\).
In
\cite{Ma4},
we will show that this information is sufficient
to identify the second \(3\)-class group \(\mathrm{G}_3^2(K)\)
for \(565\) negative discriminants \((65\%)\)
and for \(188\) positive discriminants \((70\%)\).
For the remainder, the second layer of the TTT and TKT is required.

\subsection{Quadratic and biquadratic fields of type \((2,2,2)\)}
\label{ss:BiQdr2x2x2}

On the other hand, we have the case \(p=2\)
for quadratic, resp. quartic, fields \(K\) with
\(2\)-class group \(\mathrm{Cl}_2(K)\)
of type \((4,2)\) or \((2,2,2)\),
where extensions in the second layer are of absolute degree \(8\), resp. \(16\).

We were particularly interested in fields \(K\) of \(2\)-class rank \(3\),
where the \(2\)-tower length \(\ell_2(K)\) is still an open problem.
We found that the coclass tree \(\mathcal{T}(\langle 16,11\rangle)\),
which is the unique tree of coclass graph \(\mathcal{G}(2,2)\)
containing groups with abelianization of type \((2,2,2)\),
is populated by second \(2\)-class groups \(\mathrm{G}_2^2(K)\)
of real quadratic fields \(K\) of type \((2,2,2)\).
The tree \(\mathcal{T}(\langle 16,11\rangle)\) corresponds to
the pro-\(2\) group \(S_5\) in
\cite{Fs,EkFs}
with periodic sequences \(K_x^{i}\), \(46\le i\le 51\),
given by explicit parametrized presentations for \(x\ge 0\) in
\cite{Fs}.
It also corresponds to the so-called family \(\#59\)
with explicit pro-\(2\) presentation given in
\cite{NmOb}.

Further, we obtained deeper results concerning
second \(2\)-class groups \(G=\mathrm{G}_2^2(K)\)
of complex quadratic fields of type \((2,2,2)\),
which have been classified in terms of the
smallest non-abelian lower central quotient \(G/\gamma_3(G)\),
usually coinciding with the root of the coclass tree \(\mathcal{T}\) such that \(G\in\mathcal{T}\),
by E. Benjamin, F. Lemmermeyer, and C. Snyder
\cite{Lm1,BLS}.
Note that these authors use the Hall-Senior classification
\cite{HaSn},
whereas we give identifiers of the SmallGroups library
\cite{BEO}.
The groups \(G\) are mainly, but not exclusively, located
at the terminal top vertices \(\langle 32,32\rangle\) and \(\langle 32,33\rangle\)
of coclass graph \(\mathcal{G}(2,3)\) and on the coclass trees
\(\mathcal{T}(\langle 32,29\rangle)\), \(\mathcal{T}(\langle 32,30\rangle)\), \(\mathcal{T}(\langle 32,35\rangle)\),
corresponding to the families \(\#75\), \(\#76\), \(\#79\) with explicit pro-\(2\) presentations given in
\cite{NmOb}.
These subtrees of \(\mathcal{G}(2,3)\) seem to be populated on every branch,
with the only exception of the root.

We intend to include these results on complex quadratic fields in
\cite{AZTM},
where the principal aim is to investigate
bicyclic biquadratic fields \(K=\mathbb{Q}\left(\sqrt{\strut -1},\sqrt{\strut d}\right)\),
called \textit{special Dirichlet fields} by Hilbert
\cite{Hi},
with \(2\)-class groups of type \((2,2,2)\),
based on work by A. Azizi, A. Zekhnini, and M. Taous
\cite{AzTs,AZT1,AZT2}.
The second \(2\)-class groups \(\mathrm{G}_2^2(K)\)
for certain series of radicands \(d>0\),
for example \(d\in\lbrace 170,730,2314\rbrace\),
seem to be distributed on every branch
of the coclass tree \(\mathcal{T}(\langle 64,140\rangle)\) of coclass graph \(\mathcal{G}(2,3)\),
which corresponds to family \(\#73\) with explicit pro-\(2\) presentation given in
\cite{NmOb}.



\section{Acknowledgements}
\label{s:Thanks}

The author is indebted to
Nigel Boston, University of Wisconsin, Madison, and
Michael R. Bush, Washington and Lee University, Lexington,
for intriguing discussions about the length of \(p\)-towers
and Schur \(\sigma\)-groups in \S\
\ref{ss:TowerLength}.

Sincere thanks are given to 
Mike F. Newman, Australian National University, Canberra,
for valuable suggestions concerning use of
the SmallGroups library
\cite{BEO}
and ANUPQ package
\cite{GNO}
of GAP 4
\cite{GAP}
and MAGMA
\cite{MAGMA},
and for precious aid in identifying finite metabelian \(p\)-groups,
produced by various approaches to the classification problem,
in particular, by Blackburn
\cite{Bl1},
James
\cite{Jm}, 
Ascione
\cite{As1},
and Nebelung
\cite{Ne1}.

We thank
Abdelmalek Azizi and Mohammed Talbi, Facult\'e des Sciences, Oujda,
and A\"issa Derhem, Casablanca,
for our joint investigation of
bicyclic biquadratic fields containing third roots of unity in \S\
\ref{ss:NewRsltESR}.

Further, we gratefully acknowledge
helpful advice for constructing class fields
\cite{Fi}
with the aid of MAGMA
\cite{MAGMA,BCP,BCFS}
by Claus Fieker, University of Kaiserslautern.

Concerning the coclass graph \(\mathcal{G}(5,1)\) in \S\
\ref{sss:5Cocl1},
we thank Heiko Dietrich, University of Trento,
for making available unpublished details of the tree structure.




\end{document}